\documentclass{article}

\input{hexboard.sty}
\input{hex-3term.sty}

\noshadows

\title{On 3-terminal positions in Hex}

\author{Eric Demer, UCLA\\
  Peter Selinger, Dalhousie University}

\date{}

\begin{document}

\maketitle

\begin{abstract}
  This paper is about 3-terminal regions in Hex. A 3-terminal region
  is a region of the Hex board that is completely surrounded by black
  and white stones, in such a way that the black boundary stones form
  3 connected components. We characterize Hex as the universal planar
  Shannon game of degree 3. This ensures that every Hex position can
  be decomposed into 3-terminal regions. We then investigate the
  combinatorial game theory of 3-terminal regions. We show that there
  are infinitely many distinct Hex-realizable values for such
  regions. We introduce an infinite family of 3-terminal positions
  called superswitches and investigate their properties. We also
  present a database of Hex-realizable 3-terminal values, and
  illustrate its utility as a problem-solving tool by giving various
  applications. The applications include the automated verification of
  connects-both templates and pivoting templates, a new handicap
  strategy for $11\times 11$ Hex, and a method for constructing
  witnesses for the non-inferiority of probes in many Hex templates.
  These methods allow us to disprove a conjecture by Henderson and
  Hayward.
\end{abstract}

\section{Introduction}

Hex is a perfect information game for two players. It was invented by
Piet Hein in 1942 {\cite{Hein}}, and is usually played on an $n\times
n$ rhombic grid of hexagonal cells like this:
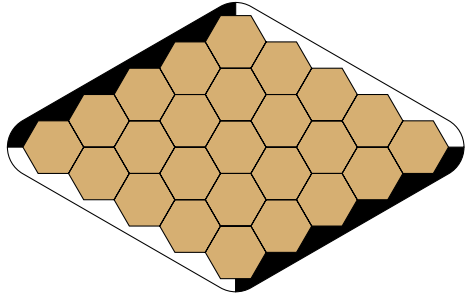
\begin{equation}
  \label{eqn:hex-5x5}
  \m{$
    \begin{hexboard}
      \board(5,5)
    \end{hexboard}
    $}
\end{equation}
Two opposing board edges are colored black, and the other two are
colored white. The players, called Black and White, alternately place
a stone of their color on an empty cell, with Black going first. The
winner is the player who connects their two board edges. It is an
interesting property of Hex that draws are not possible: there is
always exactly one winner {\cite{Hein}}. Another interesting property
is that the first player has a theoretical winning strategy on all
boards of size $n\times n$; however, the proof is non-constructive,
and no concrete winning strategy is known except for some very small
board sizes {\cite{Nash}}.

In principle, the game ends as soon as one player connects their
edges. In practice, most actual games end long before this happens,
with the losing player resigning. However, for theoretical purposes,
it is often simpler to assume that the game continues until the board
is completely filled; since the surplus moves cannot change the
winner, this assumption is without loss of generality. Also, since
Black has a theoretical winning strategy and a significant practical
advantage, in competitive play the \emph{swap rule} is used: right
after Black makes the first move, White has the option to switch
colors. The swap rule makes the game more fair, but since it only
affects the first two moves of the game, it is not relevant to most of
the results of this paper. Except for one application in
Section~\ref{ssec:handicap-strategy}, we do not consider the swap rule
here.

Combinatorial game theory is a general theory of two-player sequential
games with perfect information. It was introduced by Conway
{\cite{ONAG}} and Berlekamp, Conway, and Guy {\cite{WinningWays}}.
Combinatorial game theory was initially developed for \emph{normal
play} games, in which a player loses when they cannot make a move.
Because Hex is not a normal play game, it requires an adaptation of
combinatorial game theory; such an adaptation was given in
{\cite{S2022-hex-cgt}}.

This paper is about 3-terminal regions in Hex. A \emph{region} of the
Hex board is just a subset of its cells. A region that is completely
surrounded by black and white stones is called an \emph{$n$-terminal
region} when its boundary has $n$ black components and $n$ white
components. We call these boundary components the black and white
\emph{terminals}. Figure~\ref{fig:3terminal-example1} shows an example
of a 3-terminal region.
\begin{figure}
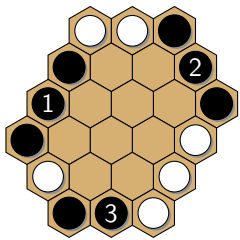

  \[
  \begin{hexboard}[scale=0.8]
    \shadows
    \rotation{-150}
    \foreach\i in {3,...,5} {\hex(1,\i)}
    \foreach\i in {2,...,5} {\hex(2,\i)}
    \foreach\i in {1,...,5} {\hex(3,\i)}
    \foreach\i in {1,...,5} {\hex(4,\i)}
    \foreach\i in {1,...,4} {\hex(5,\i)}
    \foreach\i in {1,...,3} {\hex(6,\i)}
    \black(3,1)
    \black(2,2)\sflabel{2}
    \black(1,3)
    \white(1,4)
    \white(1,5)
    \black(2,5)
    \black(3,5)\sflabel{1}
    \black(4,5)
    \white(5,4)
    \black(6,3)
    \black(6,2)\sflabel{3}
    \white(6,1)
    \white(5,1)
    \white(4,1)    
  \end{hexboard}
  \]
  \caption{A 3-terminal region}
  \label{fig:3terminal-example1}
\end{figure}

When Hex is played on the whole board, there are only two possible
outcomes at the end of a game: either Black wins or White
wins. However, when we consider play inside a 3-terminal region, the
situation is more complicated. There are five possible outcomes within
the region: either Black connects all three black terminals, or White
connects all three white terminals, or Black and White each connect
exactly two of their terminals, which can happen in three different
ways.

\subsection*{Outline of the paper}

In Section~\ref{sec:hex-universal}, we explain why 3-terminal regions
(and not, say, 4-terminal regions or 5-terminal regions) are
fundamental for Hex: we show that every Hex position can be decomposed
into 3-terminal regions. We do this by showing that Hex is equivalent
to a class of games called planar Shannon games of
degree~3. Section~\ref{sec:hex-universal} does not require
combinatorial game theory, but the rest of the paper does.

In Section~\ref{sec:background-cgt}, we provide background material on
some basic definitions and results from the combinatorial game theory
of Hex.

In Section~\ref{sec:superswitches}, we introduce some interesting
families of 3-terminal positions called \emph{superswitches},
\emph{simpleswitches}, and \emph{tripleswitches}. We use this to prove
that there are infinitely many non-equivalent 3-terminal Hex
positions. As an application, we show how these switches can be used,
in conjunction with a Hex solver, to verify a certain kind of Hex
template called a \emph{connects-both template}.

In Section~\ref{sec:database}, we present a database of Hex-realizable
3-terminal values. We explain how we generated the database and how it
can be used.

In Section~\ref{sec:pivoting}, we apply our theory of 3-terminal
positions to the verification of another kind of Hex template, called
a \emph{pivoting template}. We characterize both sente and gote
pivoting templates in terms of combinatorial game theory, and we use
the database from Section~\ref{sec:database} to find a particular
3-terminal context that can be used, in conjunction with a Hex solver,
to verify pivoting templates. As another application, we describe a
new handicap winning strategy for $11\times 11$ Hex.

Finally, in Section~\ref{sec:witnesses}, we use our theory to solve an
open problem. A move in a Hex region is called \emph{inferior} if it
can never be the unique winning move. Henderson and Hayward
{\cite{Henderson-Hayward}} conjectured that $5$ of the $8$ possible
intrusions into the region called the \emph{ziggurat} or \emph{4-3-2
edge template} are inferior. We disprove the conjecture by exhibiting
specific Hex positions in which each of the $8$ intrusions is the
unique winning move. We describe a general method for computing such
witnessing positions, and also settle the (non-)inferiority of
intrusions into many other Hex templates.
  
\section{Hex is the universal 3-planar Shannon game}
\label{sec:hex-universal}

In this section, we prove that Hex has a nice mathematical
characterization as the universal 3-planar Shannon game.  We start by
explaining what this means.

\subsection{Set coloring games and Shannon games}
\label{sec:set-coloring}

We first define a general class of games called \emph{set coloring
games}. When $X$ and $Y$ are sets, we write $Y^X$ for the set of all
functions from $X$ to $Y$.

\begin{definition}
  Let $\Bool=\s{\black,\white}$. A \emph{set coloring game} over
  $\Bool$ is given by a finite set $X$, whose elements are called
  \emph{cells}, and a function $\pi:\Bool^X\to\Bool$, called the
  \emph{payoff function}. It is played as follows: there is a game
  board, initially empty, whose cells are in one-to-one correspondence
  with the elements of $X$. The players, Black and White, take turns
  choosing an empty cell and placing a stone of their color on it. The
  game ends when the board is completely filled, at which point the
  coloring of the cells determines a function $f:X\to\Bool$. The
  winner is determined by the payoff function: Black wins if
  $\pi(f)=\black$, and White wins otherwise.
\end{definition}

We say that two set coloring games $(X,\pi)$ and $(X',\pi')$ are
\emph{isomorphic} if there is a bijection $\phi:X'\to X$ respecting
the payoff function, i.e., such that for all $f:X\to\Bool$, we have
$\pi(f) = \pi'(f\circ\phi)$. A \emph{position} in a set coloring game
is an assignment of black and/or white stones to some subset of the
cells. Each such position can itself be regarded as a set coloring
game on the remaining empty cells.

Hex is evidently a set coloring game. It will also be useful to
consider another class of set coloring games, the so-called
\emph{Shannon games}. The Shannon games we consider in this paper are
vertex Shannon games on hypergraphs.  Recall that a \emph{hypergraph}
is a pair $(V,E)$, where $V$ is a set whose elements are called
\emph{vertices}, and $E$ is a set of subsets of $V$, called
\emph{hyperedges}.  A hypergraph where each hyperedge contains exactly
two vertices is also called a \emph{graph}, and in this case, the
hyperedges are usually just called \emph{edges}. In pictures, we will
represent vertices as circles, edges as lines connecting two vertices,
and hyperedges as clusters of lines joining any number of vertices, as
in Figure~\ref{fig:hypergraph-picture}.
\begin{figure}
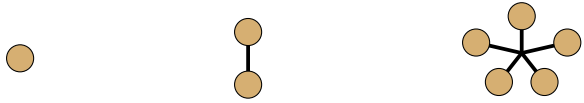

  \[
  \begin{hexboard}
    \vertex(0,0)
    \gedge(3,2)(2,3)
    \vertex(3,2)\vertex(2,3)
    \gedge(5.3,4.7)(5.6,5.4)
    \gedge(4.8,5.7)(5.6,5.4)
    \gedge(5.3,6.2)(5.6,5.4)
    \gedge(6.3,5.7)(5.6,5.4)
    \gedge(6.3,4.7)(5.6,5.4)
    \vertex(5.3,4.7)
    \vertex(4.8,5.7)
    \vertex(5.3,6.2)
    \vertex(6.3,5.7)
    \vertex(6.3,4.7)
  \end{hexboard}
  \]
  \caption{A vertex, an edge, and a hyperedge.}
  \label{fig:hypergraph-picture}
\end{figure}

The Shannon game is played on a finite hypergraph with two
distinguished vertices called \emph{terminals}. The cells are the
non-terminal vertices. As usual, the players take turns coloring the
cells, and at the end of the game, Black wins if and only if the two
terminals are connected by an uninterrupted path of black vertices,
i.e., if there exists a sequence of vertices $v_1,\ldots,v_n$ such
that $v_1$ and $v_n$ are the terminals, each $v_i$ is connected to
$v_{i+1}$ by a hyperedge, and $v_2,\ldots,v_{n-1}$ are colored black.
If Black does not have such a connection, then White wins. We note in
passing that the class of Shannon games is not self-dual, i.e.,
exchanging the roles of Black and White in a Shannon game does not
necessarily result in a Shannon game.

To see why Hex is a special case of the Shannon game, consider
the following graph, which corresponds to $5\times 5$ Hex:
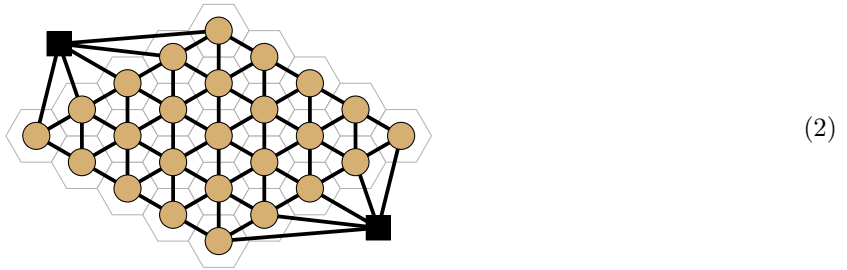
\begin{equation}
  \label{eqn:shannon-5x5}
  \m{$\begin{hexboard}
      {
        \definecolor{linecolor}{HTML}{B0B0B0}
        \definecolor{cellcolorflat}{HTML}{FFFFFF}
        \tikzset{flatboard/.style={draw=linecolor,fill=cellcolorflat,very thin}}
        \foreach\i in {1,...,5} {
          \foreach\j in {1,...,5} {
            \hex(\i,\j)
          }
        }
      }
      \foreach\i in {1,...,4} {
        \foreach\j in {1,...,5} {
          \gedge(\i,\j)(\i+1,\j);
        }
      }
      \foreach\i in {1,...,5} {
        \foreach\j in {1,...,4} {
          \gedge(\i,\j)(\i,\j+1);
        }
      }
      \foreach\i in {1,...,4} {
        \foreach\j in {2,...,5} {
          \gedge(\i,\j)(\i+1,\j-1);
        }
      }
      \foreach\i in {1,...,5} {
        \gedge(\i,1)(3,-0.5);
        \gedge(\i,5)(3,6.5);
      }
      \foreach\i in {1,...,5} {
        \foreach\j in {1,...,5} {
          \vertex(\i,\j);
        }
      }
      \terminal(3,-0.5);
      \terminal(3,6.5);
    \end{hexboard}
    $}
\end{equation}
Here, the brown circles represent the non-terminal vertices and
correspond to the cells of the Hex board, which are faintly shown in
the background. The two black squares are the terminals and correspond
to Black's board edges. Note that two vertices are adjacent in the
graph if and only if the corresponding cells or board edges are
adjacent on the Hex board. Therefore, the games in
{\eqref{eqn:hex-5x5}} and {\eqref{eqn:shannon-5x5}} are isomorphic.

A Shannon game is \emph{planar} if its underlying hypergraph is
planar, with terminals on the outside. Recall that a hypergraph is
\emph{planar} if it can be drawn in the 2-dimensional plane without
any hyperedges crossing. The spaces between the hyperedges are called
\emph{faces}, and every finite planar hypergraph has an \emph{outside
face}, i.e., a face that stretches to infinity. When we require the
terminals to be \emph{on the outside}, we mean that they must be
adjacent to this outside face. For example, the Shannon game in
{\eqref{eqn:shannon-5x5}} is planar.

The \emph{degree} of a vertex in a hypergraph is the number of
hyperedges that contain it. We say that a Shannon game has degree $n$
if every vertex has degree at most $n$. For example, the game in
{\eqref{eqn:shannon-5x5}} has degree $6$. If a game is planar and of
degree $n$, we also say that the game is \emph{$n$-planar}.

\begin{remark}\label{rem:contract}
  In the literature, vertex Shannon games are usually played on
  graphs, rather than on hypergraphs. The reason we consider
  hypergraphs instead of graphs is that it yields a more convenient
  definition of planarity and degree. However, for games of degree at
  least 3, the two definitions are equivalent if we permit games where
  some of the cells have already been colored at the start. Here is an
  example:
  \begin{equation}
    \label{eqn:shannon-position}
    \m{$
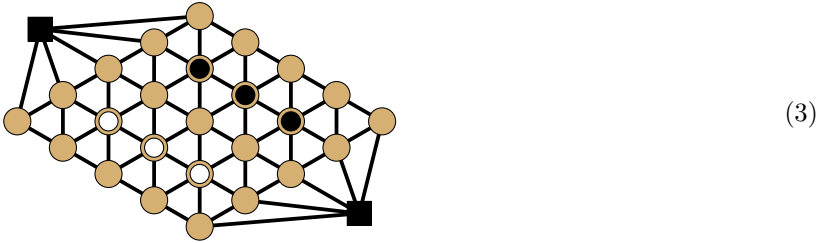
      \begin{hexboard}
        \foreach\i in {1,...,4} {
          \foreach\j in {1,...,5} {
            \gedge(\i,\j)(\i+1,\j);
          }
        }
        \foreach\i in {1,...,5} {
          \foreach\j in {1,...,4} {
            \gedge(\i,\j)(\i,\j+1);
          }
        }
        \foreach\i in {1,...,4} {
          \foreach\j in {2,...,5} {
            \gedge(\i,\j)(\i+1,\j-1);
          }
        }
        \foreach\i in {1,...,5} {
          \gedge(\i,1)(3,-0.5);
          \gedge(\i,5)(3,6.5);
        }
        \foreach\i in {1,...,5} {
          \foreach\j in {1,...,5} {
            \vertex(\i,\j);
          }
        }
        \terminal(3,-0.5);
        \terminal(3,6.5);
        \whitevertex(2,2);
        \whitevertex(2,3);
        \whitevertex(2,4);
        \blackvertex(4,2);
        \blackvertex(4,3);
        \blackvertex(4,4);
      \end{hexboard}
      $}
  \end{equation}
  In this game, some of the cells are already occupied by black or
  white stones. Such a game can be simplified by deleting white
  vertices and contracting black (non-terminal) vertices. When
  deleting a white vertex, we also remove the vertex from any
  hyperedges that contained it. To contract a black vertex, we remove
  the vertex, and merge all of its incident hyperedges into a single
  hyperedge. We do this repeatedly until no more black or white cells
  are left. This results in an isomorphic game; for example, the
  position in {\eqref{eqn:shannon-position}} is isomorphic to the
  following:
  \begin{equation}
    \label{eqn:shannon-hypergraph}
    \m{$
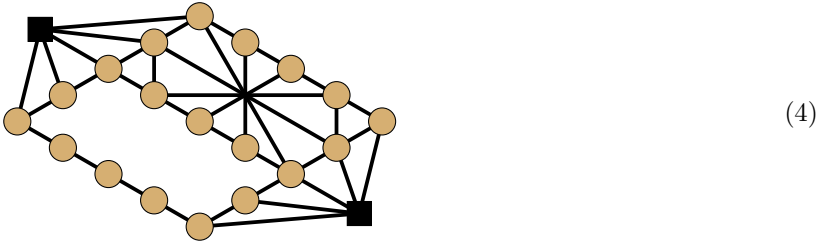
      \begin{hexboard}
        \foreach\i in {1,...,5} {
          \gedge(\i,1)(3,-0.5);
          \gedge(\i,5)(3,6.5);
        }
        \foreach\i in {1,...,4} {
          \gedge(\i,1)(\i+1,1);
          \gedge(\i,5)(\i+1,5);
          \gedge(1,\i)(1,\i+1);
          \gedge(3,\i)(3,\i+1);
          \gedge(5,\i)(5,\i+1);
        }
        \foreach\i in {1,...,4} {
          \gedge(3,\i+1)(4,3);
          \gedge(5,\i)(4,3);
        }
        \gedge(4,1)(4,3);
        \gedge(4,5)(4,3);
        \gedge(5,4)(4,5);
        \gedge(3,2)(4,1);
        \vertex(1,1);
        \vertex(1,2);
        \vertex(1,3);
        \vertex(1,4);
        \vertex(1,5);
        \vertex(2,1);
        \vertex(2,5);
        \vertex(3,1);
        \vertex(3,2);
        \vertex(3,3);
        \vertex(3,4);
        \vertex(3,5);
        \vertex(4,1);
        \vertex(4,5);
        \vertex(5,1);
        \vertex(5,2);
        \vertex(5,3);
        \vertex(5,4);
        \vertex(5,5);
        \terminal(3,-0.5);
        \terminal(3,6.5);
      \end{hexboard}
      $}
  \end{equation}
  Note that the operations of deleting a white vertex or contracting a
  black one preserve the planarity of the hypergraph, and do not
  increase the degree. (Remember that the degree measures the number
  of hyperedges at a given vertex, not the number of vertices at a
  given hyperedge). Therefore, any position in a (planar) graph
  Shannon game is isomorphic to a (planar) hypergraph Shannon game of
  the same degree.

  Conversely, any (planar) Shannon hypergraph game of degree $3$ or
  higher is isomorphic to a position in a (planar) Shannon graph game
  of the same degree in which some cells have possibly already been
  colored black. Namely, each hyperedge with $k$ vertices, as shown on
  the left, can be replaced by a cluster of $k-2$ additional
  black-colored cells, as shown on the right:
  
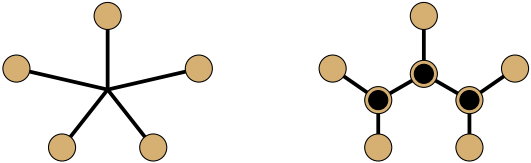
\[
  \begin{hexboard}
    \gedge(1.6,0.4)(2.2,1.8)
    \gedge(0.6,2.4)(2.2,1.8)
    \gedge(1.6,3.4)(2.2,1.8)
    \gedge(3.6,2.4)(2.2,1.8)
    \gedge(3.6,0.4)(2.2,1.8)
    \vertex(1.6,0.4)
    \vertex(0.6,2.4)
    \vertex(1.6,3.4)
    \vertex(3.6,2.4)
    \vertex(3.6,0.4)
  \end{hexboard}
  \qquad
  \qquad
  \begin{hexboard}
    \gedge(1.6,0.4)(1.5,1.5)
    \gedge(0.6,2.4)(1.5,1.5)
    \gedge(1.6,3.4)(2.5,2.5)
    \gedge(3.6,2.4)(2.5,2.5)
    \gedge(3.6,0.4)(2.5,1.5)
    \gedge(1.5,1.5)(2.5,1.5)
    \gedge(2.5,2.5)(2.5,1.5)
    \vertex(1.6,0.4)
    \vertex(0.6,2.4)
    \vertex(1.6,3.4)
    \vertex(3.6,2.4)
    \vertex(3.6,0.4)
    \vertex(1.5,1.5)\blackvertex(1.5,1.5)
    \vertex(2.5,1.5)\blackvertex(2.5,1.5)
    \vertex(2.5,2.5)\blackvertex(2.5,2.5)
  \end{hexboard}
  \]
\end{remark}

\subsection{Hex is the universal 3-planar Shannon game}

We are now ready to state the main theorem of this section:

\begin{theorem}\label{thm:hex-universal}
  Every 3-planar Shannon game is isomorphic to a Hex position.
  Conversely, every Hex position is isomorphic to a 3-planar Shannon
  game.
\end{theorem}

The reader is invited to try to prove this theorem now, before looking
at the proof. Note that the second part is not obvious, since Hex, as
shown in {\eqref{eqn:shannon-5x5}}, appears to be of degree 6, rather
than 3. Although the proof is extremely simple, to our knowledge, this
result was not previously known.

\begin{proof}[Proof of Theorem~\ref{thm:hex-universal}]
  For the first claim, consider any 3-planar Shannon game, such as the
  one shown in Figure~\ref{fig:planar}(a). It is obvious that this
  game can be embedded in a sufficiently large Hex board in such a way
  that every non-terminal vertex corresponds to an empty cell, every
  terminal vertex corresponds to a black board edge, every hyperedge
  corresponds to a connected group of black cells, and everything else
  (i.e., the space between the hyperedges) is filled with white
  cells. An example of such an embedding is shown in
  Figure~\ref{fig:planar}(b). Specifically, since each vertex has
  degree at most 3, there are at most 3 hyperedges that must be
  attached to it, and this is possible because the corresponding cell
  on the Hex board is a hexagon. Also, since the terminals are on the
  outside of the hypergraph, they can be connected to opposite black
  board edges.

  For the converse, refer to Figure~\ref{fig:invisible-terminals}.
  Consider a Hex board, such as the one in (a). Insert black and white
  polygons between the cells, as shown in (b). Observe that for any
  pair of adjacent cells of (a), the corresponding cells in (b) are
  connected by both a black and a white polygon. Therefore, for any
  position in which Black's cells are connected in (a), Black's cells
  are also connected in (b), and dually for White. It follows that the
  game in (a) is isomorphic to the one in (b). Moreover, in (b), it is
  obvious that each cell has degree at most 3. In fact, the game in
  (b) is isomorphic to the 3-planar Shannon game shown in (c). This
  works for boards of any size. Moreover, if we start from a Hex
  position, rather than an empty board, we can color the corresponding
  vertices in (c), and then delete the white ones and contract the
  black ones as in Remark~\ref{rem:contract}, which still yields a
  3-planar Shannon game.
\end{proof}

\begin{remark}
  Due to Theorem~\ref{thm:hex-universal}, every problem about Hex
  positions is equivalent to a problem about 3-planar Shannon games,
  and vice versa. But since Hex is usually played starting from an
  empty board, and not from some arbitrary initial position, there are
  many elements of Hex strategy that are specific to Hex, such as
  templates, ladders, etc. In other words, a good Hex player is not
  necessarily a good player of arbitrary 3-planar Shannon
  games. Naturally, positions such as the one in
  Figure~\ref{fig:planar} do not normally arise in actual Hex games.
\end{remark}

\subsection{Application: Embedding Hex in itself}

The proof of Theorem~\ref{thm:hex-universal} immediately implies that
Hex can be isomorphically embedded in a larger version of itself. For
example, Figure~\ref{fig:hex-auto} shows two different ways of
embedding $4\times 4$ Hex in larger Hex boards. In fact, the second of
these embeddings is an iterated version of the first.
\begin{figure}
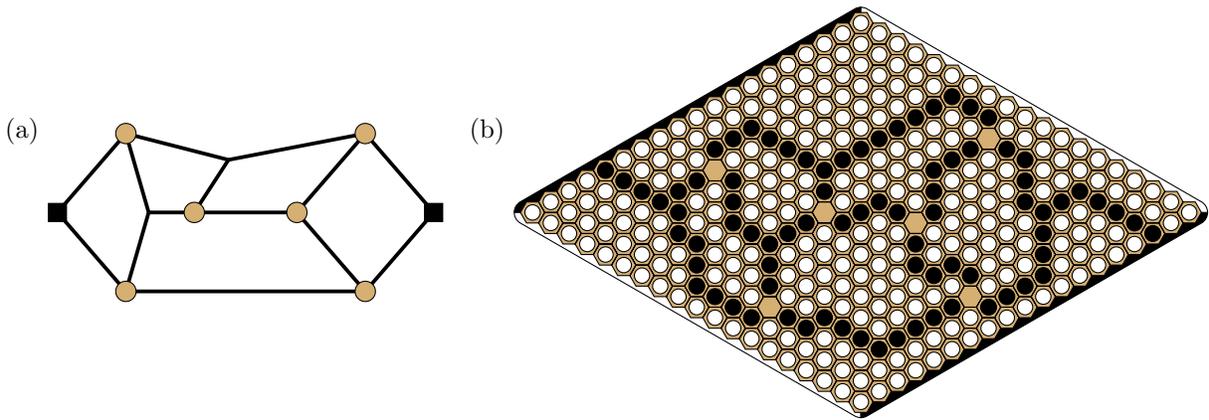

  \[
  \raisebox{1cm}{(a)}~
  \m{$
    \begin{hexboard}[scale=0.75]
      \gedge(1,1)(4,0);
      \gedge(0,4)(2.3333,2.3333);
      \gedge(4,0)(2.3333,2.3333);
      \gedge(3,3)(2.3333,2.3333);
      \gedge(4.5,4.5)(7.5,3.5);
      \gedge(3.5,7.5)(6.5,6.5);
      \gedge(1,1)(0,4);
      \gedge(4,0)(4.8333,2.1667);
      \gedge(3,3)(4.8333,2.1667);
      \gedge(7.5,3.5)(4.8333,2.1667);
      \gedge(4.5,4.5)(3.5,7.5);
      \gedge(7.5,3.5)(6.5,6.5);
      \gedge(3,3)(4.5,4.5);
      \gedge(0,4)(3.5,7.5);
      \terminal(1,1);
      \vertex(3,3);
      \vertex(4.5,4.5);
      \terminal(6.5,6.5);
      \vertex(4,0);
      \vertex(0,4);
      \vertex(7.5,3.5);
      \vertex(3.5,7.5);
    \end{hexboard}
    $}
  \quad
  \raisebox{1cm}{(b)}~
  \m{$
    \begin{hexboard}[scale=0.4]
      \board(19,19)
      \orange(3,12)\orange(9,9)\orange(8,4)\orange(11,12)\orange(9,17)
      \orange(17,10)
      \white(1,1)\white(2,1)\white(3,1)\white(4,1)
      \black(5,1)\white(6,1)\white(7,1)\white(8,1)
      \white(9,1)\white(10,1)\white(11,1)\white(12,1)
      \white(13,1)\white(14,1)\white(15,1)\white(16,1)
      \white(17,1)\white(18,1)\white(19,1)\white(1,2)
      \white(2,2)\white(3,2)\white(4,2)\black(5,2)
      \white(6,2)\white(7,2)\white(8,2)\white(9,2)
      \white(10,2)\white(11,2)\white(12,2)\white(13,2)
      \white(14,2)\white(15,2)\white(16,2)\white(17,2)
      \white(18,2)\white(19,2)\white(1,3)\white(2,3)
      \white(3,3)\white(4,3)\black(5,3)\white(6,3)
      \white(7,3)\white(8,3)\black(9,3)\black(10,3)
      \black(11,3)\white(12,3)\white(13,3)\white(14,3)
      \white(15,3)\white(16,3)\white(17,3)\white(18,3)
      \white(19,3)\white(1,4)\white(2,4)\white(3,4)
      \white(4,4)\black(5,4)\black(6,4)\black(7,4)
      \white(9,4)\white(10,4)\black(11,4)\white(12,4)
      \white(13,4)\white(14,4)\white(15,4)\white(16,4)
      \white(17,4)\white(18,4)\white(19,4)\white(1,5)
      \white(2,5)\white(3,5)\white(4,5)\black(5,5)
      \white(6,5)\white(7,5)\black(8,5)\white(9,5)
      \white(10,5)\black(11,5)\white(12,5)\white(13,5)
      \white(14,5)\white(15,5)\white(16,5)\white(17,5)
      \white(18,5)\white(19,5)\white(1,6)\white(2,6)
      \white(3,6)\black(4,6)\white(5,6)\white(6,6)
      \black(7,6)\white(8,6)\white(9,6)\white(10,6)
      \black(11,6)\white(12,6)\white(13,6)\white(14,6)
      \white(15,6)\white(16,6)\white(17,6)\white(18,6)
      \white(19,6)\white(1,7)\white(2,7)\white(3,7)
      \black(4,7)\white(5,7)\black(6,7)\white(7,7)
      \white(8,7)\white(9,7)\white(10,7)\black(11,7)
      \black(12,7)\black(13,7)\black(14,7)\black(15,7)
      \black(16,7)\black(17,7)\black(18,7)\white(19,7)
      \white(1,8)\white(2,8)\black(3,8)\white(4,8)
      \white(5,8)\black(6,8)\white(7,8)\white(8,8)
      \white(9,8)\black(10,8)\white(11,8)\white(12,8)
      \white(13,8)\white(14,8)\white(15,8)\white(16,8)
      \white(17,8)\black(18,8)\white(19,8)\white(1,9)
      \black(2,9)\white(3,9)\white(4,9)\white(5,9)
      \black(6,9)\black(7,9)\black(8,9)\white(10,9)
      \white(11,9)\white(12,9)\white(13,9)\white(14,9)
      \white(15,9)\white(16,9)\white(17,9)\black(18,9)
      \white(19,9)\white(1,10)\black(2,10)\white(3,10)
      \white(4,10)\black(5,10)\white(6,10)\white(7,10)
      \white(8,10)\black(9,10)\black(10,10)\black(11,10)
      \white(12,10)\white(13,10)\black(14,10)\black(15,10)
      \black(16,10)\white(18,10)\white(19,10)\white(1,11)
      \black(2,11)\white(3,11)\black(4,11)\white(5,11)
      \white(6,11)\white(7,11)\white(8,11)\white(9,11)
      \white(10,11)\black(11,11)\white(12,11)\black(13,11)
      \white(14,11)\white(15,11)\white(16,11)\black(17,11)
      \white(18,11)\white(19,11)\white(1,12)\black(2,12)
      \white(4,12)\white(5,12)\white(6,12)\white(7,12)
      \white(8,12)\white(9,12)\white(10,12)\black(12,12)
      \white(13,12)\white(14,12)\white(15,12)\white(16,12)
      \black(17,12)\white(18,12)\white(19,12)\white(1,13)
      \white(2,13)\black(3,13)\white(4,13)\white(5,13)
      \white(6,13)\white(7,13)\white(8,13)\white(9,13)
      \black(10,13)\white(11,13)\white(12,13)\white(13,13)
      \white(14,13)\white(15,13)\black(16,13)\white(17,13)
      \white(18,13)\white(19,13)\white(1,14)\white(2,14)
      \black(3,14)\black(4,14)\white(5,14)\white(6,14)
      \white(7,14)\white(8,14)\black(9,14)\white(10,14)
      \white(11,14)\white(12,14)\white(13,14)\white(14,14)
      \black(15,14)\white(16,14)\white(17,14)\white(18,14)
      \white(19,14)\white(1,15)\white(2,15)\white(3,15)
      \black(4,15)\white(5,15)\white(6,15)\white(7,15)
      \white(8,15)\black(9,15)\black(10,15)\white(11,15)
      \white(12,15)\white(13,15)\white(14,15)\black(15,15)
      \black(16,15)\black(17,15)\white(18,15)\white(19,15)
      \white(1,16)\white(2,16)\white(3,16)\black(4,16)
      \white(5,16)\white(6,16)\white(7,16)\white(8,16)
      \white(9,16)\black(10,16)\white(11,16)\white(12,16)
      \white(13,16)\black(14,16)\white(15,16)\white(16,16)
      \black(17,16)\white(18,16)\white(19,16)\white(1,17)
      \white(2,17)\white(3,17)\black(4,17)\black(5,17)
      \black(6,17)\black(7,17)\black(8,17)\white(10,17)
      \white(11,17)\white(12,17)\black(13,17)\white(14,17)
      \white(15,17)\white(16,17)\black(17,17)\white(18,17)
      \white(19,17)\white(1,18)\white(2,18)\white(3,18)
      \white(4,18)\white(5,18)\white(6,18)\white(7,18)
      \white(8,18)\black(9,18)\black(10,18)\black(11,18)
      \black(12,18)\white(13,18)\white(14,18)\white(15,18)
      \white(16,18)\black(17,18)\white(18,18)\white(19,18)
      \white(1,19)\white(2,19)\white(3,19)\white(4,19)
      \white(5,19)\white(6,19)\white(7,19)\white(8,19)
      \white(9,19)\white(10,19)\white(11,19)\white(12,19)
      \white(13,19)\white(14,19)\white(15,19)\white(16,19)
      \black(17,19)\white(18,19)\white(19,19)
    \end{hexboard}
    $}
  \]
  \caption{Converting a 3-planar Shannon game to a Hex position.}
  \label{fig:planar}
\end{figure}

\begin{figure}
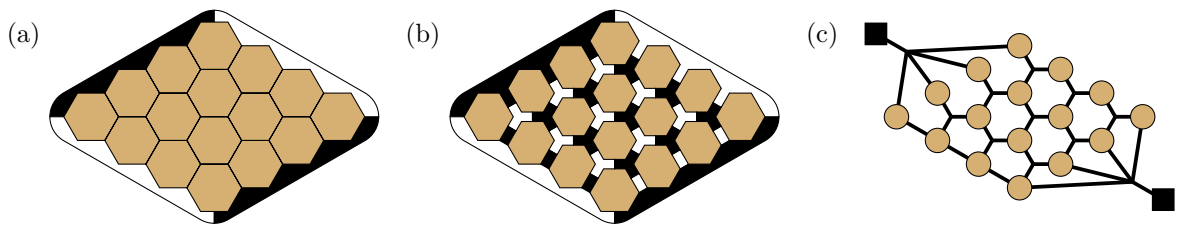

  \[
  \raisebox{1cm}{(a)}~
  \m{$
    \begin{hexboard}[scale=0.9]
      \board(4,4)
    \end{hexboard}
    $}
  \quad
  \raisebox{1cm}{(b)}~
  \m{$
    \begin{hexboard}[scale=0.9]
      \board(4,4)
      \foreach\i in {1,...,3} {
        \foreach\j in {1,...,3} {
          \tripodwhite(\i,\j);
        }
      }
      \foreach\i in {2,...,4} {
        \foreach\j in {2,...,4} {
          \tripodblack(\i,\j);
        }
      }
      \foreach\i in {1,...,3} {
        \monopodwhite(\i,4);
      }
      \foreach\i in {2,...,4} {
        \monopodblack(\i,1);
      }
      \foreach\i in {1,...,3} {
        \monopodwhiteb(4,\i);
      }
      \foreach\i in {2,...,4} {
        \monopodblackb(1,\i);
      }
    \end{hexboard}
    $}
  \quad
  \raisebox{1cm}{(c)}\quad
  \m{$\begin{hexboard}[scale=0.9]
      \foreach\i in {2,...,4} {
        \foreach\j in {2,...,4} {
          \gedge(\i,\j)(\i-0.3333,\j-0.3333);
        }
      }
      \foreach\i in {1,...,3} {
        \foreach\j in {2,...,4} {
          \gedge(\i,\j)(\i+0.6667,\j-0.3333);
        }
      }
      \foreach\i in {2,...,4} {
        \foreach\j in {1,...,3} {
          \gedge(\i,\j)(\i-0.3333,\j+0.6667);
        }
      }
      \foreach\i in {1,...,3} {
        \gedge(1,\i)(1,\i+1);
      }
      \foreach\i in {1,...,4} {
        \gedge(\i,1)(2.5,-0.25);
        \gedge(\i,4)(2.5,5.25);
      }
      \gedge(2.5,-0.25)(2.5,-0.75);
      \gedge(2.5,5.25)(2.5,5.75);
      \foreach\i in {1,...,4} {
        \foreach\j in {1,...,4} {
          \vertex(\i,\j);
        }
      }
      \terminal(2.5,-1);
      \terminal(2.5,6);
    \end{hexboard}
    $}
  \]
  \caption{Converting a Hex position to a 3-planar Shannon game.}
  \label{fig:invisible-terminals}
\end{figure}

\subsection{{\BridgIt} is the universal 2-planar Shannon game}

{\BridgIt} is a connection game that was invented by David Gale. It was
described in Martin Gardner's 1958 Scientific American column
{\cite{Gardner-Gale}} under the name ``Gale''. It is played on a
square board such as the one shown in Figure~\ref{fig:bridg-it}(a).
\begin{figure}
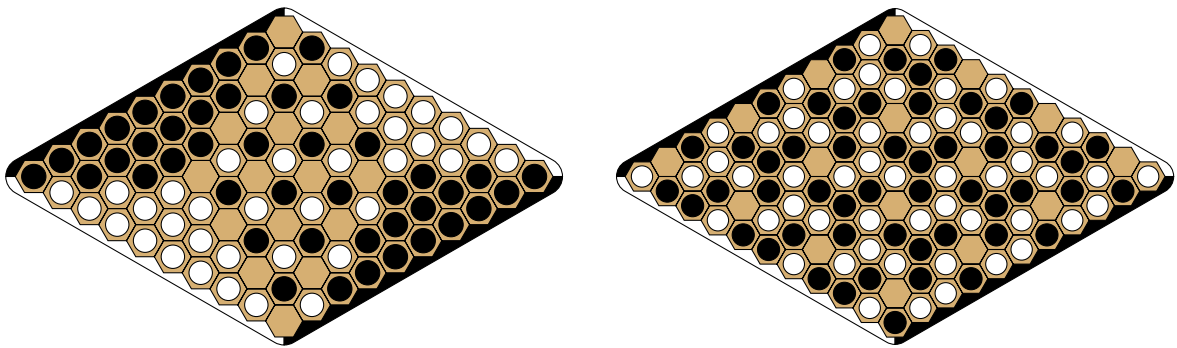

  \[
  \begin{hexboard}[scale=0.61,baseline={($(current bounding box.center)-(0,1ex)$)}]
    \board(10,10)
    \foreach\i in {0,...,3} {
      \foreach\j in {0,...,3} {
        \orange(4+2*\i-\j,4-\i+2*\j)
      }
    }
    \black(1,1)\black(2,1)\black(3,1)\black(4,1)\black(5,1)\black(6,1)
    \black(7,1)\black(8,1)\black(9,1)\white(1,2)\black(2,2)\black(3,2)
    \black(4,2)\black(5,2)\black(6,2)\black(7,2)\white(9,2)\black(10,2)
    \white(1,3)\white(2,3)\black(3,3)\black(4,3)\black(5,3)\white(7,3)
    \black(8,3)\white(10,3)\white(1,4)\white(2,4)\white(3,4)\white(5,4)
    \black(6,4)\white(8,4)\black(9,4)\white(10,4)\white(1,5)\white(2,5)
    \white(3,5)\black(4,5)\white(6,5)\black(7,5)\white(9,5)\white(10,5)
    \white(1,6)\white(2,6)\white(4,6)\black(5,6)\white(7,6)\black(8,6)
    \white(9,6)\white(10,6)\white(1,7)\white(2,7)\black(3,7)\white(5,7)
    \black(6,7)\white(8,7)\white(9,7)\white(10,7)\white(1,8)\white(3,8)
    \black(4,8)\white(6,8)\black(7,8)\black(8,8)\white(9,8)\white(10,8)
    \white(1,9)\black(2,9)\white(4,9)\black(5,9)\black(6,9)\black(7,9)
    \black(8,9)\black(9,9)\white(10,9)\white(2,10)\black(3,10)\black(4,10)
    \black(5,10)\black(6,10)\black(7,10)\black(8,10)\black(9,10)\black(10,10)
  \end{hexboard}
  \qquad
  \begin{hexboard}[scale=0.555,baseline={($(current bounding box.center)-(0,1ex)$)}]
    \board(11,11)
    \foreach\i in {2,5,8,11} {
      \foreach\j in {1,4,7,10} {
        \orange(\i,\j)
      }
    }
    \white(1,1)\nofil(2,1)\black(3,1)\white(4,1)\nofil(5,1)\black(6,1)\white(7,1)\nofil(8,1)\black(9,1)\white(10,1)\nofil(11,1)
    \black(1,2)\white(2,2)\white(3,2)\black(4,2)\white(5,2)\white(6,2)\black(7,2)\white(8,2)\white(9,2)\black(10,2)\white(11,2)
    \black(1,3)\black(2,3)\white(3,3)\black(4,3)\black(5,3)\white(6,3)\black(7,3)\black(8,3)\white(9,3)\black(10,3)\black(11,3)
    \white(1,4)\nofil(2,4)\black(3,4)\white(4,4)\nofil(5,4)\black(6,4)\white(7,4)\nofil(8,4)\black(9,4)\white(10,4)\nofil(11,4)
    \black(1,5)\white(2,5)\white(3,5)\black(4,5)\white(5,5)\white(6,5)\black(7,5)\white(8,5)\white(9,5)\black(10,5)\white(11,5)
    \black(1,6)\black(2,6)\white(3,6)\black(4,6)\black(5,6)\white(6,6)\black(7,6)\black(8,6)\white(9,6)\black(10,6)\black(11,6)
    \white(1,7)\nofil(2,7)\black(3,7)\white(4,7)\nofil(5,7)\black(6,7)\white(7,7)\nofil(8,7)\black(9,7)\white(10,7)\nofil(11,7)
    \black(1,8)\white(2,8)\white(3,8)\black(4,8)\white(5,8)\white(6,8)\black(7,8)\white(8,8)\white(9,8)\black(10,8)\white(11,8)
    \black(1,9)\black(2,9)\white(3,9)\black(4,9)\black(5,9)\white(6,9)\black(7,9)\black(8,9)\white(9,9)\black(10,9)\black(11,9)
    \white(1,10)\nofil(2,10)\black(3,10)\white(4,10)\nofil(5,10)\black(6,10)\white(7,10)\nofil(8,10)\black(9,10)\white(10,10)\nofil(11,10)
    \black(1,11)\white(2,11)\white(3,11)\black(4,11)\white(5,11)\white(6,11)\black(7,11)\white(8,11)\white(9,11)\black(10,11)\white(11,11)
  \end{hexboard}
  \]
  \caption{Two embeddings of $4\times 4$ Hex in larger boards.}
  \label{fig:hex-auto}
\end{figure}

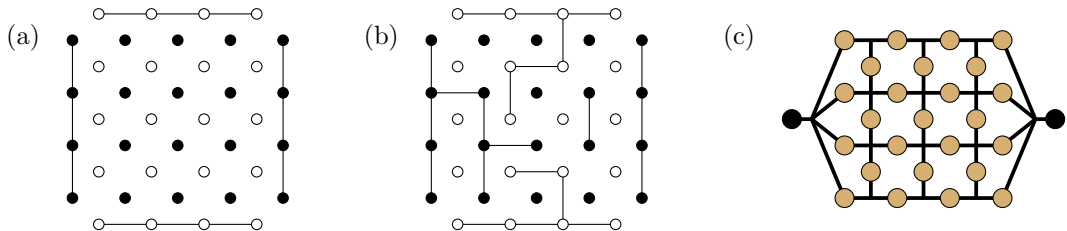
\begin{figure}
  \[
  \raisebox{1cm}{(a)}\quad
  \m{$
    \begin{tikzpicture}[scale=0.7]
      \draw (0,0.5) -- (0,3.5);
      \draw (4,0.5) -- (4,3.5);
      \draw (0.5,0) -- (3.5,0);
      \draw (0.5,4) -- (3.5,4);
      \foreach\i in {0,...,4} {
        \foreach\j in {0,...,3} {
          \draw[fill=black] (\i,\j+0.5) circle (0.1);
          \draw[fill=white] (\j+0.5,\i) circle (0.1);
        }
      }
    \end{tikzpicture}
    $}
  \hspace{1cm}
  \raisebox{1cm}{(b)}\quad
  \m{$
    \begin{tikzpicture}[scale=0.7]
      \draw (0,0.5) -- (0,3.5);
      \draw (4,0.5) -- (4,3.5);
      \draw (0.5,0) -- (3.5,0);
      \draw (0.5,4) -- (3.5,4);
      \draw (0,2.5) -- (1,2.5) -- (1,1.5) -- (2,1.5);
      \draw (3,1.5) -- (3,2.5);
      \draw (1,1.5) -- (1,0.5);
      \draw (2.5,0) -- (2.5,1) -- (1.5,1);
      \draw (2.5,4) -- (2.5,3) -- (1.5,3) -- (1.5,2);
      \foreach\i in {0,...,4} {
        \foreach\j in {0,...,3} {
          \draw[fill=black] (\i,\j+0.5) circle (0.1);
          \draw[fill=white] (\j+0.5,\i) circle (0.1);
        }
      }
    \end{tikzpicture}
    $}
  \hspace{1cm}
  \raisebox{1cm}{(c)}\quad
  \m{$
    \begin{tikzpicture}[scale=0.7]
      \draw[tterminal] (-0.5,2) -- (-0.125,2);
      \draw[tterminal] (4.125,2) -- (4.5,2);
      \foreach\i in {0.5,...,3.5} {
        \draw[tterminal] (-0.125,2) -- (0.5,\i);
        \draw[tterminal] (0.5,\i) -- (3.5,\i);
        \draw[tterminal] (3.5,\i) -- (4.125,2);
      }
      \foreach\i in {1,...,3} {
        \draw[tterminal] (\i,0.5) -- (\i,3.5);
      }
      \draw[fill=black] (-0.5,2) circle (0.18);
      \draw[fill=black] (4.5,2) circle (0.18);
      \foreach\i in {1,...,3} {
        \foreach\j in {1,...,3} {
          \draw[fill=vertexcolor] (\i,\j) circle (0.18);
        }
      }
      \foreach\i in {0.5,...,3.5} {
        \foreach\j in {0.5,...,3.5} {
          \draw[fill=vertexcolor] (\i,\j) circle (0.18);
        }
      }
    \end{tikzpicture}
    $}
  \]
  \caption{(a) The game of {\BridgIt}. (b) A partially completed
    game. (c) An isomorphic Shannon game.}
  \label{fig:bridg-it}
\end{figure}
On Black's turn, Black connects two adjacent black dots by a
straight line. On White's turn, White similarly connects two adjacent
white dots. Crossing lines are not permitted. Black's goal is to
connect the black dots on the left with the ones on the right, and
White's goal is to connect the white dots at the top with the ones at
the bottom. Figure~\ref{fig:bridg-it}(b) shows what the game may look
like after a few moves.

Like Hex, {\BridgIt} has the property that one player always wins:
draws are not possible. In fact, {\BridgIt} is a 2-planar Shannon
game. Its hypergraph is shown in Figure~\ref{fig:bridg-it}(c).  Note
that the black dots of the {\BridgIt} board become hyperedges of the
Shannon game, and the potential black edges of the {\BridgIt} board
become vertices. Moreover, {\BridgIt} is the universal 2-planar
Shannon game. On the one hand, the game in
Figure~\ref{fig:bridg-it}(c) is obviously of degree 2. On the other
hand, any 2-planar Shannon game can be embedded on a sufficiently
large {\BridgIt} board, in a way analogous to how it was done for Hex
in Figure~\ref{fig:planar}.

Since every planar Shannon game of degree 2 is also of degree 3, it
follows that {\BridgIt} can be isomorphically embedded in Hex. Indeed,
Figure~\ref{fig:bridgit-in-hex} shows a Hex realization of the game
from Figure~\ref{fig:bridg-it}(a). This realization previously
appeared in {\cite[Fig.~6.5]{Hayward-full-story}}.
\begin{figure}
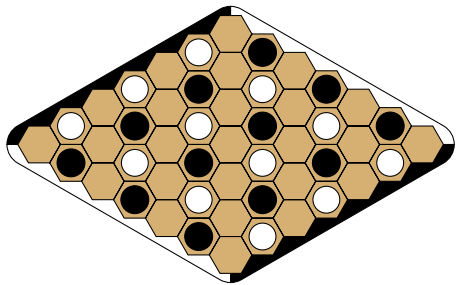

  \[
  \begin{hexboard}[scale=0.7]
    \board(7,7)
    \foreach\i in {1,3,...,7} {
      \foreach\j in {2,4,...,6} {
        \black(\i,\j);
        \white(\j,\i);
      }
    }
    \foreach\i in {1,3,...,7} {
      \foreach\j in {1,3,...,7} {
        \orange(\i,\j);
      }
    }
    \foreach\i in {2,4,...,6} {
      \foreach\j in {2,4,...,6} {
        \orange(\i,\j);
      }
    }
  \end{hexboard}
  \]
  \caption{A Hex realization of {\BridgIt}}
  \label{fig:bridgit-in-hex}
\end{figure}

The converse is not true: Hex cannot be isomorphically embedded in
{\BridgIt}.  In fact, there is a qualitative difference between the
two games. In {\BridgIt}, there exists a computationally efficient
algorithm for determining the winner of any position
{\cite{edge-Shannon-solved}}. The corresponding problem in Hex is
PSPACE-complete {\cite{Reisch}}. Therefore, the passage from degree 2
to degree 3 is non-trivial.

We finish this subsection by remarking that there is a variant of
Shannon games called \emph{edge Shannon games}. These games are played
on a graph, and the players alternately color \emph{edges}, rather
than vertices. As usual, Black wins if at the end of the game, two
distinguished terminals are connected by black edges. Note that every
hypergraph has a \emph{dual}, where the vertices of the dual
hypergraph are the hyperedges of the original hypergraph, and vice
versa. A hypergraph of degree 2 is just the dual of an ordinary
graph. In this way, an edge Shannon game is essentially the same thing
as a hypergraph vertex Shannon game of degree 2, and {\BridgIt} is the
universal planar edge Shannon game. In fact, like {\BridgIt}, the edge
Shannon game is also solved: there exists an efficient algorithm for
determining which player has a winning strategy
{\cite{edge-Shannon-solved}}

\subsection{Planar Shannon games of degree greater than 3}

We just saw that {\BridgIt} is the universal 2-planar Shannon game,
Hex is the universal 3-planar Shannon game, and the passage from
degree 2 to degree 3 has a dramatic effect on the complexity of planar
Shannon games. It is therefore natural to ask whether games of degree
4 or higher get even more complex or interesting. We may also ask what
class of games is universal for $n$-planar Shannon games when $n\geq
4$.

The answer to the first question is no: Hex is already
PSPACE-complete, and the problem of deciding a winning strategy for
any Shannon game is in PSPACE. Therefore, planar Shannon games of
degree 4 and higher are not more complex than Hex from a computational
point of view.

The answer to the second question is also easy: a universal planar
Shannon game of degree 4 is Hex with octagons. More precisely, it is
not necessary for \emph{all} cells to be octagons; there just need to
be an unbounded number of them as the board size increases. One way to
make a (topological) octagon on a Hex board is to merge two adjacent
cells into a single cell. Having an unbounded number of octagons is
then easily seen to be sufficient for embedding any planar Shannon
game of degree 4, in a manner analogous to
Figure~\ref{fig:planar}. Similarly, a universal planar Shannon game of
degree 5 is Hex with sufficiently many decagons, and so on.

\section{A quick primer on the combinatorial game theory of Hex}
\label{sec:background-cgt}

The fundamentals of the combinatorial game theory of Hex were
developed in {\cite{S2022-hex-cgt}}. Here, we briefly summarize some
definitions and results that are useful in the rest of this paper. For
a more complete and thorough introduction, see
{\cite{S2022-hex-cgt}}. For combinatorial game theory in general, see
{\cite{ONAG, WinningWays}}. In this section, we will mostly be
talking about Hex, but the same concepts also apply to more general
Shannon games.

\subsection{Regions and outcome posets}
\label{ssec:regions}

As mentioned in the introduction, a \emph{region} of a Hex board is a
subset of its cells; some or all of the cells may already be occupied
by black and/or white stones. By a \emph{completion} of the region, we
mean any possible way of filling its empty cells with black and/or
white stones. If $a$ and $b$ are completions of the same region, we
say that $a\leq b$ if for every way of filling the rest of the board
(i.e., the outside of the region) with stones, if $a$ is a win for
Black then so is $b$. We say that $a$ and $b$ are \emph{equivalent} if
$a\leq b$ and $b\leq a$. An \emph{outcome} for the region is an
equivalence class of its completions. The outcomes are partially
ordered via $\leq$, and we refer to this partial order as the
\emph{outcome poset} for the region.

The outcome poset for an arbitrary Hex region can in general be very
complicated. But in this paper, we are primarily interested in
3-terminal regions, and as mentioned in the introduction, there are
only five possible outcomes for such a region. Specifically, if the
region's black terminals are labelled $1$, $2$, and $3$, we name the
outcomes as follows:
\begin{itemize}
\item $\bot$: None of the black terminals are connected.
\item $a$: Terminals $2$ and $3$ are connected, but terminal $1$ is
  not.
\item $b$: Terminals $1$ and $3$ are connected, but terminal $2$ is
  not.
\item $c$: Terminals $1$ and $2$ are connected, but terminal $3$ is
  not.
\item $\top$: All of the black terminals are connected.
\end{itemize}
An example of each outcome is shown in
Figure~\ref{fig:3-terminal-outcomes}(a). Let $P_3 =
\s{\bot,a,b,c,\top}$ be this outcome poset, with the partial order
shown in Figure~\ref{fig:3-terminal-outcomes}(b).
\begin{figure}
  \[
  \def\scale{0.6}
  \raisebox{6em}{(a)}
  ~
  \begin{hexboard}[scale=\scale,baseline=(current bounding box.south)]
    \rotation{-30}
    \foreach\i in {3,...,5} {\hex(\i,1)}
    \foreach\i in {2,...,5} {\hex(\i,2)}
    \foreach\i in {1,...,4} {\hex(\i,3)}
    \foreach\i in {1,...,3} {\hex(\i,4)}
    \black(3,1)\sflabel{1}
    \white(4,1)
    \white(5,1)
    \white(5,2)
    \black(4,3)\sflabel{2}
    \white(3,4)
    \white(2,4)
    \black(1,4)\sflabel{3}
    \white(1,3)
    \white(2,2)
    \black(3,2)\black(4,2)\black(2,3)\white(3,3)
    \node at \coord(1,6) {$\top$};
    \node(center) at \coord(3,2.5) {};
  \end{hexboard}
  \quad
  \begin{hexboard}[scale=\scale,baseline=(current bounding box.south)]
    \rotation{-30}
    \foreach\i in {3,...,5} {\hex(\i,1)}
    \foreach\i in {2,...,5} {\hex(\i,2)}
    \foreach\i in {1,...,4} {\hex(\i,3)}
    \foreach\i in {1,...,3} {\hex(\i,4)}
    \black(3,1)\sflabel{1}
    \white(4,1)
    \white(5,1)
    \white(5,2)
    \black(4,3)\sflabel{2}
    \white(3,4)
    \white(2,4)
    \black(1,4)\sflabel{3}
    \white(1,3)
    \white(2,2)
    \white(3,2)\white(4,2)\black(2,3)\black(3,3)
    \node at \coord(1,6) {$a$};
    \node(center) at \coord(3,2.5) {};
  \end{hexboard}
  \quad
  \begin{hexboard}[scale=\scale,baseline=(current bounding box.south)]
    \rotation{-30}
    \foreach\i in {3,...,5} {\hex(\i,1)}
    \foreach\i in {2,...,5} {\hex(\i,2)}
    \foreach\i in {1,...,4} {\hex(\i,3)}
    \foreach\i in {1,...,3} {\hex(\i,4)}
    \black(3,1)\sflabel{1}
    \white(4,1)
    \white(5,1)
    \white(5,2)
    \black(4,3)\sflabel{2}
    \white(3,4)
    \white(2,4)
    \black(1,4)\sflabel{3}
    \white(1,3)
    \white(2,2)
    \black(3,2)\white(4,2)\black(2,3)\white(3,3)
    \node at \coord(1,6) {$b$};
    \node(center) at \coord(3,2.5) {};
  \end{hexboard}
  \quad
  \begin{hexboard}[scale=\scale,baseline=(current bounding box.south)]
    \rotation{-30}
    \foreach\i in {3,...,5} {\hex(\i,1)}
    \foreach\i in {2,...,5} {\hex(\i,2)}
    \foreach\i in {1,...,4} {\hex(\i,3)}
    \foreach\i in {1,...,3} {\hex(\i,4)}
    \black(3,1)\sflabel{1}
    \white(4,1)
    \white(5,1)
    \white(5,2)
    \black(4,3)\sflabel{2}
    \white(3,4)
    \white(2,4)
    \black(1,4)\sflabel{3}
    \white(1,3)
    \white(2,2)
    \black(3,2)\black(4,2)\white(2,3)\white(3,3)
    \node at \coord(1,6) {$c$};
    \node(center) at \coord(3,2.5) {};
  \end{hexboard}
  \quad
  \begin{hexboard}[scale=\scale,baseline=(current bounding box.south)]
    \rotation{-30}
    \foreach\i in {3,...,5} {\hex(\i,1)}
    \foreach\i in {2,...,5} {\hex(\i,2)}
    \foreach\i in {1,...,4} {\hex(\i,3)}
    \foreach\i in {1,...,3} {\hex(\i,4)}
    \black(3,1)\sflabel{1}
    \white(4,1)
    \white(5,1)
    \white(5,2)
    \black(4,3)\sflabel{2}
    \white(3,4)
    \white(2,4)
    \black(1,4)\sflabel{3}
    \white(1,3)
    \white(2,2)
    \white(3,2)\white(4,2)\white(2,3)\white(3,3)
    \node at \coord(1,6) {$\bot$};
    \node(center) at \coord(3,2.5) {};
  \end{hexboard}
  \hspace{0.8cm}
  \raisebox{6em}{(b)}
  \begin{tikzpicture}[xscale=0.45,yscale=0.49,baseline=(current bounding box.south)]
    \node(top) at (0,2) {$\top$};
    \node(a) at (-3,0) {$a$};
    \node(b) at (0,0) {$b$};
    \node(c) at (3,0) {$c$};
    \node(bot) at (0,-2) {$\bot$};
    \draw (bot) -- (a) -- (top);
    \draw (bot) -- (b) -- (top);
    \draw (bot) -- (c) -- (top);
  \end{tikzpicture}
  \]
  \caption{(a) The five possible outcomes of a 3-terminal region. (b)
    The outcome poset for 3-terminal regions.}
  \label{fig:3-terminal-outcomes}
\end{figure}
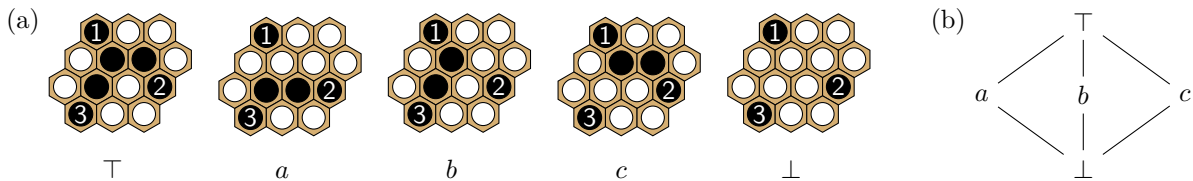

\subsection{Game forms}

An outcome represents the state of a region when it is completely
filled with stones, i.e., at the end of the game. To understand the
value of a region that still has some empty cells in it, we need the
concept of an abstract game form.

\begin{definition}[Game forms over an outcome poset {\cite[Def.~4.1]{S2022-hex-cgt}}]
  Fix an outcome poset $A$. The \emph{game forms over $A$} are
  defined as follows:
  \begin{itemize}
  \item Whenever $a\in A$, then $[a]$ is a game form, called an
    \emph{atomic} game.
  \item Whenever $L$ and $R$ are non-empty sets of game forms, $\g{L|R}$ is
    a game form, called a \emph{composite} game.
  \end{itemize}
  Moreover, the class of game forms is the smallest class generated by
  the above two rules.
\end{definition}

We often shorten ``game form'' to ``game'' when there is no potential
for confusion.  We use some customary notations from combinatorial
game theory. For an atomic game, we often write $a$ instead of $[a]$.
If $G=\g{L|R}$ is a composite game, the elements of $L$ and $R$ are
called the \emph{left} and \emph{right options} of $G$,
respectively. The represent the positions that the left, respectively
right, player can reach in the next move. Atomic games have no
options. If $G$ is any game, we write $G^L$ for a typical left option
and $G^R$ for a typical right option of $G$. Thus, when we write ``for
all $G^L$'', we mean ``for all left options $G^L$ of $G$''; such a
statement is vacuously true when $G$ is atomic. Similarly, ``there
exists $G^L$'' means ``there exists a left option $G^L$ of $G$'', and
such a statement is trivially false when $G$ is atomic. We also use
the notation $G^{(L)}$ to mean $G^L$ if $G$ is composite, and $G$
itself if $G$ is atomic, and similarly for $G^{(R)}$.

The \emph{followers} of a game $G$ are $G$ itself, all of its options,
the options of the options, and so on. A game is \emph{short} if it
has finitely many followers. In this paper, we only consider short
games, because all Hex positions are naturally short. The \emph{depth}
of a short game is defined in the obvious way: atomic games have depth
0, and the depth of a composite game is one more than the maximum of
the depths of its options.

Every Hex region can be converted to a game form over its outcome
poset. This may best be seen in an example. Consider the region in
Figure~\ref{fig:value-example}.
\begin{figure}
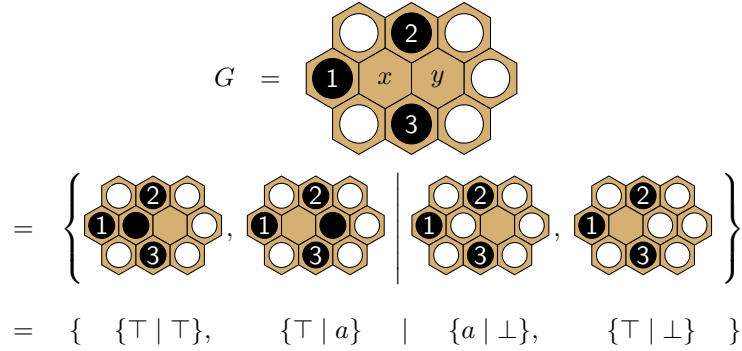

    \def\G#1{\lift{$\begin{hexboard}[scale=\scale]
    \rotation{-30}
    \foreach\i in {2,...,4} {\hex(\i,1)}
    \foreach\i in {1,...,4} {\hex(\i,2)}
    \foreach\i in {1,...,3} {\hex(\i,3)}
    \black(1,2)\sflabel{1}
    \black(3,1)\sflabel{2}
    \black(2,3)\sflabel{3}
    \white(1,3)
    \white(3,3)
    \white(4,2)
    \white(4,1)
    \white(2,1)
    #1
  \end{hexboard}$}}
  \[
  \def\lift{\mp{0.45}}
  \def\scale{1}
  G ~~=~~ \G{\cell(2,2)\label{$x$}\cell(3,2)\label{$y$}}
  \]
  \vspace{-0.25ex}
  \[
  \def\lift{\mp{0.4}}
  \def\scale{0.65}
  \def\xpace{\hspace{2.8ex}}
  \begin{array}{lll}
     &=&
    \left\{
    \G{\black(2,2)}\,,~
    \G{\black(3,2)}~\middle|~
    \G{\white(2,2)}\,,~
    \G{\white(3,2)}~\right\}
    \\\\ &=& ~
    \g{
      \xpace\g{\top|\top},\xpace
      \xpace\g{\top|a}\xpace|
      \xpace\g{a|\bot},\xpace
      \xpace\g{\top|\bot}\xpace
    }
  \end{array}
  \]
  \caption{The game form of a Hex position}
  \label{fig:value-example}
\end{figure}
We identify Black with Left and White with Right. The game $G$ has two
left options, corresponding to Black making a move at $x$ or $y$. The
game $G$ also has two right options, corresponding to White making a
move at $x$ or $y$. The resulting positions can be recursively
converted to game forms, with atomic positions acting as the base
case. For example, if Black occupies both $x$ and $y$, all three black
terminals are connected, so the outcome is $\top$. If Black occupies
$y$ and White occupies $x$, terminals 2 and 3 are connected but
terminal 1 is not, so the outcome is $a$, etc. Thus, by recursively
expanding $G$, we find that its game form is
$\g{\g{\top|\top},\g{\top|a}|\g{a|\bot},\g{\top|\bot}}$.

Note that, as always in combinatorial game theory, we do not enforce
alternating turns. This is because if Black makes a move in the
region, White may choose to make a move elsewhere on the board
(outside the region), and then Black may make another move in the
region. Thus, although the players alternate in the global game, it is
possible for a player to make consecutive moves within a region.

\subsection{The order relations}

To be able to compare game forms, we define two relations $\leq$
and $\tri$. Intuitively, $G\leq H$ means that $H$ is at least as good
for Black as $G$, whereas $G\tri H$ means that moving first in
$H$ is at least as good for Black as moving second in $G$.

\begin{definition}[Order relations {\cite[Def.~4.3]{S2022-hex-cgt}}]
  The relations $\leq$ and $\tri$ on game forms are defined by mutual
  recursion as follows:
  \begin{itemize}
  \item $G\leq H$ if both of the following hold:
    \begin{itemize}
    \item for all $G^{(L)}$, $G^{(L)}\tri H$, and
    \item for all $H^{(R)}$, $G\tri H^{(R)}$.
    \end{itemize}
  \item $G\tri H$ if at least one of the following holds:
    \begin{itemize}
    \item there exists $G^R$ with $G^R\leq H$, or
    \item there exists $H^L$ with $G\leq H^L$, or
    \item $G=[a]$ and $H=[b]$ are both atomic and $a\leq b$.
    \end{itemize}
  \end{itemize}
\end{definition}

We say that two game forms $G,H$ (over the same outcome poset) are
\emph{equivalent}, in symbols $G\eq H$, if $G\leq H$ and $H\leq G$.
Many useful properties of the order and of equivalence are proved in
{\cite{S2022-hex-cgt}}. Among them are several transitivity
properties {\cite[Lem.~4.6]{S2022-hex-cgt}}:
\begin{itemize}
  \item If $G\leq H$ and $H\leq K$, then $G\leq K$.
  \item If $G\tri H$ and $H\leq K$, then $G\tri K$.
  \item If $G\leq H$ and $H\tri K$, then $G\tri K$.
\end{itemize}

For short games, each equivalence class has a unique simplest member
called the game's \emph{canonical form}
{\cite[Lem.~4.23]{S2022-hex-cgt}}. Moreover, canonical forms can be
efficiently computed. From now on, by a \emph{value}, we mean an
equivalence class of game forms. Values are usually given in
canonical form.

\subsection{Monotone and passable games}

Among all of the game forms, those that are realizable as Hex
positions have an important additional property. In Hex, an additional
black stone can only improve Black's position, and dually, an
additional white stone can only improve White's position. This is the
property of \emph{monotonicity}, which we can formulate for
game forms as follows: for all $G^L$ and $G^R$, we have $G^R\leq G\leq
G^L$. Unfortunately, it turns out that the class of monotone 
game forms is not very robust; for example, the canonical form of a
monotone game need not be monotone (see
{\cite[Example~4.29]{S2022-hex-cgt}}). This problem is solved by
introducing the class of \emph{passable} games in the following
definition.

\begin{definition}[Monotone and passable games {\cite[Defs.~4.24 and 6.2]{S2022-hex-cgt}}]
  A game form $G$ is called \emph{monotone} if all left options
  satisfy $G\leq G^L$, all right options satisfy $G^R\leq G$, and
  recursively all options are monotone. A game form $G$ is called
  \emph{passable} if $G\tri G$ and recursively all options are passable.
\end{definition}

The concept of passable games is justified by the following
theorem.

\begin{theorem}[Fundamental theorem of monotone games {\cite[Cor.~6.6]{S2022-hex-cgt}}]
  \label{thm:fundamental}
  Let $A$ be an atom poset with top and bottom elements, and let $G$
  be a short game over $A$. Then $G$ is equivalent to a monotone game
  if and only if $G$ is equivalent to a passable game, if and only if
  the canonical form of $G$ is passable.
\end{theorem}

The following lemma is simple but useful.

\begin{lemma}\label{lem:topleqG}
  If $G$ is monotone and composite and all $G^R$ satisfy $\top\tri
  G^R$, then $\top\leq G$.
\end{lemma}

\begin{proof}
  To show $\top\leq G$, we only need to show two things: first, that
  all $G^R$ satisfy $\top\tri G^R$, which holds by assumption, and
  second that $\top\tri G$. Since $G$ is composite, it has some right
  option $G^R$, and since $G$ is monotone, we have $G^R\leq G$. By
  assumption, $\top\tri G^R\leq G$, which implies $\top\tri G$ by
  transitivity.
\end{proof}

\subsection{Disjunctive sums}

We will also need the notion of disjunctive sum of games. Informally,
the sum of two or more Hex positions is obtained by putting them side
by side and combining them into a larger position in a prescribed way.
For an example, consider equation~\ref{eqn:plusc} below, which
illustrates one particular way of combining two 3-terminal positions
into a single 3-terminal position.

Because there can be multiple different ways of summing two games, our
definition of the sum of abstract game forms is parameterized by a
function $f$ that determines how to combine atomic outcomes. For
example, Figure~\ref{fig:concat-atoms} illustrates several cases of
how the sum operation of equation~\ref{eqn:plusc} acts on atomic
outcomes. We have the following definition for the sum of game forms:

\begin{definition}[Sum]\label{def:sum}
  Let $A,B,C$ be posets and let $f:A\times B\to C$ be a monotone
  function, i.e., such that $a\leq a'$ and $b\leq b'$ implies
  $f(a,b)\leq f(a',b')$. Given a game $G$ over $A$ and a game $H$ over
  $B$, their sum $G\plusf H$ is a game over $C$, defined recursively as
  follows:
  \begin{itemize}
  \item $G\plusf  H = \g{G^L\plusf H, G\plusf H^L|G^R\plusf H, G\plusf
    H^R}$, when at least one of $G$ or $H$ is composite, and
  \item $[a]\plusf[b] = [f(a,b)]$.
  \end{itemize}
\end{definition}

Note that the situation is different from normal play games
{\cite{ONAG,WinningWays}}, where there is a single disjunctive sum
operation. In our setting, there are many alternative ways of summing
games, depending on the function $f$.

\subsection{Special 3-terminal regions: corners and forks}
\label{ssec:corners-and-forks}

As mentioned in Section~\ref{ssec:regions}, the outcome poset for a
general 3-terminal region is $P_3=\s{\bot,a,b,c,\top}$ with the
partial order shown in
Figure~\ref{fig:3-terminal-outcomes}(b). However, some special kinds
of 3-terminal regions have more specialized outcome posets. One way
in which this happens is when two or more of the region's terminals
are board edges. A 3-terminal region whose terminals include adjacent
black and white board edges is called a \emph{corner}, as in
Figure~\ref{fig:corner-fork}(a). A 3-terminal region whose terminals
include two opposite white board edges is called a \emph{fork}, as in
Figure~\ref{fig:corner-fork}(b). As before, outcome $a$ means that
terminals $2$ and $3$ are connected, outcome $b$ means that terminals
$1$ and $3$ are connected, and so on. But compared to a generic
3-terminal region, a corner has the additional property that $c\leq
a$. Indeed, the rest of the board is a 4-terminal region, and the
reader can verify that for each of the 14 possible outcomes for the
rest of the board, if $c$ is winning for Black in the corner, then so
is $a$.  For a fork, the situation is even more constrained: in this
case, outcome $c$ is equivalent to $\bot$, since in both cases,
White's edges are connected within the region and White wins
regardless of what happens on the rest of the board. The outcome
posets for a corner and a fork are shown in
Figure~\ref{fig:corner-fork}. Note that both of these posets are
quotients of $P_3$.
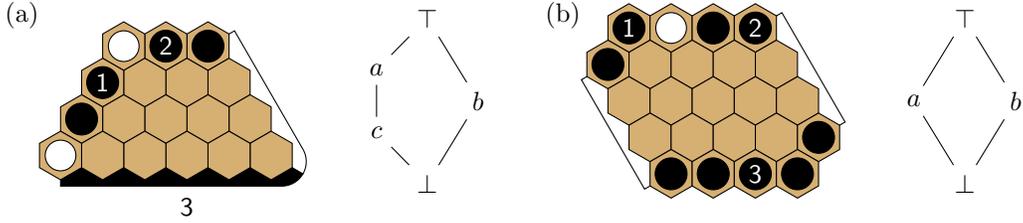
\begin{figure}
  \[
  \def\scale{0.8}
  \raisebox{3em}{(a)}
  \begin{hexboard}[scale=\scale,baseline=(center)]
    \rotation{-30}
    \foreach\i in {5,...,3} {\hex(\i,1)}
    \foreach\i in {5,...,2} {\hex(\i,2)}
    \foreach\i in {5,...,1} {\hex(\i,3)}
    \foreach\i in {5,...,0} {\hex(\i,4)}
    \edge[\noobtusecorner](0.5,4)(5,4)
    \edge[\noobtusecorner](5,1.5)(5,4)
    \white(0,4)
    \black(1,3)
    \black(2,2)\sflabel{1}
    \white(3,1)
    \black(4,1)\sflabel{2}
    \black(5,1)
    \cell(2.3,5.4)\sflabel{3}
    \node(center) at \coord(3,2.5) {};
  \end{hexboard}
  \qquad
  \begin{tikzpicture}[xscale=0.45,yscale=0.55,baseline=(current bounding box.center)]
    \node(top) at (0,2) {$\top$};
    \node(a) at (-1.5,-0.75) {$c$};
    \node(b) at (-1.5,0.75) {$a$};
    \node(c) at (1.5,0) {$b$};
    \node(bot) at (0,-2) {$\bot$};
    \draw (bot) -- (a) -- (b) -- (top);
    \draw (bot) -- (c) -- (top);
  \end{tikzpicture}
  \qquad
  \raisebox{3em}{(b)}
  \begin{hexboard}[scale=\scale,baseline=(center)]
    \rotation{-30}
    \foreach\i in {2,...,5} {\hex(\i,1)}
    \foreach\i in {1,...,5} {\hex(\i,2)}
    \foreach\i in {1,...,5} {\hex(\i,3)}
    \foreach\i in {1,...,5} {\hex(\i,4)}
    \foreach\i in {1,...,4} {\hex(\i,5)}
    \edge[\noobtusecorner\noacutecorner](5,1.5)(5,3.5)
    \edge[\noobtusecorner\noacutecorner](1,4.5)(1,2.5)
    \black(1,2)
    \black(2,1)\sflabel{1}
    \white(3,1)
    \black(4,1)
    \black(5,1)\sflabel{2}
    \black(1,5)
    \black(2,5)
    \black(3,5)\sflabel{3}
    \black(4,5)
    \black(5,4)
    \node(center) at \coord(3,3) {};
  \end{hexboard}
  \qquad
  \begin{tikzpicture}[xscale=0.45,yscale=0.55,baseline=(current bounding box.center)]
    \node(top) at (0,2) {$\top$};
    \node(a) at (-1.5,0) {$a$};
    \node(b) at (1.5,0) {$b$};
    \node(bot) at (0,-2) {$\bot$};
    \draw (bot) -- (a) -- (top);
    \draw (bot) -- (b) -- (top);
  \end{tikzpicture}
  \]
  \caption{(a) A corner and its outcome poset. (b) A fork and its
    outcome poset.}
  \label{fig:corner-fork}
\end{figure}

\section{Infinitely many non-equivalent 3-terminal positions}
\label{sec:superswitches}

Because there are only five possible outcomes for a 3-terminal
position, one may be tempted to think that there is not very much
going on in such a position. Maybe there are only finitely many
3-terminal positions up to equivalence? In this section, we show that
this is not the case. We will consider several infinite families of
3-terminal positions with interesting properties.

\subsection{Superswitches}

Consider an outcome poset with two incomparable atoms $a$ and $b$. It
was shown in {\cite[Prop.~10.2]{S2022-hex-cgt}} that there are infinitely
many passable game values over these outcomes. Specifically, it was
shown that the sequence of games defined by $G_0=a$ and $G_{n+1} =
\g{a,b|G_n}$ for all $n\geq 0$ is an infinite, strictly increasing
sequence of passable game values. The first few values in the sequence
are:
\[
\begin{array}{l}
  G_0 = a, \\
  G_1 = \g{a,b|a}, \\
  G_2 = \g{a,b|\g{a,b|a}}, \\
  G_3 = \g{a,b|\g{a,b|\g{a,b|a}}}, \\
  \ldots
\end{array}
\]
We call these games \emph{superswitches}. The idea is that the default
outcome is initially $a$, but the left player gets $n$ chances to
change (``switch'') the outcome to $b$. In other words, even if the
right player gets $n-1$ moves first, the left player can still choose
between outcomes $a$ and $b$.

The question was left open in {\cite{S2022-hex-cgt}} whether the
superswitches are realizable as Hex positions, and if so, whether they
are still distinct when regarded as Hex positions. We give positive
answers to these questions below.

As before, we assume that the black terminals of all 3-terminal
positions are numbered $1,2,3$. Recall that we write $a$ for the
outcome ``Black connects terminals 2 and 3'', $b$ for the outcome
``Black connects terminals 1 and 3'', and $c$ for the outcome ``Black
connects terminals 1 and 2''. Our goal is to realize the
superswitches as 3-terminal Hex positions.

We begin by considering the following operation on 3-terminal
positions. If $G$ and $H$ are 3-terminal positions, their
\emph{concatenation}, written $G\plusc H$, is the 3-terminal position
shown schematically in the following diagram:
\begin{equation}\label{eqn:plusc}
  G\plusc H \quad = \quad
  \begin{tikzpicture}[scale=0.75,baseline={(a.south)}]
    \draw[fill=black!10] (0,0) rectangle node{$G$} (3,2);
    \draw[fill=black!10] (4,0) rectangle node{$H$} (7,2);
    \draw[terminal] (3,1) node[left]{$2$} -- (4,1) node[right]{$3$};
    \draw[terminal] (-1,1) node[left]{$3$} -- (0,1) node[right]{$3$};
    \draw[terminal] (7,1) node[left]{$2$} -- (8,1) node[right]{$2$};
    \draw[terminal] (1.5,2) node[below]{$1$} -- (1.5,2.5) -- (5.5,2.5)
    -- (5.5,2) node[below]{$1$};
    \draw[terminal] (3.5,2.5) -- (3.5,3) node[above]{$1$};
    \path (0,-0.3);
    \node (a) at (0,1) {};
  \end{tikzpicture}
\end{equation}
In words, $G\plusc H$ is obtained by connecting terminal 1 of $G$ to
terminal 1 of $H$ and terminal 2 of $G$ to terminal 3 of
$H$. Terminals 1, 2, and 3 of the combined position are terminals 1 of
$G$ and $H$, terminal 2 of $H$, and terminal 3 of $G$, respectively.
As illustrated in Figure~\ref{fig:concat-atoms},
the atoms $a$ and $b$ satisfy the following identities:
\[
a\plusc a = a, \quad
a\plusc b = b, \quad
b\plusc a = b, \quad
b\plusc b = b.
\]
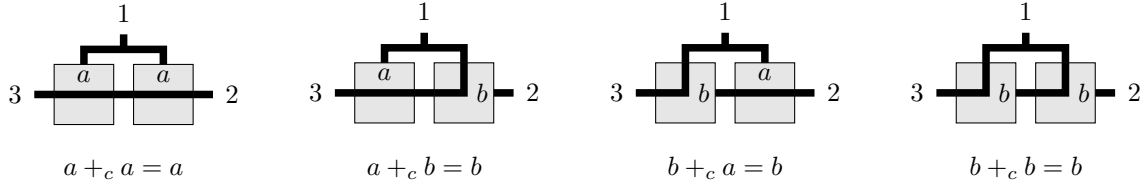
\begin{figure}
  \[
  \begin{tikzpicture}[scale=0.4,xscale=0.66]
    \draw[fill=black!10] (0,0) rectangle (3,2);
    \draw[fill=black!10] (4,0) rectangle (7,2);
    \draw[terminal] (0,1) -- node[above]{$a$} (3,1);
    \draw[terminal] (4,1) -- node[above]{$a$} (7,1);
    \draw[terminal] (3,1) -- (4,1);
    \draw[terminal] (-1,1) node[left]{$3$} -- (0,1);
    \draw[terminal] (7,1) -- (8,1) node[right]{$2$};
    \draw[terminal] (1.5,2) -- (1.5,2.5) -- (5.5,2.5) -- (5.5,2);
    \draw[terminal] (3.5,2.5) -- (3.5,3) node[above]{$1$};
    \path (3.5,-1.5) node{$a\plusc a = a$};
  \end{tikzpicture}
  \qquad
  \begin{tikzpicture}[scale=0.4,xscale=0.66]
    \draw[fill=black!10] (0,0) rectangle (3,2);
    \draw[fill=black!10] (4,0) rectangle (7,2);
    \draw[terminal] (0,1) -- node[above]{$a$} (3,1);
    \draw[terminal] (4,1) -- (5.5,1) node[right]{$b$} -- (5.5,2);
    \draw[terminal] (3,1) -- (4,1);
    \draw[terminal] (-1,1) node[left]{$3$} -- (0,1);
    \draw[terminal] (7,1) -- (8,1) node[right]{$2$};
    \draw[terminal] (1.5,2) -- (1.5,2.5) -- (5.5,2.5) -- (5.5,2);
    \draw[terminal] (3.5,2.5) -- (3.5,3) node[above]{$1$};
    \path (3.5,-1.5) node{$a\plusc b = b$};
  \end{tikzpicture}
  \qquad
  \begin{tikzpicture}[scale=0.4,xscale=0.66]
    \draw[fill=black!10] (0,0) rectangle (3,2);
    \draw[fill=black!10] (4,0) rectangle (7,2);
    \draw[terminal] (0,1) -- (1.5,1) node[right]{$b$} -- (1.5,2);
    \draw[terminal] (4,1) -- node[above]{$a$} (7,1);
    \draw[terminal] (3,1) -- (4,1);
    \draw[terminal] (-1,1) node[left]{$3$} -- (0,1);
    \draw[terminal] (7,1) -- (8,1) node[right]{$2$};
    \draw[terminal] (1.5,2) -- (1.5,2.5) -- (5.5,2.5) -- (5.5,2);
    \draw[terminal] (3.5,2.5) -- (3.5,3) node[above]{$1$};
    \path (3.5,-1.5) node{$b\plusc a = b$};
  \end{tikzpicture}
  \qquad
  \begin{tikzpicture}[scale=0.4,xscale=0.66]
    \draw[fill=black!10] (0,0) rectangle (3,2);
    \draw[fill=black!10] (4,0) rectangle (7,2);
    \draw[terminal] (0,1) -- (1.5,1) node[right]{$b$} -- (1.5,2);
    \draw[terminal] (4,1) -- (5.5,1) node[right]{$b$} -- (5.5,2);
    \draw[terminal] (3,1) -- (4,1);
    \draw[terminal] (-1,1) node[left]{$3$} -- (0,1);
    \draw[terminal] (7,1) -- (8,1) node[right]{$2$};
    \draw[terminal] (1.5,2) -- (1.5,2.5) -- (5.5,2.5) -- (5.5,2);
    \draw[terminal] (3.5,2.5) -- (3.5,3) node[above]{$1$};
    \path (3.5,-1.5) node{$b\plusc b = b$};
  \end{tikzpicture}
  \]
  \caption{Some identities for the concatenation of atomic positions}
  \label{fig:concat-atoms}
\end{figure}

\begin{lemma}\label{lem:G-n+1}
  For all $n$, we have $G_n \plusc G_1 \eq G_{n+1}$.
\end{lemma}

\begin{proof}
  First note that for all games $G$ involving only the atoms $a$ and
  $b$, we have $a\plusc G\,\eq\,G\,\eq\, G\plusc a$ and $b\plusc
  G\,\eq\, b\,\eq\, G\plusc b$. This is easily shown by induction on
  $G$.

  We now prove the lemma by induction. The base case holds because
  $G_0\plusc G_1 = a\plusc G_1 \eq G_1$. Now consider some
  $n>0$. Recall that, by definition, we have $G_n =
  \g{a,b|G_{n-1}}$. Also, by the induction hypothesis, we have
  $G_{n-1}\plusc G_1\eq G_n$. Then
  \[
  \begin{array}{rcl}
    G_n \plusc G_1
    &=& \g{G_n^L\plusc G_1, G_n\plusc G_1^L|G_n^R\plusc G_1, G_n\plusc G_1^R}\\
    &\eq& \g{a\plusc G_1, b\plusc G_1, G_n\plusc a, G_n\plusc b|G_{n-1}\plusc G_1, G_n\plusc a}\\
    &\eq& \g{G_1,b,G_n,b|G_n, G_n}.
  \end{array}
  \]
  From {\cite[Prop.~10.2]{S2022-hex-cgt}}, we know that $G_1\leq G_n$, and
  therefore the left option $G_1$ is dominated; therefore, the game
  $G_n\plusc G_1$ simplifies to $\g{G_n,b|G_n}$. From $a\leq G_n$, we
  get $\g{a,b|G_n}\leq\g{G_n,b|G_n}$. On the other hand, from $G_n\leq
  G_{n+1}$, since the games are passable, we get $G_n\tri
  G_{n+1}=\g{a,b|G_n}$, which implies $\g{G_n,b|G_n}\leq \g{a,b|G_n}$.
  It follows that $G_n\plusc G_1\eq\g{G_n,b|G_n}\eq\g{a,b|G_n}\eq
  G_{n+1}$, as claimed.
\end{proof}

\subsection{Superswitches are Hex realizable}
\label{ssec:superswitch-realizable}

\begin{lemma}\label{lem:g1-realizable}
  The following 3-terminal position has value $G_1=\g{a,b|a}$.
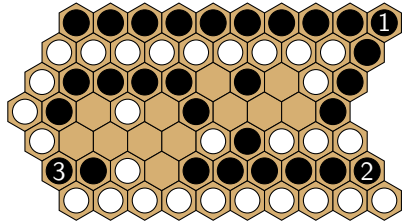
  \[
  \begin{hexboard}[scale=0.65]
    \rotation{-30}
    \hex(0,2)\white(0,2)
    \hex(0,3)\white(0,3)
    \hex(0,4)\black(0,4)\sflabel{3}
    \hex(0,5)\white(0,5)
    \hex(1,1)\white(1,1)
    \hex(1,2)\black(1,2)
    \hex(1,3)
    \hex(1,4)\black(1,4)
    \hex(1,5)\white(1,5)
    \hex(2,0)\white(2,0)
    \hex(2,1)\black(2,1)
    \hex(2,2)
    \hex(2,3)
    \hex(2,4)\white(2,4)
    \hex(2,5)\white(2,5)
    \hex(3,0)\white(3,0)
    \hex(3,1)\black(3,1)
    \hex(3,2)\white(3,2)
    \hex(3,3)
    \hex(3,4)
    \hex(3,5)\white(3,5)
    \hex(4,0)\white(4,0)
    \hex(4,1)\black(4,1)
    \hex(4,2)
    \hex(4,3)
    \hex(4,4)\black(4,4)
    \hex(4,5)\white(4,5)
    \hex(5,0)\white(5,0)
    \hex(5,1)\black(5,1)
    \hex(5,2)\black(5,2)
    \hex(5,3)\white(5,3)
    \hex(5,4)\black(5,4)
    \hex(5,5)\white(5,5)
    \hex(6,0)\white(6,0)
    \hex(6,1)
    \hex(6,2)
    \hex(6,3)\black(6,3)
    \hex(6,4)\black(6,4)
    \hex(6,5)\white(6,5)
    \hex(7,0)\white(7,0)
    \hex(7,1)\black(7,1)
    \hex(7,2)
    \hex(7,3)\white(7,3)
    \hex(7,4)\black(7,4)
    \hex(7,5)\white(7,5)
    \hex(8,0)\white(8,0)
    \hex(8,1)
    \hex(8,2)
    \hex(8,3)\white(8,3)
    \hex(8,4)\black(8,4)
    \hex(8,5)\white(8,5)
    \hex(9,0)\white(9,0)
    \hex(9,1)\white(9,1)
    \hex(9,2)\black(9,2)
    \hex(9,3)\white(9,3)
    \hex(9,4)\black(9,4)\sflabel{2}
    \hex(9,5)\white(9,5)
    \hex(10,0)\white(10,0)
    \hex(10,1)\black(10,1)
    \hex(11,0)\black(11,0)
    \foreach\i in {3,...,11} {\hex(\i,-1)\black(\i,-1)}
    \hex(12,-1)\black(12,-1)\sflabel{1}
  \end{hexboard}
  \]
\end{lemma}

\begin{proof}
  This can be checked by direct computation, by converting the
  position to a game form and then computing its canonical form. We
  will say more about how we actually found this position in
  Section~\ref{sec:database}.
\end{proof}

\begin{corollary}
  Each value in the infinite sequence $G_0,G_1,\ldots$ is realizable
  as a 3-terminal Hex position.  
\end{corollary}

\begin{proof}
  $G_0=a$ is realizable since it is atomic, and $G_1$ is realizable by
  Lemma~\ref{lem:g1-realizable}. It is also clear that if $G$ and $H$ are
  Hex realizable, then so is $G\plusc H$, because all we have to do is
  combine $G$ and $H$ in the way shown in diagram {\eqref{eqn:plusc}}.
  By repeated applications of Lemma~\ref{lem:G-n+1}, we have $G_n \eq
  G_1\plusc G_1\plusc\ldots\plusc G_1$ when $n\geq 1$. Therefore $G_n$
  is Hex realizable for all $n$. For example, the following is a Hex
  realization of $G_3\eq G_1\plusc G_1\plusc G_1$. Notice that it is
  just 3 copies of $G_1$ concatenated together.
  
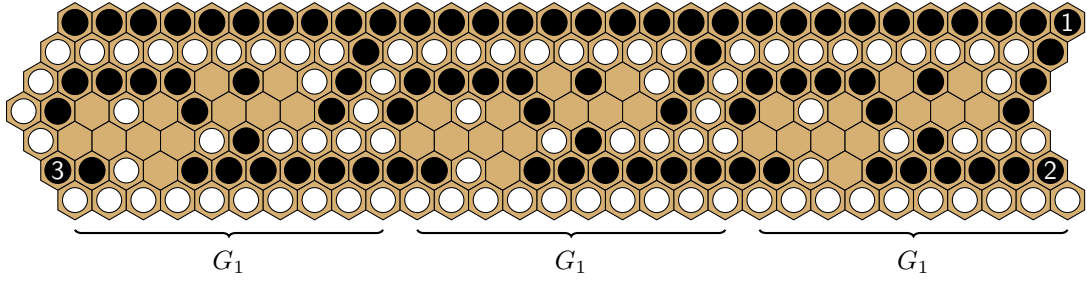
\[
  \begin{hexboard}[scale=0.65]
    \rotation{-30}
    \hex(0,2)\white(0,2)
    \hex(0,3)\white(0,3)
    \hex(0,4)\black(0,4)\sflabel{3}
    \hex(0,5)\white(0,5)
    \hex(1,1)\white(1,1)
    \hex(1,2)\black(1,2)
    \hex(1,3)
    \hex(1,4)\black(1,4)
    \hex(1,5)\white(1,5)
    \hex(2,0)\white(2,0)
    \hex(2,1)\black(2,1)
    \hex(2,2)
    \hex(2,3)
    \hex(2,4)\white(2,4)
    \hex(2,5)\white(2,5)
    \hex(3,0)\white(3,0)
    \hex(3,1)\black(3,1)
    \hex(3,2)\white(3,2)
    \hex(3,3)
    \hex(3,4)
    \hex(3,5)\white(3,5)
    \hex(4,0)\white(4,0)
    \hex(4,1)\black(4,1)
    \hex(4,2)
    \hex(4,3)
    \hex(4,4)\black(4,4)
    \hex(4,5)\white(4,5)
    \hex(5,0)\white(5,0)
    \hex(5,1)\black(5,1)
    \hex(5,2)\black(5,2)
    \hex(5,3)\white(5,3)
    \hex(5,4)\black(5,4)
    \hex(5,5)\white(5,5)
    \hex(6,0)\white(6,0)
    \hex(6,1)
    \hex(6,2)
    \hex(6,3)\black(6,3)
    \hex(6,4)\black(6,4)
    \hex(6,5)\white(6,5)
    \hex(7,0)\white(7,0)
    \hex(7,1)\black(7,1)
    \hex(7,2)
    \hex(7,3)\white(7,3)
    \hex(7,4)\black(7,4)
    \hex(7,5)\white(7,5)
    \hex(8,0)\white(8,0)
    \hex(8,1)
    \hex(8,2)
    \hex(8,3)\white(8,3)
    \hex(8,4)\black(8,4)
    \hex(8,5)\white(8,5)
    \hex(9,0)\white(9,0)
    \hex(9,1)\white(9,1)
    \hex(9,2)\black(9,2)
    \hex(9,3)\white(9,3)
    \hex(9,4)\black(9,4)
    \hex(9,5)\white(9,5)
    \hex(10,0)\white(10,0)
    \hex(10,1)\black(10,1)
    \hex(11,0)\black(11,0)

    \hex(10,2)\white(10,2)
    \hex(10,3)\white(10,3)
    \hex(10,4)\black(10,4)
    \hex(10,5)\white(10,5)
    \hex(11,1)\white(11,1)
    \hex(11,2)\black(11,2)
    \hex(11,3)
    \hex(11,4)\black(11,4)
    \hex(11,5)\white(11,5)
    \hex(12,0)\white(12,0)
    \hex(12,1)\black(12,1)
    \hex(12,2)
    \hex(12,3)
    \hex(12,4)\white(12,4)
    \hex(12,5)\white(12,5)
    \hex(13,0)\white(13,0)
    \hex(13,1)\black(13,1)
    \hex(13,2)\white(13,2)
    \hex(13,3)
    \hex(13,4)
    \hex(13,5)\white(13,5)
    \hex(14,0)\white(14,0)
    \hex(14,1)\black(14,1)
    \hex(14,2)
    \hex(14,3)
    \hex(14,4)\black(14,4)
    \hex(14,5)\white(14,5)
    \hex(15,0)\white(15,0)
    \hex(15,1)\black(15,1)
    \hex(15,2)\black(15,2)
    \hex(15,3)\white(15,3)
    \hex(15,4)\black(15,4)
    \hex(15,5)\white(15,5)
    \hex(16,0)\white(16,0)
    \hex(16,1)
    \hex(16,2)
    \hex(16,3)\black(16,3)
    \hex(16,4)\black(16,4)
    \hex(16,5)\white(16,5)
    \hex(17,0)\white(17,0)
    \hex(17,1)\black(17,1)
    \hex(17,2)
    \hex(17,3)\white(17,3)
    \hex(17,4)\black(17,4)
    \hex(17,5)\white(17,5)
    \hex(18,0)\white(18,0)
    \hex(18,1)
    \hex(18,2)
    \hex(18,3)\white(18,3)
    \hex(18,4)\black(18,4)
    \hex(18,5)\white(18,5)
    \hex(19,0)\white(19,0)
    \hex(19,1)\white(19,1)
    \hex(19,2)\black(19,2)
    \hex(19,3)\white(19,3)
    \hex(19,4)\black(19,4)
    \hex(19,5)\white(19,5)
    \hex(20,0)\white(20,0)
    \hex(20,1)\black(20,1)
    \hex(21,0)\black(21,0)

    \hex(20,2)\white(20,2)
    \hex(20,3)\white(20,3)
    \hex(20,4)\black(20,4)
    \hex(20,5)\white(20,5)
    \hex(21,1)\white(21,1)
    \hex(21,2)\black(21,2)
    \hex(21,3)
    \hex(21,4)\black(21,4)
    \hex(21,5)\white(21,5)
    \hex(22,0)\white(22,0)
    \hex(22,1)\black(22,1)
    \hex(22,2)
    \hex(22,3)
    \hex(22,4)\white(22,4)
    \hex(22,5)\white(22,5)
    \hex(23,0)\white(23,0)
    \hex(23,1)\black(23,1)
    \hex(23,2)\white(23,2)
    \hex(23,3)
    \hex(23,4)
    \hex(23,5)\white(23,5)
    \hex(24,0)\white(24,0)
    \hex(24,1)\black(24,1)
    \hex(24,2)
    \hex(24,3)
    \hex(24,4)\black(24,4)
    \hex(24,5)\white(24,5)
    \hex(25,0)\white(25,0)
    \hex(25,1)\black(25,1)
    \hex(25,2)\black(25,2)
    \hex(25,3)\white(25,3)
    \hex(25,4)\black(25,4)
    \hex(25,5)\white(25,5)
    \hex(26,0)\white(26,0)
    \hex(26,1)
    \hex(26,2)
    \hex(26,3)\black(26,3)
    \hex(26,4)\black(26,4)
    \hex(26,5)\white(26,5)
    \hex(27,0)\white(27,0)
    \hex(27,1)\black(27,1)
    \hex(27,2)
    \hex(27,3)\white(27,3)
    \hex(27,4)\black(27,4)
    \hex(27,5)\white(27,5)
    \hex(28,0)\white(28,0)
    \hex(28,1)
    \hex(28,2)
    \hex(28,3)\white(28,3)
    \hex(28,4)\black(28,4)
    \hex(28,5)\white(28,5)
    \hex(29,0)\white(29,0)
    \hex(29,1)\white(29,1)
    \hex(29,2)\black(29,2)
    \hex(29,3)\white(29,3)
    \hex(29,4)\black(29,4)\sflabel{2}
    \hex(29,5)\white(29,5)
    \hex(30,0)\white(30,0)
    \hex(30,1)\black(30,1)
    \hex(31,0)\black(31,0)

    \foreach\i in {3,...,31} {\hex(\i,-1)\black(\i,-1)}
    \hex(32,-1)\black(32,-1)\sflabel{1}

    \draw[thick,decoration={brace,mirror},decorate] \coord(-0.5,6) -- node[below=1ex]{$G_1$} \coord(8.5,6);
    \draw[thick,decoration={brace,mirror},decorate] \coord(9.5,6) -- node[below=1ex]{$G_1$} \coord(18.5,6);
    \draw[thick,decoration={brace,mirror},decorate] \coord(19.5,6) -- node[below=1ex]{$G_1$} \coord(28.5,6);
  \end{hexboard}
  \vspace{-5ex}
  \]
\end{proof}

\subsection{Superswitches are Hex distinguishable}
\label{ssec:distinguish}

Now that we have found Hex realizations of superswitches as 3-terminal
positions, the question remains whether they are distinct as Hex
positions. Although it was shown in {\cite[Prop.~10.2]{S2022-hex-cgt}} that
the superswitches $G_0,G_1,\ldots$ are distinct \emph{as abstract game
values}, it does not a priori follow that they are distinct in Hex,
because to establish the latter, we need to find a
\emph{Hex realizable} context that distinguishes any pair of them.
Fortunately, the superswitches themselves can be repurposed to provide
such contexts.

First, let us consider the \emph{dual superswitches}, defined by
$G\opp_0=a$ and $G\opp_{n+1}=\g{G\opp_n|a,b}$.  They work exactly like
the superswitches, except that the roles of left and right are
reversed. In other words, it is now the right player who gets $n$
chances to switch the outcome from $a$ to $b$. We can obtain Hex
realizations of the dual superswitches by taking the Hex realizations
of the corresponding superswitches, exchanging the roles of Black and
White, and renumbering the terminals. For good measure, we also flip
the position upside down. For example, the following is a realization
of the dual superswitch $G\opp_1$.
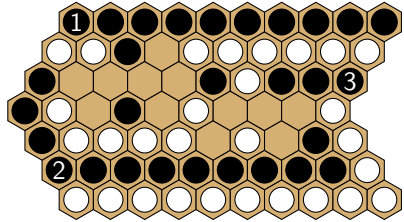
\[
\begin{hexboard}[scale=0.65,yscale=-1]
  \rotation{-30}
  \hex(0,2)\black(0,2)
  \hex(0,3)\black(0,3)
  \hex(0,4)\white(0,4)
  \hex(0,5)\black(0,5)\sflabel{1}
  \hex(1,1)\black(1,1)
  \hex(1,2)\white(1,2)
  \hex(1,3)
  \hex(1,4)\white(1,4)
  \hex(1,5)\black(1,5)
  \hex(2,0)\black(2,0)\sflabel{2}
  \hex(2,1)\white(2,1)
  \hex(2,2)
  \hex(2,3)
  \hex(2,4)\black(2,4)
  \hex(2,5)\black(2,5)
  \hex(3,0)\black(3,0)
  \hex(3,1)\white(3,1)
  \hex(3,2)\black(3,2)
  \hex(3,3)
  \hex(3,4)
  \hex(3,5)\black(3,5)
  \hex(4,0)\black(4,0)
  \hex(4,1)\white(4,1)
  \hex(4,2)
  \hex(4,3)
  \hex(4,4)\white(4,4)
  \hex(4,5)\black(4,5)
  \hex(5,0)\black(5,0)
  \hex(5,1)\white(5,1)
  \hex(5,2)\white(5,2)
  \hex(5,3)\black(5,3)
  \hex(5,4)\white(5,4)
  \hex(5,5)\black(5,5)
  \hex(6,0)\black(6,0)
  \hex(6,1)
  \hex(6,2)
  \hex(6,3)\white(6,3)
  \hex(6,4)\white(6,4)
  \hex(6,5)\black(6,5)
  \hex(7,0)\black(7,0)
  \hex(7,1)\white(7,1)
  \hex(7,2)
  \hex(7,3)\black(7,3)
  \hex(7,4)\white(7,4)
  \hex(7,5)\black(7,5)
  \hex(8,0)\black(8,0)
  \hex(8,1)
  \hex(8,2)
  \hex(8,3)\black(8,3)
  \hex(8,4)\white(8,4)
  \hex(8,5)\black(8,5)
  \hex(9,0)\black(9,0)
  \hex(9,1)\black(9,1)
  \hex(9,2)\white(9,2)
  \hex(9,3)\black(9,3)\sflabel{3}
  \hex(9,4)\white(9,4)
  \hex(9,5)\black(9,5)
  \hex(10,0)\black(10,0)
  \hex(10,1)\white(10,1)
  \hex(11,0)\white(11,0)
  \foreach\i in {3,...,11} {\hex(\i,-1)\white(\i,-1)}
  \hex(12,-1)\white(12,-1)
\end{hexboard}
\]
As usual, the other dual superswitches $G\opp_n$ are obtained by
concatenation.

Next, we consider the following operation: If $G,H$ are 3-terminal
positions, their \emph{juxtaposition}, written $G\plusj H$, is the
2-terminal position shown schematically in the following diagram:
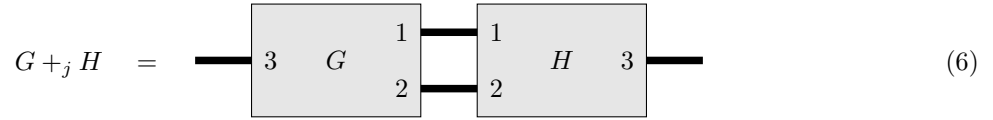
\begin{equation}\label{eqn:plusj}
  G\plusj H\quad=\quad
  \begin{tikzpicture}[scale=0.75,baseline={(a.south)}]
    \draw[fill=black!10] (0,0) rectangle node{$G$} (3,2);
    \draw[fill=black!10] (4,0) rectangle node{$H$} (7,2);
    \draw[terminal] (3,0.5) node[left]{$2$} -- (4,0.5) node[right]{$2$};
    \draw[terminal] (3,1.5) node[left]{$1$} -- (4,1.5) node[right]{$1$};
    \draw[terminal] (-1,1) -- (0,1) node[right]{$3$};
    \draw[terminal] (7,1) node[left]{$3$} -- (8,1);
    \node (a) at (0,1) {};
  \end{tikzpicture}
\end{equation}
Since terminal 1 of $G$ is connected to terminal 1 of $H$ and terminal
2 of $G$ is connected to terminal 2 of $H$, it is immediately obvious
that the atoms $a$ and $b$ satisfy the following identities:
\[
a\plusj a = \top, \quad
a\plusj b = \bot, \quad
b\plusj a = \bot, \quad
b\plusj b = \top.
\]
The following proposition shows what happens when we juxtapose a
superswitch with a dual superswitch. Recall that $\Star=\g{\top|\bot}$
is the value of a game that is a first-player win.
\begin{proposition}\label{prop:juxta}
  We have
  \[
  G_n \plusj G\opp_m ~\eq\,
  \begin{mychoices}
    \top & \mbox{if $n>m-1$}, \\
    \Star & \mbox{if $n=m-1$}, \\
    \bot & \mbox{if $n<m-1$}. \\
  \end{mychoices}
  \]
\end{proposition}

Before we prove the theorem, we give an intuitive explanation: the
left player owns the left switch and the right player owns the right
switch. Neither player wants to commit to a setting for their switch,
or else the other player will set their own switch accordingly and win
the game. Therefore, Left plays in the right switch as long as
possible and Right plays in the left switch as long as
possible. Whoever first finishes this ``race to the bottom'' wins,
except that the left player has a small advantage since the
``default'' setting for both switches is $a$, which gives a left
player win.

\begin{proof}
  To ease the notation in this proof, we define the \emph{sign games}
  by
  \[
  S_n =
  \begin{mychoices}
    \top & \mbox{if $n>0$,} \\
    \Star & \mbox{if $n=0$,} \\
    \bot & \mbox{if $n<0$.}
  \end{mychoices}
  \]
  The claim of the proposition can then be stated as $G_n\plusj
  G\opp_m\eq S_{n-m+1}$. We note that for all integers $n$, we have
  $S_n \eq \g{S_{n+1}|S_{n-1}}$.
  
  We start by computing $G_n\plusj a$ and $G_n\plusj b$. Recall that
  the sequence $G_0,G_1,\ldots$ is increasing. We have $G_0\plusj a =
  a\plusj a = \top$, and since $G_0\leq G_n$ for all $n$, it follows
  that $G_n\plusj a \eq \top$ as well. We have $G_0\plusj b = a\plusj
  b = \bot$, $G_1\plusj b = \g{a,b|a}\plusj b = \g{\bot,\top|\bot} \eq
  \Star$, and $G_2\plusj b = \g{a,b|G_1}\plusj b\eq\g{\bot,\top|\Star}
  \eq \top$. Therefore $G_n\plusj b \eq \top$ holds for all $n\geq 2$.
  In other words, $G_n\plusj b\eq S_{n-1}$.  By a dual computation, we
  find that $a\plusj G\opp_m\eq S_{1-m}$ and $b\plusj G\opp_m \eq
  \bot$ for all $m\geq 0$.

  We prove the proposition by induction on $n+m$. The base case for
  $m=0$ holds because we just calculated that $G_n\plusj G\opp_0 =
  G_n\plusj a \eq \top = S_{n+1}$ for all $n\geq 0$; similarly, the
  base case for $n=0$ holds because $G_0\plusj G\opp_m = a\plusj
  G\opp_m\eq S_{1-m}$ for all $m\geq 0$. Now consider the case
  $n,m>0$. We have:
  \[
  \begin{array}{rcl}
    G_n \plusj G\opp_m
    &=& \g{G_n^L\plusj G\opp_m, G_n\plusj {G\opp_m}^L
      |    G_n^R\plusj G\opp_m, G_n\plusj {G\opp_m}^R}\\
    &=& \g{a\plusj G\opp_m, b\plusj G\opp_m, G_n\plusj G\opp_{m-1}
      |    G_{n-1}\plusj G\opp_m, G_n\plusj a, G_n\plusj b}\\
    &\eq& \g{S_{1-m}, \bot, S_{n-m+2}
      |    S_{n-m}, \top, S_{n-1}}.
  \end{array}
  \]
  In the last step, we have used the above calculations as well as the
  induction hypothesis. Since $1-m<n-m+2$, we have $S_{1-m}\leq
  S_{n-m+2}$, and therefore $S_{n-m+2}$ dominates the other left
  options; similarly, $S_{n-m}$ dominates all other right
  options. Therefore $G_n \eq \g{S_{n-m+2}|S_{n-m}} \eq S_{n-m+1}$ as
  claimed.
\end{proof}

\begin{corollary}
  There are infinitely many Hex-distinguishable 3-terminal positions.
\end{corollary}

\begin{proof}
  Consider Hex realizations of the superswitches $G_n$ and $G_m$,
  where $n<m$. To show that they are Hex-distinct, it suffices to find
  a Hex context in which $G_n$ is a first player win for Black and
  $G_m$ is a first player loss for Black. By
  Proposition~\ref{prop:juxta}, $(-)\plusj G\opp_{n-1}$ is such a
  context.
\end{proof}

For example, the following shows the 2-terminal position $G_2\plusj
G\opp_1$, which has value $\Star$ by Proposition~\ref{prop:juxta}.

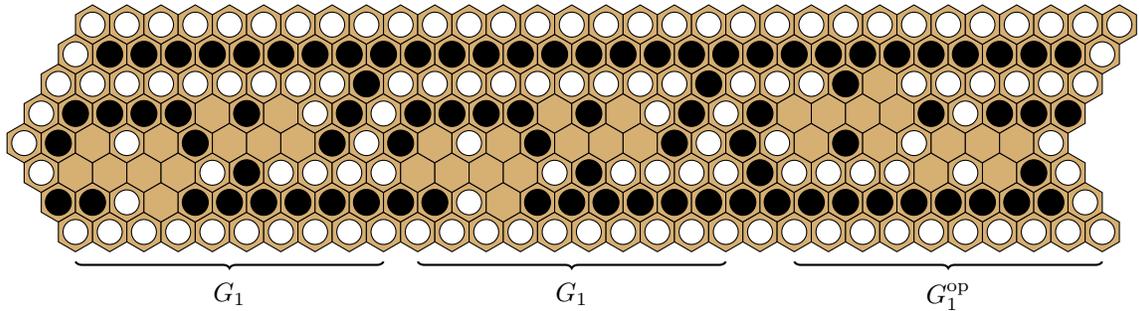
\[
\begin{hexboard}[scale=0.65]
  \rotation{-30}
  \hex(0,2)\white(0,2)
  \hex(0,3)\white(0,3)
  \hex(0,4)\black(0,4)
  \hex(0,5)\white(0,5)
  \hex(1,1)\white(1,1)
  \hex(1,2)\black(1,2)
  \hex(1,3)
  \hex(1,4)\black(1,4)
  \hex(1,5)\white(1,5)
  \hex(2,0)\white(2,0)
  \hex(2,1)\black(2,1)
  \hex(2,2)
  \hex(2,3)
  \hex(2,4)\white(2,4)
  \hex(2,5)\white(2,5)
  \hex(3,0)\white(3,0)
  \hex(3,1)\black(3,1)
  \hex(3,2)\white(3,2)
  \hex(3,3)
  \hex(3,4)
  \hex(3,5)\white(3,5)
  \hex(4,0)\white(4,0)
  \hex(4,1)\black(4,1)
  \hex(4,2)
  \hex(4,3)
  \hex(4,4)\black(4,4)
  \hex(4,5)\white(4,5)
  \hex(5,0)\white(5,0)
  \hex(5,1)\black(5,1)
  \hex(5,2)\black(5,2)
  \hex(5,3)\white(5,3)
  \hex(5,4)\black(5,4)
  \hex(5,5)\white(5,5)
  \hex(6,0)\white(6,0)
  \hex(6,1)
  \hex(6,2)
  \hex(6,3)\black(6,3)
  \hex(6,4)\black(6,4)
  \hex(6,5)\white(6,5)
  \hex(7,0)\white(7,0)
  \hex(7,1)\black(7,1)
  \hex(7,2)
  \hex(7,3)\white(7,3)
  \hex(7,4)\black(7,4)
  \hex(7,5)\white(7,5)
  \hex(8,0)\white(8,0)
  \hex(8,1)
  \hex(8,2)
  \hex(8,3)\white(8,3)
  \hex(8,4)\black(8,4)
  \hex(8,5)\white(8,5)
  \hex(9,0)\white(9,0)
  \hex(9,1)\white(9,1)
  \hex(9,2)\black(9,2)
  \hex(9,3)\white(9,3)
  \hex(9,4)\black(9,4)
  \hex(9,5)\white(9,5)
  \hex(10,0)\white(10,0)
  \hex(10,1)\black(10,1)
  \hex(11,0)\black(11,0)
  \hex(10,2)\white(10,2)
  \hex(10,3)\white(10,3)
  \hex(10,4)\black(10,4)
  \hex(10,5)\white(10,5)
  \hex(11,1)\white(11,1)
  \hex(11,2)\black(11,2)
  \hex(11,3)
  \hex(11,4)\black(11,4)
  \hex(11,5)\white(11,5)
  \hex(12,0)\white(12,0)
  \hex(12,1)\black(12,1)
  \hex(12,2)
  \hex(12,3)
  \hex(12,4)\white(12,4)
  \hex(12,5)\white(12,5)
  \hex(13,0)\white(13,0)
  \hex(13,1)\black(13,1)
  \hex(13,2)\white(13,2)
  \hex(13,3)
  \hex(13,4)
  \hex(13,5)\white(13,5)
  \hex(14,0)\white(14,0)
  \hex(14,1)\black(14,1)
  \hex(14,2)
  \hex(14,3)
  \hex(14,4)\black(14,4)
  \hex(14,5)\white(14,5)
  \hex(15,0)\white(15,0)
  \hex(15,1)\black(15,1)
  \hex(15,2)\black(15,2)
  \hex(15,3)\white(15,3)
  \hex(15,4)\black(15,4)
  \hex(15,5)\white(15,5)
  \hex(16,0)\white(16,0)
  \hex(16,1)
  \hex(16,2)
  \hex(16,3)\black(16,3)
  \hex(16,4)\black(16,4)
  \hex(16,5)\white(16,5)
  \hex(17,0)\white(17,0)
  \hex(17,1)\black(17,1)
  \hex(17,2)
  \hex(17,3)\white(17,3)
  \hex(17,4)\black(17,4)
  \hex(17,5)\white(17,5)
  \hex(18,0)\white(18,0)
  \hex(18,1)
  \hex(18,2)
  \hex(18,3)\white(18,3)
  \hex(18,4)\black(18,4)
  \hex(18,5)\white(18,5)
  \hex(19,0)\white(19,0)
  \hex(19,1)\white(19,1)
  \hex(19,2)\black(19,2)
  \hex(19,3)\white(19,3)
  \hex(19,4)\black(19,4)
  \hex(19,5)\white(19,5)
  \hex(20,0)\white(20,0)
  \hex(20,1)\black(20,1)
  \hex(21,0)\black(21,0)
  \hex(22,0)\white(22,0)
  \hex(21,1)\white(21,1)
  \hex(20,2)\white(20,2)
  \hex(20,3)\white(20,3)
  \hex(20,4)\black(20,4)
  \hex(21,2)\black(21,2)
  \hex(21,3)\black(21,3)
  \hex(21,4)\black(21,4)
  \hex(22,1)\black(22,1)
  \hex(22,2)\white(22,2)
  \hex(22,3)\white(22,3)
  \hex(22,4)\black(22,4)
  \hex(23,0)\white(23,0)
  \hex(23,1)
  \hex(23,2)
  \hex(23,3)\white(23,3)
  \hex(23,4)\black(23,4)
  \hex(24,-1)\black(24,-1)
  \hex(24,0)\white(24,0)
  \hex(24,1)
  \hex(24,2)\black(24,2)
  \hex(24,3)\white(24,3)
  \hex(24,4)\black(24,4)
  \hex(25,-1)\black(25,-1)
  \hex(25,0)\black(25,0)
  \hex(25,1)
  \hex(25,2)
  \hex(25,3)\white(25,3)
  \hex(25,4)\black(25,4)
  \hex(26,-1)\black(26,-1)
  \hex(26,0)
  \hex(26,1)
  \hex(26,2)\white(26,2)
  \hex(26,3)
  \hex(26,4)\black(26,4)
  \hex(27,-1)\black(27,-1)
  \hex(27,0)\white(27,0)
  \hex(27,1)\black(27,1)
  \hex(27,2)
  \hex(27,3)\white(27,3)
  \hex(27,4)\black(27,4)
  \hex(28,-1)\black(28,-1)
  \hex(28,0)\white(28,0)
  \hex(28,1)\white(28,1)
  \hex(28,2)
  \hex(28,3)
  \hex(28,4)\black(28,4)
  \hex(29,-1)\black(29,-1)
  \hex(29,0)\white(29,0)
  \hex(29,1)\black(29,1)
  \hex(29,2)
  \hex(29,3)\black(29,3)
  \hex(29,4)\black(29,4)
  \hex(30,-1)\black(30,-1)
  \hex(30,0)\white(30,0)
  \hex(30,1)\black(30,1)
  \hex(30,2)\white(30,2)
  \hex(30,3)\white(30,3)
  \hex(30,4)\white(30,4)
  \hex(31,-1)\black(31,-1)
  \hex(31,0)\white(31,0)
  \hex(31,1)\black(31,1)
  \hex(32,-1)\black(32,-1)
  \hex(32,0)\white(32,0)
  %
  \foreach\i in {4,...,34} {\hex(\i,-2)\white(\i,-2)}
  \foreach\i in {4,...,23} {\hex(\i,-1)\black(\i,-1)}
  \foreach\i in {20,...,30} {\hex(\i,5)\white(\i,5)}
  \hex(3,-1)\white(3,-1)
  \hex(33,-1)\white(33,-1)
  \draw[thick,decoration={brace,mirror},decorate] \coord(-0.5,6) -- node[below=1ex]{$G_1$} \coord(8.5,6);
  \draw[thick,decoration={brace,mirror},decorate] \coord(9.5,6) -- node[below=1ex]{$G_1$} \coord(18.5,6);
  \draw[thick,decoration={brace,mirror},decorate] \coord(20.5,6) -- node[below=1ex]{$G\opp_1$} \coord(29.5,6);
\end{hexboard}
\]

\subsection{Cofinality}

The increasing sequence of superswitches $G_0,G_1,\ldots$ has another
useful property: it is cofinal among the set of all finite passable
games $H$ over $\s{\bot,a,b,\top}$ in which Left cannot achieve the
outcome $\top$.

\begin{proposition}\label{prop-superswitch-cofinal}
  Let $H$ be a finite passable game over $\s{\bot,a,b,\top}$. If
  $\top\ntri H$, then there exists some $n$ such that $H\leq G_n$.
  If $\top\nleq H$, then there exists some $n$ such that $H\tri G_n$.
\end{proposition}

\begin{proof}
  By the fundamental theorem (Theorem~\ref{thm:fundamental}), we can
  assume without loss of generality that $H$ is monotone. We claim
  that the following hold for all monotone games:
  \begin{enumerate}\alphalabels
  \item If $\top\nleq H$ and $d\geq\depth(H)+2$, then $H\tri G_d$.
  \item If $\top\ntri H$ and $d\geq\depth(H)+2$, then $H\leq G_d$.
  \end{enumerate}
  We prove these claims by induction on $H$. If $H$ is atomic, then
  from either one of the assumptions $\top\nleq H$ or $\top\ntri H$,
  we know that $H\in\s{\bot,a,b}$. Then $H\leq G_2$ holds by immediate
  calculation, and therefore also $H\tri G_d$ and $H\leq G_d$ for any
  $d\geq 2$.  Now suppose that $H$ is composite and let
  $d\geq\depth(H)+2$. To prove (a), assume $\top\nleq H$. By the
  contrapositive of Lemma~\ref{lem:topleqG}, there is some $H^R$
  satisfying $\top\ntri H^R$. By the induction hypothesis (b),
  $H^R\leq G_d$, which implies $H\tri G_d$. To prove (b), assume
  $\top\ntri H$.  To prove $H\leq G_d$, first consider any left option
  $H^L$. We must show $H^L\tri G_d$. This follows from the induction
  hypothesis (a) because $\top\nleq H^L$. Next, consider the unique
  right option $G_{d-1}$ of $G_d$. We must show $H\tri G_{d-1}$. Let
  $H^R$ be any right option of $H$. Since $H^R$ is monotone, we have
  $H^R\leq H$; it follows that $\top\ntri H^R$. Also, $H^R$ has
  smaller depth than $H$, so by the induction hypothesis (b), we have
  $H^R\leq G_{d-1}$. This implies $H\tri G_{d-1}$ as claimed.
\end{proof}

\subsection{Simpleswitches}
\label{ssec:simpleswitch}

We will see an application of cofinality in
Section~\ref{ssec:connects-both}, but first we note that superswitches
are relatively inefficient: each additional level of the superswitch
(i.e., passing from $G_n$ to $G_{n+1}$) requires a Hex position with
12 additional empty cells. It turns out that there is a more efficient
cofinal sequence that requires only 6 empty cells per level. We call
these the \emph{simpleswitches}. They are defined by:
\[
\begin{array}{lll}
  H_0 &=& a, \\
  H_1 &=& \g{a,\g{\top|b}|a}, \\
  H_{n+1} &=& \g{a,b|H_n}\quad\mbox{for all $n\geq 1$.}
\end{array}
\]
Except for the variation in $H_1$, the simpleswitches are very similar
to the superswitches. It is easy to see that $G_n\leq H_n\leq G_{n+1}$
for all $n\geq 0$, and therefore the simpleswitches form a strictly
increasing sequence with the same cofinality as the superswitches.

By a proof very similar to that of Lemma~\ref{lem:G-n+1}, we find that
we have $H_n \plusc H_1 \eq H_{n+1}$ for all $n$. Also, we can verify
by direct evaluation that the following is a 3-terminal realization of
$H_1$:
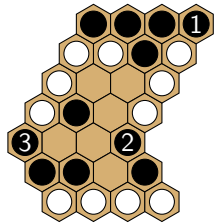
\[
\begin{hexboard}[scale=0.65]
  \rotation{-30}
  \hex(0,4)\black(0,4)\sflabel{3}
  \hex(0,5)\black(0,5)
  \hex(0,6)\white(0,6)
  \hex(1,3)\white(1,3)
  \hex(1,4)
  \hex(1,5)\black(1,5)
  \hex(1,6)\white(1,6)
  \hex(2,2)\white(2,2)
  \hex(2,3)\black(2,3)
  \hex(2,4)
  \hex(2,5)
  \hex(2,6)\white(2,6)
  \hex(3,1)\white(3,1)
  \hex(3,2)
  \hex(3,3)
  \hex(3,4)\black(3,4)\sflabel{2}
  \hex(3,5)\black(3,5)
  \hex(3,6)\white(3,6)
  \hex(4,0)\black(4,0)
  \hex(4,1)\white(4,1)
  \hex(4,2)
  \hex(4,3)\white(4,3)
  \hex(5,0)\black(5,0)
  \hex(5,1)\black(5,1)
  \hex(5,2)\white(5,2)
  \hex(6,0)\black(6,0)
  \hex(6,1)\white(6,1)
  \hex(7,0)\black(7,0)\sflabel{1}
\end{hexboard}
\]
Thus, for example, we get the following realization for the
simpleswitch $H_5 = H_1\plusc H_1\plusc H_1\plusc H_1\plusc H_1$:
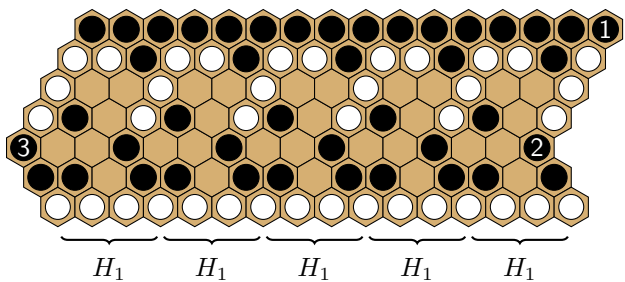
\[
\begin{hexboard}[scale=0.65]
  \rotation{-30}
  \hex(0,4)\black(0,4)\sflabel{3}
  \hex(0,5)\black(0,5)
  \hex(0,6)\white(0,6)
  \hex(1,3)\white(1,3)
  \hex(1,4)
  \hex(1,5)\black(1,5)
  \hex(1,6)\white(1,6)
  \hex(2,2)\white(2,2)
  \hex(2,3)\black(2,3)
  \hex(2,4)
  \hex(2,5)
  \hex(2,6)\white(2,6)
  \hex(3,1)\white(3,1)
  \hex(3,2)
  \hex(3,3)
  \hex(3,4)\black(3,4)
  \hex(3,5)\black(3,5)
  \hex(3,6)\white(3,6)
  \hex(4,0)\black(4,0)
  \hex(4,1)\white(4,1)
  \hex(4,2)
  \hex(4,3)\white(4,3)
  \hex(5,0)\black(5,0)
  \hex(5,1)\black(5,1)
  \hex(5,2)\white(5,2)
  \hex(6,0)\black(6,0)
  \hex(6,1)\white(6,1)
  \hex(7,0)\black(7,0)
  \hex(4,4)
  \hex(4,5)\black(4,5)
  \hex(4,6)\white(4,6)
  \hex(5,3)\black(5,3)
  \hex(5,4)
  \hex(5,5)
  \hex(5,6)\white(5,6)
  \hex(6,2)
  \hex(6,3)
  \hex(6,4)\black(6,4)
  \hex(6,5)\black(6,5)
  \hex(6,6)\white(6,6)
  \hex(7,1)\white(7,1)
  \hex(7,2)
  \hex(7,3)\white(7,3)
  \hex(8,0)\black(8,0)
  \hex(8,1)\black(8,1)
  \hex(8,2)\white(8,2)
  \hex(9,0)\black(9,0)
  \hex(9,1)\white(9,1)
  \hex(10,0)\black(10,0)
  \hex(7,4)
  \hex(7,5)\black(7,5)
  \hex(7,6)\white(7,6)
  \hex(8,3)\black(8,3)
  \hex(8,4)
  \hex(8,5)
  \hex(8,6)\white(8,6)
  \hex(9,2)
  \hex(9,3)
  \hex(9,4)\black(9,4)
  \hex(9,5)\black(9,5)
  \hex(9,6)\white(9,6)
  \hex(10,1)\white(10,1)
  \hex(10,2)
  \hex(10,3)\white(10,3)
  \hex(11,0)\black(11,0)
  \hex(11,1)\black(11,1)
  \hex(11,2)\white(11,2)
  \hex(12,0)\black(12,0)
  \hex(12,1)\white(12,1)
  \hex(13,0)\black(13,0)
  \hex(10,4)
  \hex(10,5)\black(10,5)
  \hex(10,6)\white(10,6)
  \hex(11,3)\black(11,3)
  \hex(11,4)
  \hex(11,5)
  \hex(11,6)\white(11,6)
  \hex(12,2)
  \hex(12,3)
  \hex(12,4)\black(12,4)
  \hex(12,5)\black(12,5)
  \hex(12,6)\white(12,6)
  \hex(13,1)\white(13,1)
  \hex(13,2)
  \hex(13,3)\white(13,3)
  \hex(14,0)\black(14,0)
  \hex(14,1)\black(14,1)
  \hex(14,2)\white(14,2)
  \hex(15,0)\black(15,0)
  \hex(15,1)\white(15,1)
  \hex(16,0)\black(16,0)
  \hex(13,4)
  \hex(13,5)\black(13,5)
  \hex(13,6)\white(13,6)
  \hex(14,3)\black(14,3)
  \hex(14,4)
  \hex(14,5)
  \hex(14,6)\white(14,6)
  \hex(15,2)
  \hex(15,3)
  \hex(15,4)\black(15,4)\sflabel{2}
  \hex(15,5)\black(15,5)
  \hex(15,6)\white(15,6)
  \hex(16,1)\white(16,1)
  \hex(16,2)
  \hex(16,3)\white(16,3)
  \hex(17,0)\black(17,0)
  \hex(17,1)\black(17,1)
  \hex(17,2)\white(17,2)
  \hex(18,0)\black(18,0)
  \hex(18,1)\white(18,1)
  \hex(19,0)\black(19,0)\sflabel{1}
  \draw[thick,decoration={brace,mirror},decorate] \coord(-0.4,7) -- node[below=1ex]{$H_1$} \coord(2.4,7);
  \draw[thick,decoration={brace,mirror},decorate] \coord(2.6,7) -- node[below=1ex]{$H_1$} \coord(5.4,7);
  \draw[thick,decoration={brace,mirror},decorate] \coord(5.6,7) -- node[below=1ex]{$H_1$} \coord(8.4,7);
  \draw[thick,decoration={brace,mirror},decorate] \coord(8.6,7) -- node[below=1ex]{$H_1$} \coord(11.4,7);
  \draw[thick,decoration={brace,mirror},decorate] \coord(11.6,7) -- node[below=1ex]{$H_1$} \coord(14.4,7);
\end{hexboard}
\]
It is evident that the Hex realizations of simpleswitches are much
more compact than those we gave of superswitches.

\subsection{Application: Verifying connects-both templates}
\label{ssec:connects-both}

In Hex, an \emph{edge template} is a region that includes a
distinguished black stone and a black edge, such that the following
two properties hold:
\begin{itemize}
\item \emph{Validity:} Black can guarantee to connect the given
  stone to the edge within the region, even with White moving first.
\item \emph{Minimality:} Removing any empty cell or any black stone
  from the region makes the template invalid.
\end{itemize}
The following is an example of an edge template:
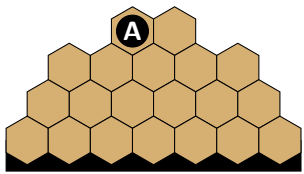
\[
\begin{hexboard}[scale=0.8]
  \template(7,4)
  \foreach \i in {5,...,6} {\hex(\i,1)}
  \foreach \i in {3,...,7} {\hex(\i,2)}
  \foreach \i in {2,...,7} {\hex(\i,3)}
  \foreach \i in {1,...,7} {\hex(\i,4)}
  \black(5,1)\sflabel{\bfseries{A}}
\end{hexboard}
\]
Here, the distinguished stone is marked ``A'', and it also
happens to be the only stone in the template. For more information
about edge templates, and a proof of the validity of the above
template, see {\cite{Seymour}}.

Large templates are difficult to verify by hand; for validity, one
must consider every possible strategy by White, and minimality is even
more difficult to verify manually. Fortunately, edge templates can
easily be verified using Hex solver software such as Mohex
{\cite{Mohex}}. Given a board position and a player to move, exactly
one player must have a winning strategy, and the Hex solver computes
whether that player is Black or White. For example, to verify the
validity of the above template, it is sufficient to run a Hex solver
on the following position, with White to move:
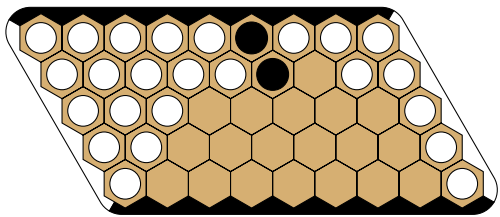
\[
\begin{hexboard}[scale=0.8]
  \rotation{-30}
  \board(9,5)
  \black(6,1)
  \black(6,2)
  \foreach\i in {1,...,5} {\white(\i,1)}
  \foreach\i in {1,...,5} {\white(\i,2)}
  \foreach\i in {1,...,3} {\white(\i,3)}
  \foreach\i in {1,...,2} {\white(\i,4)}
  \foreach\i in {1,...,1} {\white(\i,5)}
  \foreach\i in {7,...,9} {\white(\i,1)}
  \foreach\i in {8,...,9} {\white(\i,2)}
  \foreach\i in {9,...,9} {\white(\i,3)}
  \foreach\i in {9,...,9} {\white(\i,4)}
  \foreach\i in {9,...,9} {\white(\i,5)}
\end{hexboard}
\]
If the winner is Black, the template is valid. Moreover, minimality is
easily checked by placing a white stone on one of the empty cells and
re-running the solver. The winner should be White. If this can be
repeated for every empty cell of the template, the template is
minimal. (If the template contains additional black stones, one should
also try removing these stones one at a time and check that the
template is no longer valid).

Some edge templates have additional properties. For example, the
following template has the property that Black can guarantee to
connect \emph{both} stones $A$ and $B$ to the edge, with White going
first. In other words, Black does not have to choose which of $A$ and
$B$ to connect, but can always connect both of them. We call such a
template a \emph{connects-both template}.
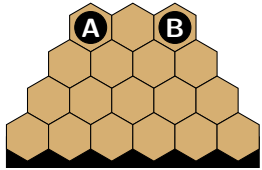
\[
\begin{hexboard}[scale=0.8]
  \template(6,4)
  \foreach \i in {4,...,6} {\hex(\i,1)}
  \foreach \i in {3,...,6} {\hex(\i,2)}
  \foreach \i in {2,...,6} {\hex(\i,3)}
  \foreach \i in {1,...,6} {\hex(\i,4)}
  \black(4,1)\label{{\sf \bfseries{A}}}
  \black(6,1)\label{{\sf \bfseries{B}}}
\end{hexboard}
\]
A priori, it is not clear how we can use a Hex solver to verify a
connects-both template. A naive first attempt might be to connect both
$A$ and $B$ to the opposite board edge, but in that case, Black would win
if Black could connect $A$ \emph{or} $B$ to the edge, rather than $A$
\emph{and} $B$. What we need is to put the would-be template into a
context such that Black wins if and only if Black connects both
stones. Unfortunately, as we will show in
Section~\ref{ssec:no-single-context}, there exists no single Hex context
with this property. However, as we will now show, there is a
\emph{sequence} of contexts that will do the job in the limit.

We say that a 3-terminal position has the \emph{connects-both}
property if Black can guarantee to connect all three terminals to each
other, with White to move first. In other words, this is the case if
and only if the position has value $\top$.

\begin{theorem}\label{thm:connects-both}
  A 3-terminal position $G$ has the connects-both property if and only
  if $G\plusj H\opp_n = \top$ for all dual simpleswitches $H\opp_n$. 
\end{theorem}

\begin{proof}
  The left-to-right implication is easy, because an easy induction
  shows that $\top\plusj H\opp_n \eq \top$ holds for all $n$.

  For the right-to-left direction, assume $G\plusj H\opp_n = \top$ for
  all dual simpleswitches $H\opp_n$. Assume, for the sake of obtaining a
  contradiction, that $G\neqq \top$. Then $G<\top$. Note that because
  $G$ is a Hex position, $G$ is necessarily finite and monotone. Let
  $G'$ be the game obtained from $G$ by replacing each occurrence of
  the atom $c$ by $\bot$. Since the function mapping $c$ to $\bot$ is
  monotone, $G'$ is a monotone game over $\s{\bot,a,b,\top}$. We also
  have $G'\leq G$, thus $G'<\top$.  By
  Proposition~\ref{prop-superswitch-cofinal}, there is some
  superswitch $G_n$ such that $G'\tri G_n$. It follows that $G'\plusj
  H\opp_{n+2} \leq G'\plusj G\opp_{n+2} \tri G_n\plusj G\opp_{n+2} \eq
  \bot$, where the last equivalence holds by
  Proposition~\ref{prop:juxta}.  On the other hand, juxtaposition
  satisfies $c\plusj x = \bot\plusj x$ for every atom $x$. Therefore,
  since $G'$ only differs from $G$ by replacing some atoms $c$ by
  $\bot$, we have $G'\plusj H\opp_{n+2} = G\plusj H\opp_{n+2}$.
  Putting this together, we have $G\plusj H\opp_{n+2}\tri\bot$,
  contradicting the assumption that $G\plusj H\opp_{n+2} = \top$.  
\end{proof}

\begin{figure}
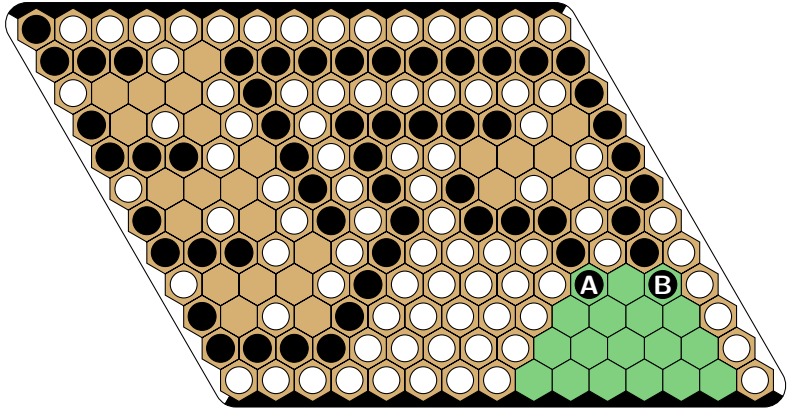

  \[
\begin{hexboard}[baseline={($(current bounding box.center)-(0,1ex)$)},scale=0.7]
  \rotation{-30}
  \board(15,12)
\black(1,1)\white(2,1)\white(3,1)\white(4,1)\white(5,1)\white(6,1)\white(7,1)\white(8,1)\white(9,1)\white(10,1)\white(11,1)\white(12,1)\white(13,1)\white(14,1)\white(15,1)
\black(1,2)\black(2,2)\black(3,2)\white(4,2)\nofil(5,2)\black(6,2)\black(7,2)\black(8,2)\black(9,2)\black(10,2)\black(11,2)\black(12,2)\black(13,2)\black(14,2)\black(15,2)
\white(1,3)\nofil(2,3)\nofil(3,3)\nofil(4,3)\white(5,3)\black(6,3)\white(7,3)\white(8,3)\white(9,3)\white(10,3)\white(11,3)\white(12,3)\white(13,3)\white(14,3)\black(15,3)
\black(1,4)\nofil(2,4)\white(3,4)\nofil(4,4)\white(5,4)\black(6,4)\white(7,4)\black(8,4)\black(9,4)\black(10,4)\black(11,4)\black(12,4)\white(13,4)\nofil(14,4)\black(15,4)
\black(1,5)\black(2,5)\black(3,5)\white(4,5)\nofil(5,5)\black(6,5)\white(7,5)\black(8,5)\white(9,5)\white(10,5)\nofil(11,5)\nofil(12,5)\nofil(13,5)\white(14,5)\black(15,5)
\white(1,6)\nofil(2,6)\nofil(3,6)\nofil(4,6)\white(5,6)\black(6,6)\white(7,6)\black(8,6)\white(9,6)\black(10,6)\nofil(11,6)\white(12,6)\nofil(13,6)\white(14,6)\black(15,6)
\black(1,7)\nofil(2,7)\white(3,7)\nofil(4,7)\white(5,7)\black(6,7)\white(7,7)\black(8,7)\white(9,7)\black(10,7)\black(11,7)\black(12,7)\white(13,7)\black(14,7)\white(15,7)
\black(1,8)\black(2,8)\black(3,8)\white(4,8)\nofil(5,8)\white(6,8)\black(7,8)\white(8,8)\white(9,8)\white(10,8)\white(11,8)\black(12,8)\white(13,8)\black(14,8)\white(15,8)
\white(1,9)\nofil(2,9)\nofil(3,9)\nofil(4,9)\white(5,9)\black(6,9)\white(7,9)\white(8,9)\white(9,9)\white(10,9)\white(11,9)\black(12,9)\sflabel{\bfseries{A}}\carry(13,9)\black(14,9)\sflabel{\bfseries{B}}\white(15,9)
\black(1,10)\nofil(2,10)\white(3,10)\nofil(4,10)\black(5,10)\white(6,10)\white(7,10)\white(8,10)\white(9,10)\white(10,10)\carry(11,10)\carry(12,10)\carry(13,10)\carry(14,10)\white(15,10)
\black(1,11)\black(2,11)\black(3,11)\black(4,11)\white(5,11)\white(6,11)\white(7,11)\white(8,11)\white(9,11)\carry(10,11)\carry(11,11)\carry(12,11)\carry(13,11)\carry(14,11)\white(15,11)
\white(1,12)\white(2,12)\white(3,12)\white(4,12)\white(5,12)\white(6,12)\white(7,12)\white(8,12)\carry(9,12)\carry(10,12)\carry(11,12)\carry(12,12)\carry(13,12)\carry(14,12)\white(15,12)
\carry(12,9)
\carry(14,9)
\end{hexboard}
  \]
  \caption{Verifying a connects-both template}
  \label{fig:connects-both}
\end{figure}

Theorem~\ref{thm:connects-both} provides a method for verifying
connects-both templates using a Hex solver: we merely need to put the
template in the context of a sufficiently high tower of dual
simpleswitches. An example of this is shown in
Figure~\ref{fig:connects-both}. Here, the template itself appears in
the lower right corner, and the rest of the board is taken up by a
dual simpleswitch of depth 4. If the position is a second-player win
for Black for dual simpleswitches of \emph{all} depths, then the
template is valid.

In practice, of course, one cannot check this for infinitely many dual
simpleswitches. Theoretically, it is sufficient to limit the depth of
the dual simpleswitch to slightly more than the depth of the candidate
template, i.e., the number of its empty cells. But in most cases, this
depth is too large to be feasibly checkable. However, with some manual
work, it is possible to verify a connects-both template using only a
single dual simpleswitch of some fixed small depth. If the solver
determines that White can win going first, the template is not
valid. Otherwise, for every possible white intrusion, the solver will
suggest a potential black winning response. If the response is such
that it connects terminals $A$ and $B$ to each other, then the depth
of the simpleswitch no longer matters; otherwise, one must recursively
consider all white follow-up moves. In this way, the validity of most
templates can be verified in a relatively small number of steps.

For the proof of minimality, one can as usual consider every possible
way of placing a single white stone in the template or removing a
black stone, and uses the solver to show that the resulting region is
a first-player win for White. As long as this procedure succeeds, the
depth of the used simpleswitches does not matter, because by
Theorem~\ref{thm:connects-both}, a dual simpleswitch of any one
depth is sufficient to invalidate a template.

\subsection{Tripleswitches}\label{ssec:tripleswitches}

In Proposition~\ref{prop-superswitch-cofinal}, we showed that the
sequence of superswitches is cofinal for the class of games over
$\s{\bot,a,b,\top}$ in which Left cannot achieve the outcome $\top$.
There is an analogous cofinal sequence of games over
$\s{\bot,a,b,c,\top}$. Define the following games:
\[
\begin{array}{lll}
  T_0 &=& \bot, \\
  T_{n+1} &=& \g{a,b,c|T_n}\quad\mbox{for all $n\geq 0$.}
\end{array}
\]
For example, $T_3 = \g{a,b,c|\g{a,b,c|\g{a,b,c|\bot}}}$.
We call these games \emph{ideal tripleswitches}. It is easy to see
that they form an increasing sequence of passable games. The following
proposition establishes the cofinality of the sequence:

\begin{proposition}\label{prop:ideal-tripleswitch-cofinal}
  Let $H$ be a finite passable game over $\s{\bot,a,b,c,\top}$. If
  $\top\ntri H$, then there exists some $n$ such that $H\leq T_n$.
  If $\top\nleq H$, then there exists some $n$ such that $H\tri T_n$.
\end{proposition}

\begin{proof}
  The proof is exactly the same as that of
  Proposition~\ref{prop-superswitch-cofinal}, with one more atom.
\end{proof}

With the exception of $T_0$, $T_1$, and $T_2$, we do not know whether
the ideal tripleswitches are realizable as 3-terminal Hex
positions. However, there is a sequence of 3-terminal Hex positions
that has the same cofinality as the ideal tripleswitches. We call
these the \emph{Hex tripleswitches}.

\begin{definition}\label{def:pinwheel}
  If $G_1$, $G_2$, and $G_3$ are 3-terminal positions, their
  \emph{pinwheel composition}, which we write as $\pinwheel(G_1, G_2,
  G_3)$, is the 3-terminal position shown schematically in the
  following diagram:
  \begin{equation}\label{eqn:pinwheel}
    \pinwheel(G_1, G_2, G_3)\quad=\quad
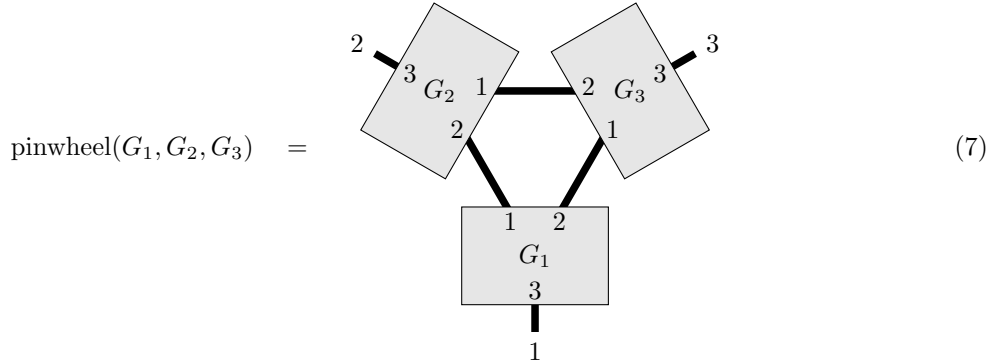
    \begin{tikzpicture}[scale=0.65,baseline={(center.south)}]
      \path (270:2.25cm) node (a) {};
      \path (150:2.25cm) node (b) {};
      \path (30:2.25cm) node (c) {};
      \path (270:4cm) node (aa) {};
      \path (150:4cm) node (bb) {};
      \path (30:4cm) node (cc) {};
      \path (270:4.2cm) node (aaa) {};
      \path (150:4.2cm) node (bbb) {};
      \path (30:4.2cm) node (ccc) {};
      \node (center) at (0,0) {};
      \draw[terminal] (a) -- (b);
      \draw[terminal] (b) -- (c);
      \draw[terminal] (c) -- (a);
      \draw[terminal] (a) -- (aa) (aaa) node{$1$};
      \draw[terminal] (b) -- (bb) (bbb) node{$2$};
      \draw[terminal] (c) -- (cc) (ccc) node{$3$};
      \draw[fill=black!10,rotate around={0:(a)}]
      (a) +(-1.5,-1) rectangle node{$G_1$} +(1.5,1)
      (a) +(0,-0.7) node {$3$}
      (a) +(-0.5,0.7) node {$1$}
      (a) +(0.5,0.7) node {$2$}
      ;
      \draw[fill=black!10,rotate around={-120:(b)}]
      (b) +(-1.5,-1) rectangle node{$G_2$} +(1.5,1)
      (b) +(0,-0.7) node {$3$}
      (b) +(-0.5,0.7) node {$1$}
      (b) +(0.5,0.7) node {$2$}
      ;
      \draw[fill=black!10,rotate around={120:(c)}]
      (c) +(-1.5,-1) rectangle node{$G_3$} +(1.5,1)
      (c) +(0,-0.7) node {$3$}
      (c) +(-0.5,0.7) node {$1$}
      (c) +(0.5,0.7) node {$2$}
      ;
    \end{tikzpicture}
  \end{equation}
\end{definition}

A number of identities for atomic pinwheels are illustrated in
Figure~\ref{fig:pinwheel-atoms}.

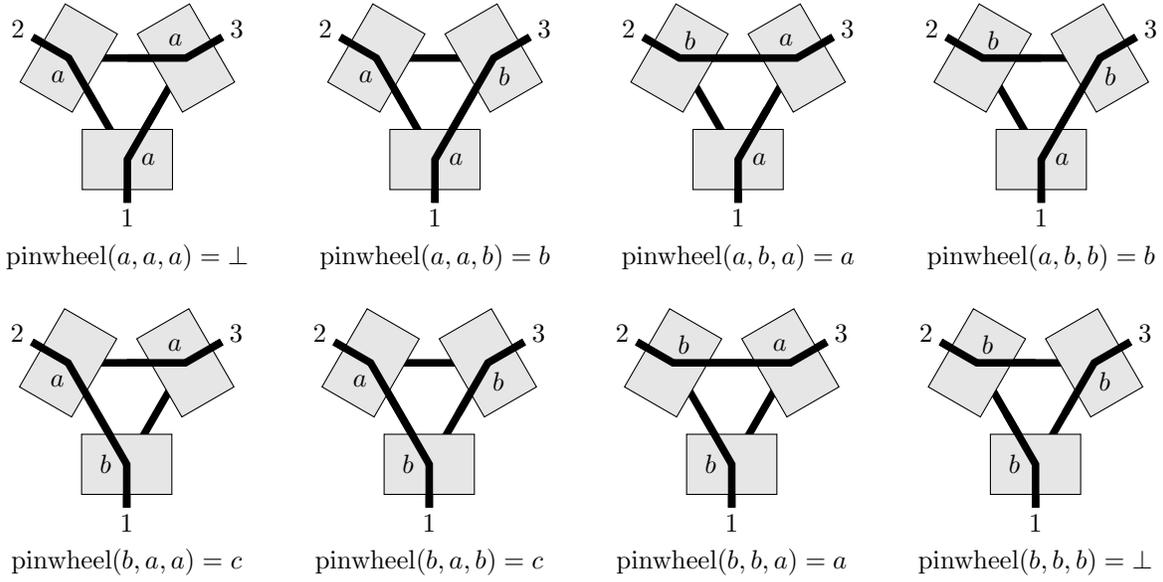
\begin{figure}
  \[
  \begin{tikzpicture}[scale=0.4]
    \path (270:2.25cm) node (a) {};
    \path (150:2.25cm) node (b) {};
    \path (30:2.25cm) node (c) {};
    \path (270:4cm) node (aa) {};
    \path (150:4cm) node (bb) {};
    \path (30:4cm) node (cc) {};
    \path (270:4.2cm) node (aaa) {};
    \path (150:4.2cm) node (bbb) {};
    \path (30:4.2cm) node (ccc) {};
    \draw[terminal] (a) -- (b);
    \draw[terminal] (b) -- (c);
    \draw[terminal] (c) -- (a);
    \draw[terminal] (a) -- (aa) (aaa) node{$1$};
    \draw[terminal] (b) -- (bb) (bbb) node{$2$};
    \draw[terminal] (c) -- (cc) (ccc) node{$3$};
    \draw[fill=black!10,rotate around={0:(a)}]
    (a) +(-1.5,-1) rectangle +(1.5,1)
    (a) +(0.7,0) node {$a$}
    ;
    \draw[fill=black!10,rotate around={-120:(b)}]
    (b) +(-1.5,-1) rectangle +(1.5,1)
    (b) +(0.7,0) node {$a$}
    ;
    \draw[fill=black!10,rotate around={120:(c)}]
    (c) +(-1.5,-1) rectangle +(1.5,1)
    (c) +(0.7,0) node {$a$}
    ;
    \draw[terminal] (aa) -- (a.center) -- ($(a)!0.5!(c)$);
    \draw[terminal] (bb) -- (b.center) -- ($(b)!0.5!(a)$);
    \draw[terminal] (cc) -- (c.center) -- ($(c)!0.5!(b)$);
    \path (0,-5.5) node{$\pinwheel(a,a,a) = \bot$};
  \end{tikzpicture}
  \qquad
  \begin{tikzpicture}[scale=0.4]
    \path (270:2.25cm) node (a) {};
    \path (150:2.25cm) node (b) {};
    \path (30:2.25cm) node (c) {};
    \path (270:4cm) node (aa) {};
    \path (150:4cm) node (bb) {};
    \path (30:4cm) node (cc) {};
    \path (270:4.2cm) node (aaa) {};
    \path (150:4.2cm) node (bbb) {};
    \path (30:4.2cm) node (ccc) {};
    \draw[terminal] (a) -- (b);
    \draw[terminal] (b) -- (c);
    \draw[terminal] (c) -- (a);
    \draw[terminal] (a) -- (aa) (aaa) node{$1$};
    \draw[terminal] (b) -- (bb) (bbb) node{$2$};
    \draw[terminal] (c) -- (cc) (ccc) node{$3$};
    \draw[fill=black!10,rotate around={0:(a)}]
    (a) +(-1.5,-1) rectangle +(1.5,1)
    (a) +(0.7,0) node {$a$}
    ;
    \draw[fill=black!10,rotate around={-120:(b)}]
    (b) +(-1.5,-1) rectangle +(1.5,1)
    (b) +(0.7,0) node {$a$}
    ;
    \draw[fill=black!10,rotate around={120:(c)}]
    (c) +(-1.5,-1) rectangle +(1.5,1)
    (c) +(-0.7,0) node {$b$}
    ;
    \draw[terminal] (aa) -- (a.center) -- ($(a)!0.5!(c)$);
    \draw[terminal] (bb) -- (b.center) -- ($(b)!0.5!(a)$);
    \draw[terminal] (cc) -- (c.center) -- ($(c)!0.5!(a)$);
    \path (0,-5.5) node{$\pinwheel(a,a,b) = b$};
  \end{tikzpicture}
  \qquad
  \begin{tikzpicture}[scale=0.4]
    \path (270:2.25cm) node (a) {};
    \path (150:2.25cm) node (b) {};
    \path (30:2.25cm) node (c) {};
    \path (270:4cm) node (aa) {};
    \path (150:4cm) node (bb) {};
    \path (30:4cm) node (cc) {};
    \path (270:4.2cm) node (aaa) {};
    \path (150:4.2cm) node (bbb) {};
    \path (30:4.2cm) node (ccc) {};
    \draw[terminal] (a) -- (b);
    \draw[terminal] (b) -- (c);
    \draw[terminal] (c) -- (a);
    \draw[terminal] (a) -- (aa) (aaa) node{$1$};
    \draw[terminal] (b) -- (bb) (bbb) node{$2$};
    \draw[terminal] (c) -- (cc) (ccc) node{$3$};
    \draw[fill=black!10,rotate around={0:(a)}]
    (a) +(-1.5,-1) rectangle +(1.5,1)
    (a) +(0.7,0) node {$a$}
    ;
    \draw[fill=black!10,rotate around={-120:(b)}]
    (b) +(-1.5,-1) rectangle +(1.5,1)
    (b) +(-0.7,0) node {$b$}
    ;
    \draw[fill=black!10,rotate around={120:(c)}]
    (c) +(-1.5,-1) rectangle +(1.5,1)
    (c) +(0.7,0) node {$a$}
    ;
    \draw[terminal] (aa) -- (a.center) -- ($(a)!0.5!(c)$);
    \draw[terminal] (bb) -- (b.center) -- ($(b)!0.5!(c)$);
    \draw[terminal] (cc) -- (c.center) -- ($(c)!0.5!(b)$);
    \path (0,-5.5) node{$\pinwheel(a,b,a) = a$};
  \end{tikzpicture}
  \qquad
  \begin{tikzpicture}[scale=0.4]
    \path (270:2.25cm) node (a) {};
    \path (150:2.25cm) node (b) {};
    \path (30:2.25cm) node (c) {};
    \path (270:4cm) node (aa) {};
    \path (150:4cm) node (bb) {};
    \path (30:4cm) node (cc) {};
    \path (270:4.2cm) node (aaa) {};
    \path (150:4.2cm) node (bbb) {};
    \path (30:4.2cm) node (ccc) {};
    \draw[terminal] (a) -- (b);
    \draw[terminal] (b) -- (c);
    \draw[terminal] (c) -- (a);
    \draw[terminal] (a) -- (aa) (aaa) node{$1$};
    \draw[terminal] (b) -- (bb) (bbb) node{$2$};
    \draw[terminal] (c) -- (cc) (ccc) node{$3$};
    \draw[fill=black!10,rotate around={0:(a)}]
    (a) +(-1.5,-1) rectangle +(1.5,1)
    (a) +(0.7,0) node {$a$}
    ;
    \draw[fill=black!10,rotate around={-120:(b)}]
    (b) +(-1.5,-1) rectangle +(1.5,1)
    (b) +(-0.7,0) node {$b$}
    ;
    \draw[fill=black!10,rotate around={120:(c)}]
    (c) +(-1.5,-1) rectangle +(1.5,1)
    (c) +(-0.7,0) node {$b$}
    ;
    \draw[terminal] (aa) -- (a.center) -- ($(a)!0.5!(c)$);
    \draw[terminal] (bb) -- (b.center) -- ($(b)!0.5!(c)$);
    \draw[terminal] (cc) -- (c.center) -- ($(c)!0.5!(a)$);
    \path (0,-5.5) node{$\pinwheel(a,b,b) = b$};
  \end{tikzpicture}
  \]
  \[
  \begin{tikzpicture}[scale=0.4]
    \path (270:2.25cm) node (a) {};
    \path (150:2.25cm) node (b) {};
    \path (30:2.25cm) node (c) {};
    \path (270:4cm) node (aa) {};
    \path (150:4cm) node (bb) {};
    \path (30:4cm) node (cc) {};
    \path (270:4.2cm) node (aaa) {};
    \path (150:4.2cm) node (bbb) {};
    \path (30:4.2cm) node (ccc) {};
    \draw[terminal] (a) -- (b);
    \draw[terminal] (b) -- (c);
    \draw[terminal] (c) -- (a);
    \draw[terminal] (a) -- (aa) (aaa) node{$1$};
    \draw[terminal] (b) -- (bb) (bbb) node{$2$};
    \draw[terminal] (c) -- (cc) (ccc) node{$3$};
    \draw[fill=black!10,rotate around={0:(a)}]
    (a) +(-1.5,-1) rectangle +(1.5,1)
    (a) +(-0.7,0) node {$b$}
    ;
    \draw[fill=black!10,rotate around={-120:(b)}]
    (b) +(-1.5,-1) rectangle +(1.5,1)
    (b) +(0.7,0) node {$a$}
    ;
    \draw[fill=black!10,rotate around={120:(c)}]
    (c) +(-1.5,-1) rectangle +(1.5,1)
    (c) +(0.7,0) node {$a$}
    ;
    \draw[terminal] (aa) -- (a.center) -- ($(a)!0.5!(b)$);
    \draw[terminal] (bb) -- (b.center) -- ($(b)!0.5!(a)$);
    \draw[terminal] (cc) -- (c.center) -- ($(c)!0.5!(b)$);
    \path (0,-5.5) node{$\pinwheel(b,a,a) = c$};
  \end{tikzpicture}
  \qquad
  \begin{tikzpicture}[scale=0.4]
    \path (270:2.25cm) node (a) {};
    \path (150:2.25cm) node (b) {};
    \path (30:2.25cm) node (c) {};
    \path (270:4cm) node (aa) {};
    \path (150:4cm) node (bb) {};
    \path (30:4cm) node (cc) {};
    \path (270:4.2cm) node (aaa) {};
    \path (150:4.2cm) node (bbb) {};
    \path (30:4.2cm) node (ccc) {};
    \draw[terminal] (a) -- (b);
    \draw[terminal] (b) -- (c);
    \draw[terminal] (c) -- (a);
    \draw[terminal] (a) -- (aa) (aaa) node{$1$};
    \draw[terminal] (b) -- (bb) (bbb) node{$2$};
    \draw[terminal] (c) -- (cc) (ccc) node{$3$};
    \draw[fill=black!10,rotate around={0:(a)}]
    (a) +(-1.5,-1) rectangle +(1.5,1)
    (a) +(-0.7,0) node {$b$}
    ;
    \draw[fill=black!10,rotate around={-120:(b)}]
    (b) +(-1.5,-1) rectangle +(1.5,1)
    (b) +(0.7,0) node {$a$}
    ;
    \draw[fill=black!10,rotate around={120:(c)}]
    (c) +(-1.5,-1) rectangle +(1.5,1)
    (c) +(-0.7,0) node {$b$}
    ;
    \draw[terminal] (aa) -- (a.center) -- ($(a)!0.5!(b)$);
    \draw[terminal] (bb) -- (b.center) -- ($(b)!0.5!(a)$);
    \draw[terminal] (cc) -- (c.center) -- ($(c)!0.5!(a)$);
    \path (0,-5.5) node{$\pinwheel(b,a,b) = c$};
  \end{tikzpicture}
  \qquad
  \begin{tikzpicture}[scale=0.4]
    \path (270:2.25cm) node (a) {};
    \path (150:2.25cm) node (b) {};
    \path (30:2.25cm) node (c) {};
    \path (270:4cm) node (aa) {};
    \path (150:4cm) node (bb) {};
    \path (30:4cm) node (cc) {};
    \path (270:4.2cm) node (aaa) {};
    \path (150:4.2cm) node (bbb) {};
    \path (30:4.2cm) node (ccc) {};
    \draw[terminal] (a) -- (b);
    \draw[terminal] (b) -- (c);
    \draw[terminal] (c) -- (a);
    \draw[terminal] (a) -- (aa) (aaa) node{$1$};
    \draw[terminal] (b) -- (bb) (bbb) node{$2$};
    \draw[terminal] (c) -- (cc) (ccc) node{$3$};
    \draw[fill=black!10,rotate around={0:(a)}]
    (a) +(-1.5,-1) rectangle +(1.5,1)
    (a) +(-0.7,0) node {$b$}
    ;
    \draw[fill=black!10,rotate around={-120:(b)}]
    (b) +(-1.5,-1) rectangle +(1.5,1)
    (b) +(-0.7,0) node {$b$}
    ;
    \draw[fill=black!10,rotate around={120:(c)}]
    (c) +(-1.5,-1) rectangle +(1.5,1)
    (c) +(0.7,0) node {$a$}
    ;
    \draw[terminal] (aa) -- (a.center) -- ($(a)!0.5!(b)$);
    \draw[terminal] (bb) -- (b.center) -- ($(b)!0.5!(c)$);
    \draw[terminal] (cc) -- (c.center) -- ($(c)!0.5!(b)$);
    \path (0,-5.5) node{$\pinwheel(b,b,a) = a$};
  \end{tikzpicture}
  \qquad
  \begin{tikzpicture}[scale=0.4]
    \path (270:2.25cm) node (a) {};
    \path (150:2.25cm) node (b) {};
    \path (30:2.25cm) node (c) {};
    \path (270:4cm) node (aa) {};
    \path (150:4cm) node (bb) {};
    \path (30:4cm) node (cc) {};
    \path (270:4.2cm) node (aaa) {};
    \path (150:4.2cm) node (bbb) {};
    \path (30:4.2cm) node (ccc) {};
    \draw[terminal] (a) -- (b);
    \draw[terminal] (b) -- (c);
    \draw[terminal] (c) -- (a);
    \draw[terminal] (a) -- (aa) (aaa) node{$1$};
    \draw[terminal] (b) -- (bb) (bbb) node{$2$};
    \draw[terminal] (c) -- (cc) (ccc) node{$3$};
    \draw[fill=black!10,rotate around={0:(a)}]
    (a) +(-1.5,-1) rectangle +(1.5,1)
    (a) +(-0.7,0) node {$b$}
    ;
    \draw[fill=black!10,rotate around={-120:(b)}]
    (b) +(-1.5,-1) rectangle +(1.5,1)
    (b) +(-0.7,0) node {$b$}
    ;
    \draw[fill=black!10,rotate around={120:(c)}]
    (c) +(-1.5,-1) rectangle +(1.5,1)
    (c) +(-0.7,0) node {$b$}
    ;
    \draw[terminal] (aa) -- (a.center) -- ($(a)!0.5!(b)$);
    \draw[terminal] (bb) -- (b.center) -- ($(b)!0.5!(c)$);
    \draw[terminal] (cc) -- (c.center) -- ($(c)!0.5!(a)$);
    \path (0,-5.5) node{$\pinwheel(b,b,b) = \bot$};
  \end{tikzpicture}
  \]
  \caption{Some identities for pinwheels of atomic positions}
  \label{fig:pinwheel-atoms}
\end{figure}

We define the Hex position $T'_n =
\pinwheel(G_{n+2},G_{n+2},G_{n+2})$, where $G_{n+2}$ is a Hex
realization of the $(n+2)$nd superswitch. Note that, since the only atoms
occurring in the canonical value of $G_{n+2}$ are $a$ and $b$, by the
identities in Figure~\ref{fig:pinwheel-atoms}, the only atoms that can
possibly occur in the canonical value of $T'_n$ are
$\s{\bot,a,b,c}$. Therefore, outcome $\top$ is not achievable, and in
particular, $\top\ntri T'_n$. It follows by
Proposition~\ref{prop:ideal-tripleswitch-cofinal} that $T'_n\leq T_m$
for some $m$. Conversely, we claim that $T_n \leq T'_n$, and
therefore the sequence $(T'_n)_{n\in\N}$ has the same cofinality as
the sequence $(T_n)_{n\in\N}$.

\begin{lemma}\label{lem:hex-tripleswitch}
  For all $n$, $T_n \leq T'_n$.
\end{lemma}

\begin{proof}
  Informally, the proof can be summarized by saying that even after
  Right gets $n$ free moves in $T'_n$, Left can still set all 3 of
  the pinwheel's superswitches to atomic values of Left's choice,
  thereby achieving any desired outcome $a$, $b$, or $c$.

  More formally, we first note that $a,b\leq \g{a,b|\g{a,b|a}} = G_2$,
  and therefore $\pinwheel(x,y,z)\leq \pinwheel(G_2,G_2,G_2) = T'_2$
  for all $x,y,z\in\s{a,b}$. By the identities in
  Figure~\ref{fig:pinwheel-atoms}, it follows that $a,b,c\leq T'_0$,
  and therefore also $a,b,c\tri T'_0$.

  We prove the lemma by induction. The base case is clear since
  $T_0=\bot$. For the induction step, consider $n>0$. To show $T_n
  \leq T'_n$, first consider any left option $x$ of $T_n$. Then
  $x\in\s{a,b,c}$, and therefore $x\tri T'_0\leq T'_n$ as
  desired. Next, consider any right option $H$ of $T'_n =
  \pinwheel(G_{n+2},G_{n+2},G_{n+2})$. Using the canonical form of
  $G_{n+2}$, there are exactly three such right options, namely
  $\pinwheel(G_{n+1},G_{n+2},G_{n+2})$,
  $\pinwheel(G_{n+2},G_{n+1},G_{n+2})$, and
  $\pinwheel(G_{n+2},G_{n+2},G_{n+1})$. In all three cases, it is the
  case that\linebreak $\pinwheel(G_{n+1},G_{n+1},G_{n+1}) \leq H$, or
  in other words, $T'_{n-1} \leq H$. By the induction hypothesis, we know
  $T_{n-1} \leq T'_{n-1}$, and since $T_{n-1}$ is a right option of
  $T_n$, it follows that $T_n\tri H$ as desired.
\end{proof}

\subsection{Application: No single context verifies connects-both templates}
\label{ssec:no-single-context}

We can use the Hex tripleswitches to prove something that was claimed
in Section~\ref{ssec:connects-both}, namely, that there is no single
Hex context for a 3-terminal position such that Black wins if and only
if the position has the connects-both property. Here, by a ``Hex
context for a 3-terminal position'', we mean any position $C$ (atomic
or not) on a Hex board with a ``hole'' into which a 3-terminal
position $G$ fits, as suggested by the following illustration:
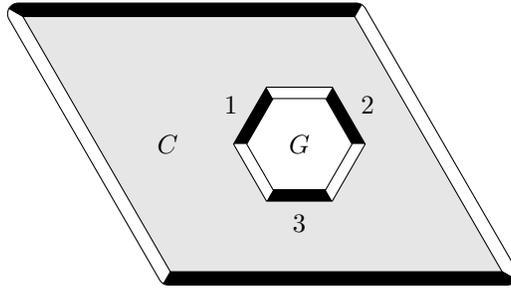
\[
\begin{hexboard}[scale=0.35]
  \rotation{-30}
  \begin{pgfonlayer}{edges}
    \fill[black!10]
    \coord(-10,-7) -- \coord(-10,9) -- \coord(8,9) -- \coord(8,-7) -- cycle
    (300:2) -- (240:2) -- (180:2) -- (120:2) -- (60:2) -- (0:2) -- cycle;
  \end{pgfonlayer}
  \edge[\nw](-10,-7)(8,-7)
  \edge[\ne](8,-7)(8,9)
  \edge[\se](8,9)(-10,9)
  \edge[\sw](-10,9)(-10,-7)
  \fill[black] (0:2) -- (0:2.5) -- (60:2.5) -- (60:2) -- cycle;
  \fill[black] (120:2) -- (120:2.5) -- (180:2.5) -- (180:2) -- cycle;
  \fill[black] (240:2) -- (240:2.5) -- (300:2.5) -- (300:2) -- cycle;
  \fill[white] (60:2) -- (60:2.5) -- (120:2.5) -- (120:2) -- cycle;
  \fill[white] (180:2) -- (180:2.5) -- (240:2.5) -- (240:2) -- cycle;
  \fill[white] (300:2) -- (300:2.5) -- (360:2.5) -- (360:2) -- cycle;
  \draw (0:2) -- (60:2) -- (120:2) -- (180:2) -- (240:2) -- (300:2) -- cycle;
  \draw (0:2.5) -- (60:2.5) -- (120:2.5) -- (180:2.5) -- (240:2.5) -- (300:2.5) -- cycle;
  \path (150:3) node {$1$};
  \path (30:3) node {$2$};
  \path (270:3) node {$3$};
  \path (0,0) node {$G$};
  \path (180:5) node {$C$};
\end{hexboard}
\]
If $C$ is such a context and $G$ is a 3-terminal position, we write
$G\plusctx C$ for the operation of plugging $G$ into the hole in the
context.

\begin{proposition}
  There is no Hex context $C$ with the property that for all
  3-terminal Hex positions $G$, we have $G\plusctx C \eq\top$ if and
  only if $G\eq\top$.
\end{proposition}

\begin{proof}
  We first observe that if $C$ is an \emph{atomic} context (i.e., if
  $C$ has no empty cells) and if $\top\plusctx C$ is winning for Black,
  then one is of $a\plusctx C$, $b\plusctx C$, or $c\plusctx C$ is
  already winning. To see this, consider any Black winning path in
  $\top\plusctx C$. If the path lies completely outside the region $G$,
  the claim is trivial and $\bot\plusctx C$ is already winning in this
  case. Otherwise, since the path must both enter and exit the region
  $G$, one of $G$'s terminals must be connected to the top board edge
  and one of $G$'s terminals must be connected to the bottom board
  edge. It follows that one of $a\plusctx C$, $b\plusctx C$, or
  $c\plusctx C$ is winning.

  Now suppose, for the purpose of obtaining a contradiction, that
  there is a Hex context $C$ with the claimed property. By assumption,
  Black has a second-player winning strategy in $\top\plusctx C$. Let
  $d$ be the depth of $C$, i.e., the number of empty cells in it. Let
  $T'_{d+2}$ be the Hex tripleswitch defined in
  Section~\ref{ssec:tripleswitches}. Recall that $T'_{d+2}<\top$. We claim
  that Black has a second-player winning strategy in $T'_{d+2}\plusctx C$,
  contradicting our assumption about $C$.

  By Lemma~\ref{lem:hex-tripleswitch}, we have $T_{d+2}\leq T'_{d+2}$,
  and therefore it is sufficient to show that Black has a
  second-player winning strategy in $T_{d+2}\plusctx C$. Black's
  strategy is as follows. If White plays in $C$, Black responds in $C$
  by following the second-player winning strategy for $\top\plusctx
  C$. If White plays in $T_{d+2}$, Black ignores White's move and
  makes an arbitrary move in $C$. Due to monotonicity, this move can
  only help Black in $C$, i.e., Black continues to have a
  second-player winning strategy in $\top\plusctx C$. In this way,
  after at most $d$ moves by each player, the context $C$ reaches an
  atomic position, say $C_0$. Meanwhile, Black has not yet made any
  moves in the $T_d$ component of the game, and White has made at most
  $d$ moves there, so that component's value is $T_k$ for some $k\geq
  2$.  In summary, at the end of this phase, the game's value is
  $T_k\plusctx C_0$.
  
  By assumption, $\top\plusctx C_0$ is a win for Black. By our first
  observation above, there is an atom $x\in\s{a,b,c}$ such that
  $x\plusctx C_0$ is a win for Black, i.e., $x\plusctx C_0\eq\top$.
  On the other hand, by properties of tripleswitches, we have $x \leq
  T_2\leq T_k$, and therefore $T_k\plusctx C_0\eq\top$. It follows
  that Black has a winning strategy from the current position, which
  finishes the proof.
\end{proof}

\section{A database of Hex realizable 3-terminal values}
\label{sec:database}

In Section~\ref{ssec:superswitch-realizable}, we used the fact that
the abstract game value $\g{a,b|a}$ is realizable as a specific
3-terminal Hex position. Similarly, in
Section~\ref{ssec:simpleswitch}, we used a specific Hex realization of
the value $\g{a,\g{\top|b}|a}$. Both of these Hex positions are
non-obvious, and the reader may wonder how we found them. More
generally: given a passable game value $G$ over the 3-terminal
outcome poset, what is a practical method for finding a Hex
realization of $G$?

While we do not currently know how to solve this problem in general,
we have created a tool that has proven useful in practice. Namely, we
created a database of more than a million Hex realizable 3-terminal
values.  In Section~\ref{ssec:database-create}, we will briefly
explain how the database was created, and in
Section~\ref{ssec:database-use}, we will illustrate how to use it.

\subsection{How the database was created}
\label{ssec:database-create}

When trying to create a list of Hex realizable values, the first idea
is to enumerate many small 3-terminal Hex positions and record their
values.  However, this idea does not work well. The problem is that
even if one restricts attention to positions that are small enough to
be efficiently evaluated (in our case, positions with up to
approximately 15 empty cells), the number of such positions is
astronomical, and most of them yield values that are either repetitive
or much too complicated to be of interest.

Instead, we started from a small initial set of realizable 3-terminal
values, and then repeatedly closed the set under symmetries, duals,
and the pinwheel construction of Definition~\ref{def:pinwheel}. After
each iteration, we reduced all values to canonical form, weeded out
duplicates, and removed values that were unreasonably complicated (we
arbitrarily chose to eliminate values whose canonical forms contained
more than 50 atoms). We do not claim the method to be exhaustive (it
is not clear whether all realizable 3-terminal values can be obtained
by repeated applications of pinwheel composition). Nevertheless, and
perhaps surprisingly, the method generated a very large number of
interesting Hex-realizable values with a reasonable computational
effort.

\subsection{How to use the database}
\label{ssec:database-use}

\begin{figure}
  \input{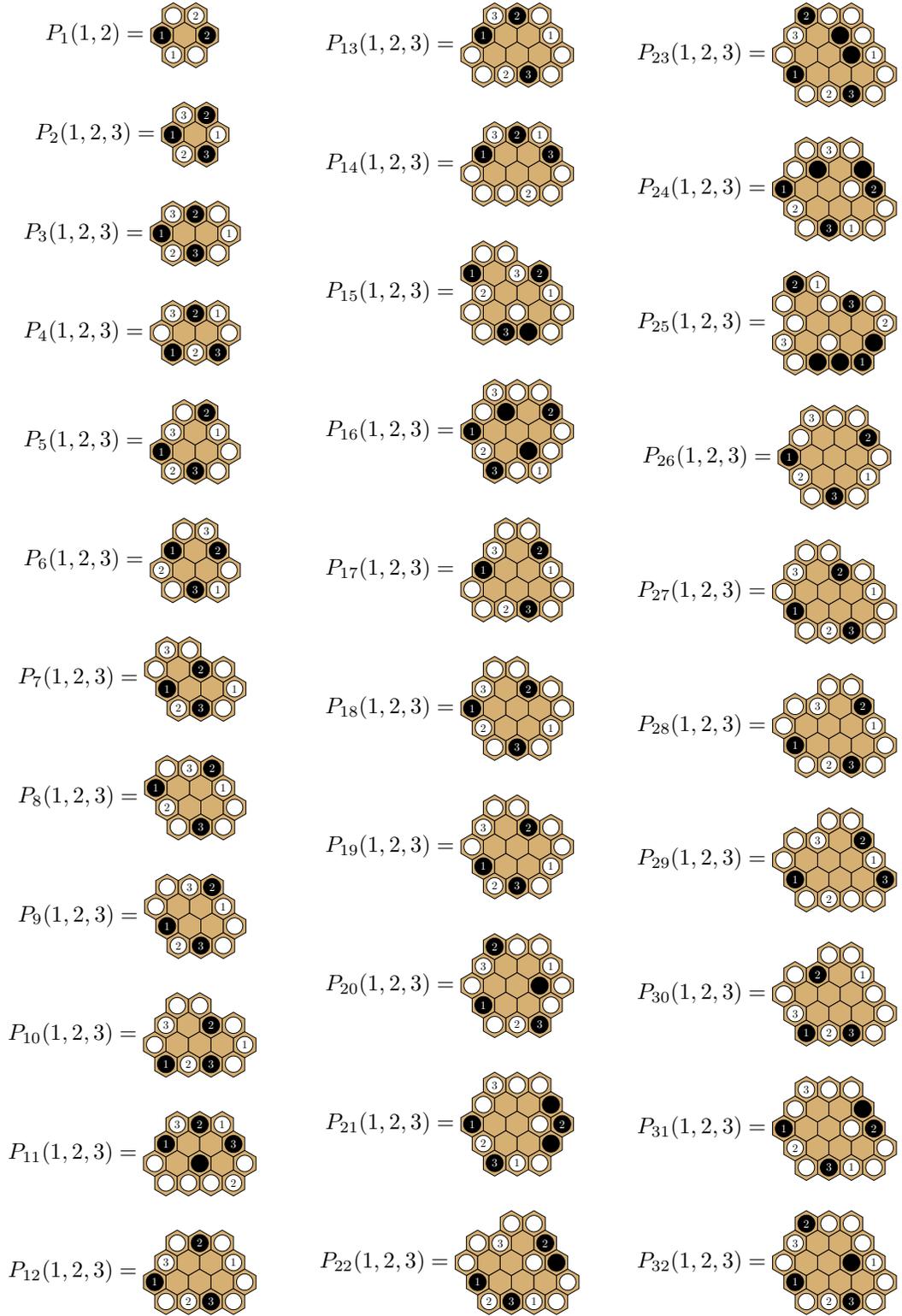}
  \caption{Some primary 3-terminal positions}
  \label{fig:primaries}
\end{figure}

The database is available from {\cite{database}}. It consists of two
files in a format that is readable by both humans and machines. The
file {\sf primaries.txt} contains a list of 122 so-called
\emph{primary positions}. The first 32 of these are shown in
Figure~\ref{fig:primaries}. We write $P_n$ for the $n$th primary
position, and $P\opp_n$ for its dual, which is obtained by switching
the roles of black and white. To facilitate taking duals, we have
labelled both the black terminals and the white terminals of the
primary positions. For example:
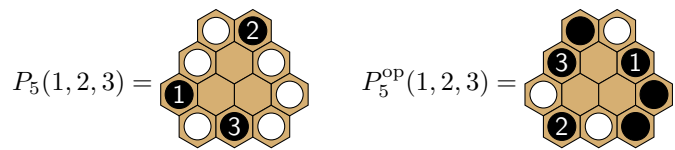
\[
P_5(1,2,3) =
\begin{hexboard}[baseline={($(current bounding box.center)-(0,1ex)$)},scale=0.7]
  \rotation{-30}
  \hex(0,2)\black(0,2)\sflabel{1}
  \hex(0,3)\white(0,3)
  \hex(1,1)\white(1,1)
  \hex(1,2)
  \hex(1,3)\black(1,3)\sflabel{3}
  \hex(2,0)\white(2,0)
  \hex(2,1)
  \hex(2,2)
  \hex(2,3)\white(2,3)
  \hex(3,0)\black(3,0)\sflabel{2}
  \hex(3,1)\white(3,1)
  \hex(3,2)\white(3,2)
\end{hexboard}
\qquad
P\opp_5(1,2,3) =
\begin{hexboard}[baseline={($(current bounding box.center)-(0,1ex)$)},scale=0.7]
  \rotation{-30}
  \hex(0,2)\white(0,2)
  \hex(0,3)\black(0,3)\sflabel{2}
  \hex(1,1)\black(1,1)\sflabel{3}
  \hex(1,2)
  \hex(1,3)\white(1,3)
  \hex(2,0)\black(2,0)
  \hex(2,1)
  \hex(2,2)
  \hex(2,3)\black(2,3)
  \hex(3,0)\white(3,0)
  \hex(3,1)\black(3,1)\sflabel{1}
  \hex(3,2)\black(3,2)
\end{hexboard}
\]
We write $P_n(x,y,z)$ for the analogous position whose terminals are
numbered $x,y,z$ instead of $1,2,3$. For example:
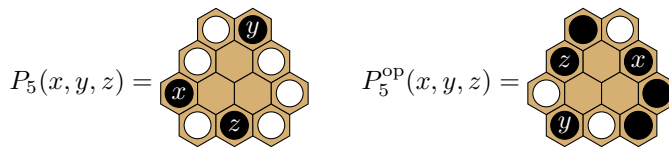
\[
P_5(x,y,z) =
\begin{hexboard}[baseline={($(current bounding box.center)-(0,1ex)$)},scale=0.7]
  \rotation{-30}
  \hex(0,2)\black(0,2)\label{$x$}
  \hex(0,3)\white(0,3)
  \hex(1,1)\white(1,1)
  \hex(1,2)
  \hex(1,3)\black(1,3)\label{$z$}
  \hex(2,0)\white(2,0)
  \hex(2,1)
  \hex(2,2)
  \hex(2,3)\white(2,3)
  \hex(3,0)\black(3,0)\label{$y$}
  \hex(3,1)\white(3,1)
  \hex(3,2)\white(3,2)
\end{hexboard}
\qquad
P\opp_5(x,y,z) =
\begin{hexboard}[baseline={($(current bounding box.center)-(0,1ex)$)},scale=0.7]
  \rotation{-30}
  \hex(0,2)\white(0,2)
  \hex(0,3)\black(0,3)\label{$y$}
  \hex(1,1)\black(1,1)\label{$z$}
  \hex(1,2)
  \hex(1,3)\white(1,3)
  \hex(2,0)\black(2,0)
  \hex(2,1)
  \hex(2,2)
  \hex(2,3)\black(2,3)
  \hex(3,0)\white(3,0)
  \hex(3,1)\black(3,1)\label{$x$}
  \hex(3,2)\black(3,2)
\end{hexboard}
\]
We also use the superscript $\times$ as a \emph{chirality} annotation,
indicating that the position should be mirrored, so that its terminals
are arranged counterclockwise instead of clockwise. Finally, the
notation $\Eq(x,y)$ denotes two terminals that are connected to each
other, and similarly $\Eq(x,y,z)$.  Thus:
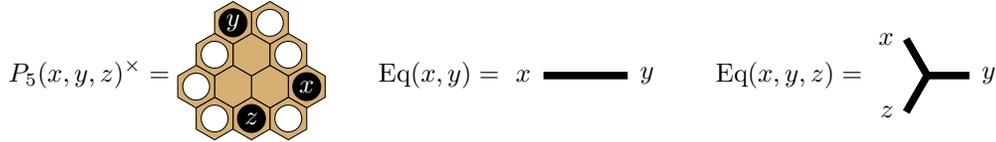
\[
P_5(x,y,z)\xx =
\begin{hexboard}[baseline={($(current bounding box.center)-(0,1ex)$)},scale=0.7,xscale=-1]
  \rotation{-30}
  \hex(0,2)\black(0,2)\label{$x$}
  \hex(0,3)\white(0,3)
  \hex(1,1)\white(1,1)
  \hex(1,2)
  \hex(1,3)\black(1,3)\label{$z$}
  \hex(2,0)\white(2,0)
  \hex(2,1)
  \hex(2,2)
  \hex(2,3)\white(2,3)
  \hex(3,0)\black(3,0)\label{$y$}
  \hex(3,1)\white(3,1)
  \hex(3,2)\white(3,2)
\end{hexboard}
\qquad
\Eq(x,y) = 
\begin{hexboard}[baseline={($(current bounding box.center)-(0,0.55ex)$)},scale=0.8]
  \rotation{-30}
  \draw[terminal] \coord(0,0) -- \coord(-1,0) node[left] {$x$};
  \draw[terminal] \coord(0,0) -- \coord(1,0) node[right] {$y$};
\end{hexboard}
\qquad
\Eq(x,y,z) = 
\begin{hexboard}[baseline={($(current bounding box.center)-(0,0.55ex)$)},scale=0.8]
  \rotation{-30}
  \draw[terminal] \coord(0,0) -- \coord(0,-1) node[left] {$x$};
  \draw[terminal] \coord(0,0) -- \coord(1,0) node[right] {$y$};
  \draw[terminal] \coord(0,0) -- \coord(-1,1) node[left] {$z$};
\end{hexboard}
\]
The file {\sf realizable-3terminal.txt} describes one Hex position per
line, in the format shown Figure~\ref{fig:realizable}. Each line
contains three pieces of information about a position: its number of
empty cells, its value in canonical form, and a description of the
position itself. The latter expresses each position as a composition
of primary positions.
\begin{figure}
  \[
  \begin{array}{@{}lll}
    0 & \quotemath{\bot} & G(1,2,3) = \Empty \\
    0 & \quotemath{\top} & G(1,2,3) = \Eq(1,2,3) \\
    0 & \quotemath{a} & G(1,2,3) = \Eq(2,3) \\
    0 & \quotemath{b} & G(1,2,3) = \Eq(1,3) \\
    0 & \quotemath{c} & G(1,2,3) = \Eq(1,2) \\
    1 & \quotemath{\g{a|\bot}} & G(1,2,3) = P_{1}(2,3) \\
    1 & \quotemath{\g{\top|a}} & G(1,2,3) = \Eq(2,3), P_{1}(1,2) \\
    2 & \quotemath{\g{\top|\g{a|\bot}}} & G(1,2,3) = P_{3}(1,2,3) \\
    2 & \quotemath{\g{\g{\top|b}|\bot}} & G(1,2,3) = P_{3}\opp(2,3,1) \\
    3 & \quotemath{\g{a,\g{\top|b}|\bot}} & G(1,2,3) = P_{5}(1,2,3) \\
    3 & \quotemath{\g{a,\g{\top|c}|\bot}} & G(1,2,3) = P_{5}(1,3,2)\xx \\
    5 & \quotemath{\g{b,\g{\top|a,\g{c|\bot}}|\bot}} & G(1,2,3) = P_{3}\opp(1,4,3), P_{5}(2,1,4)\xx \\
    6 & \quotemath{\g{a,\g{\top|b}|a}} & G(1,2,3) = P_{25}\opp(2,1,3)\xx \\
    12&  \quotemath{\g{a,b|a}} & G(1,2,3) = P_{1}(3,4), P_{3}\opp(4,5,3), P_{3}(2,4,6), P_{5}(2,1,6)\xx, P_{9}(2,5,4) \\
    14& \quotemath{\g{a,b|\g{a,b|\g{a|\bot}}}} & G(1,2,3) = P_{5}\opp(4,5,1), P_{15}\opp(4,6,2)\xx, P_{25}\opp(6,5,3)\xx \\
  \end{array}
  \]

  \caption{Excerpt from the 3-terminal database}
  \label{fig:realizable}
\end{figure}

For example, consider the second-to-last line of
Figure~\ref{fig:realizable}. It states that the game value $\g{a,b|a}$
is realized by the position
\[
G(1,2,3) =
P_{1}(3,4), P_{3}\opp(4,5,3), P_{3}(2,4,6), P_{5}(2,1,6)\xx, P_{9}(2,5,4).
\]
To decipher this notation, we first draw the five primary positions
indicated on the right-hand side, dualizing and mirroring them as
indicated by the ``op'' and ``$\times$'' annotations:
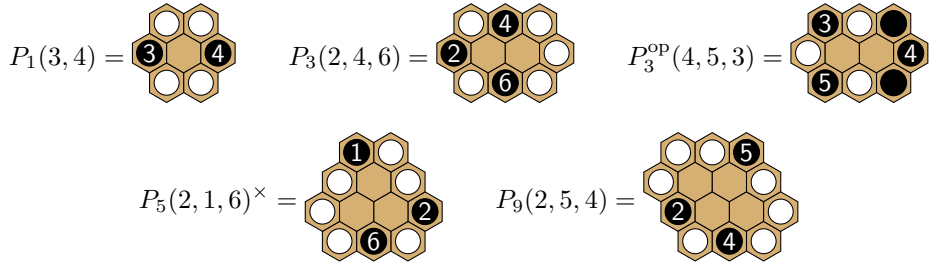
\[
P_{1}(3,4) = 
\begin{hexboard}[baseline={($(current bounding box.center)-(0,1ex)$)},scale=0.65]
  \rotation{-30}
  \hex(0,1)\black(0,1)\sflabel{3}
  \hex(0,2)\white(0,2)
  \hex(1,0)\white(1,0)
  \hex(1,1)
  \hex(1,2)\white(1,2)
  \hex(2,0)\white(2,0)
  \hex(2,1)\black(2,1)\sflabel{4}
\end{hexboard}
\qquad
P_{3}(2,4,6) = 
\begin{hexboard}[baseline={($(current bounding box.center)-(0,1ex)$)},scale=0.65]
  \rotation{-30}
  \hex(0,1)\black(0,1)\sflabel{2}
  \hex(0,2)\white(0,2)
  \hex(1,0)\white(1,0)
  \hex(1,1)
  \hex(1,2)\black(1,2)\sflabel{6}
  \hex(2,0)\black(2,0)\sflabel{4}
  \hex(2,1)
  \hex(2,2)\white(2,2)
  \hex(3,0)\white(3,0)
  \hex(3,1)\white(3,1)
\end{hexboard}
\qquad
P\opp_{3}(4,5,3) = 
\begin{hexboard}[baseline={($(current bounding box.center)-(0,1ex)$)},scale=0.65]
  \rotation{-30}
  \hex(0,1)\white(0,1)
  \hex(0,2)\black(0,2)\sflabel{5}
  \hex(1,0)\black(1,0)\sflabel{3}
  \hex(1,1)
  \hex(1,2)\white(1,2)
  \hex(2,0)\white(2,0)
  \hex(2,1)
  \hex(2,2)\black(2,2)
  \hex(3,0)\black(3,0)
  \hex(3,1)\black(3,1)\sflabel{4}
\end{hexboard}
\]
\[
P_{5}(2,1,6)\xx = 
\begin{hexboard}[baseline={($(current bounding box.center)-(0,1ex)$)},xscale=-1,scale=0.65]
  \rotation{-30}
  \hex(0,2)\black(0,2)\sflabel{2}
  \hex(0,3)\white(0,3)
  \hex(1,1)\white(1,1)
  \hex(1,2)
  \hex(1,3)\black(1,3)\sflabel{6}
  \hex(2,0)\white(2,0)
  \hex(2,1)
  \hex(2,2)
  \hex(2,3)\white(2,3)
  \hex(3,0)\black(3,0)\sflabel{1}
  \hex(3,1)\white(3,1)
  \hex(3,2)\white(3,2)
\end{hexboard}
\qquad
P_{9}(2,5,4) = 
\begin{hexboard}[baseline={($(current bounding box.center)-(0,1ex)$)},scale=0.65]
  \rotation{-30}
  \hex(0,1)\white(0,1)
  \hex(0,2)\black(0,2)\sflabel{2}
  \hex(0,3)\white(0,3)
  \hex(1,0)\white(1,0)
  \hex(1,1)
  \hex(1,2)
  \hex(1,3)\black(1,3)\sflabel{4}
  \hex(2,0)\white(2,0)
  \hex(2,1)
  \hex(2,2)
  \hex(2,3)\white(2,3)
  \hex(3,0)\black(3,0)\sflabel{5}
  \hex(3,1)\white(3,1)
  \hex(3,2)\white(3,2)
\end{hexboard}
\]
Next, we connect like-numbered terminals and use $1$, $2$, and $3$ as
the exterior terminals of the position $G(1,2,3)$, like this:
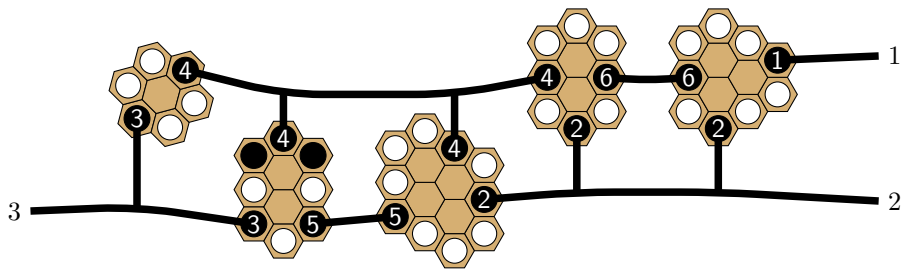
\[
\scalebox{1}{$
  \begin{hexboard}[baseline={($(current bounding box.center)-(0,1ex)$)},scale=-0.65]
    \begin{scope}[shift={(1.5cm,0.35cm)},xscale=-1]
      \rotation{-120}
      \hex(0,2)\black(0,2)\sflabel{2}
      \hex(0,3)\white(0,3)
      \hex(1,1)\white(1,1)
      \hex(1,2)
      \hex(1,3)\black(1,3)\sflabel{6}
      \hex(2,0)\white(2,0)
      \hex(2,1)
      \hex(2,2)
      \hex(2,3)\white(2,3)
      \hex(3,0)\black(3,0)\sflabel{1}
      \hex(3,1)\white(3,1)
      \hex(3,2)\white(3,2)
      \node[circle, minimum size=0.25cm, inner sep=0] (p5x-2) at \coord(0,2) {};
      \node[circle, minimum size=0.25cm, inner sep=0] (p5x-6) at \coord(1,3) {};
      \node[circle, minimum size=0.25cm, inner sep=0] (p5x-1) at \coord(3,0) {};
    \end{scope}
    \begin{scope}[shift={(5cm,0cm)}]
      \rotation{-120}
      \hex(0,1)\black(0,1)\sflabel{2}
      \hex(0,2)\white(0,2)
      \hex(1,0)\white(1,0)
      \hex(1,1)
      \hex(1,2)\black(1,2)\sflabel{6}
      \hex(2,0)\black(2,0)\sflabel{4}
      \hex(2,1)
      \hex(2,2)\white(2,2)
      \hex(3,0)\white(3,0)
      \hex(3,1)\white(3,1)
      \node[circle, minimum size=0.25cm, inner sep=0] (p3-2) at \coord(0,1) {};
      \node[circle, minimum size=0.25cm, inner sep=0] (p3-6) at \coord(1,2) {};
      \node[circle, minimum size=0.25cm, inner sep=0] (p3-4) at \coord(2,0) {};
    \end{scope}
    \begin{scope}[shift={(7.5cm,2.5cm)}]
      \rotation{-60}
      \hex(0,1)\white(0,1)
      \hex(0,2)\black(0,2)\sflabel{2}
      \hex(0,3)\white(0,3)
      \hex(1,0)\white(1,0)
      \hex(1,1)
      \hex(1,2)
      \hex(1,3)\black(1,3)\sflabel{4}
      \hex(2,0)\white(2,0)
      \hex(2,1)
      \hex(2,2)
      \hex(2,3)\white(2,3)
      \hex(3,0)\black(3,0)\sflabel{5}
      \hex(3,1)\white(3,1)
      \hex(3,2)\white(3,2)
      \node[circle, minimum size=0.25cm, inner sep=0] (p9-2) at \coord(0,2) {};
      \node[circle, minimum size=0.25cm, inner sep=0] (p9-4) at \coord(1,3) {};
      \node[circle, minimum size=0.25cm, inner sep=0] (p9-5) at \coord(3,0) {};
    \end{scope}
    \begin{scope}[shift={(11cm,2.3cm)}]
      \rotation{-120}
      \hex(0,1)\white(0,1)
      \hex(0,2)\black(0,2)\sflabel{5}
      \hex(1,0)\black(1,0)\sflabel{3}
      \hex(1,1)
      \hex(1,2)\white(1,2)
      \hex(2,0)\white(2,0)
      \hex(2,1)
      \hex(2,2)\black(2,2)
      \hex(3,0)\black(3,0)
      \hex(3,1)\black(3,1)\sflabel{4}
      \node[circle, minimum size=0.25cm, inner sep=0] (p3op-5) at \coord(0,2) {};
      \node[circle, minimum size=0.25cm, inner sep=0] (p3op-3) at \coord(1,0) {};
      \node[circle, minimum size=0.25cm, inner sep=0] (p3op-4) at \coord(3,1) {};
    \end{scope}
    \begin{scope}[shift={(13.5cm,0cm)}]
      \rotation{195}
      \hex(0,1)\black(0,1)\sflabel{3}
      \hex(0,2)\white(0,2)
      \hex(1,0)\white(1,0)
      \hex(1,1)
      \hex(1,2)\white(1,2)
      \hex(2,0)\white(2,0)
      \hex(2,1)\black(2,1)\sflabel{4}
      \node[circle, minimum size=0.25cm, inner sep=0] (p1-3) at \coord(0,1) {};
      \node[circle, minimum size=0.25cm, inner sep=0] (p1-4) at \coord(2,1) {};
    \end{scope}
    \node (t1) at (-1.5cm,-0.8cm) {$1$};
    \node (t2) at (-1.5cm,2.2cm) {$2$};
    \node (t3) at (16.5cm,2.4cm) {$3$};
    \coordinate (l2) at (0cm,2.0cm);
    \coordinate (l3) at (0,2.3cm);
    \coordinate (l4) at (0cm,0cm);
    \pgfsetcornersarced{\pgfpoint{0.5cm}{0.5cm}}
    \draw [terminal] (p3op-3) -- (l3 -| p1-3) -- (t3);
    \draw [terminal] (l3 -| p1-3) -- (p1-3);
    \draw [terminal] (p3-4) -- (l4 -| p9-4) -- (l4 -| p3op-4) -- (p1-4);
    \draw [terminal] (l4 -| p9-4) -- (p9-4);
    \draw [terminal] (l4 -| p3op-4) -- (p3op-4);
    \draw [terminal] (p9-5) -- (p3op-5);
    \draw [terminal] (t2) -- (l2 -| p5x-2) -- (l2 -| p3-2) -- (p9-2);
    \draw [terminal] (l2 -| p5x-2) -- (p5x-2);
    \draw [terminal] (l2 -| p3-2) -- (p3-2);
    \draw [terminal] (t1) -- (p5x-1);
    \draw [terminal] (p5x-6) to[out=5, in=175] (p3-6);
  \end{hexboard}
  $}
\]
Finally, if desired, we can manually try to fit the position into a
small amount of space. The following is a compact representation of
the value $\g{a,b|a}$. We have numbered the internal terminals to make
it easier to compare the position with the one above; note that
terminal 5 is not a stone, but a direct connection between adjacent
cells.  This realization of the superswitch is in fact the one that we
first saw in Section~\ref{ssec:superswitch-realizable}.
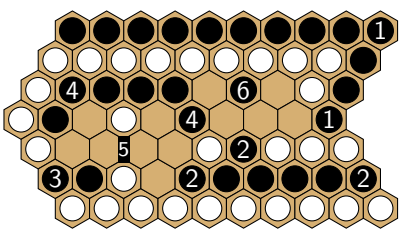
\[
\begin{hexboard}[scale=0.65]
  \rotation{-30}
  \foreach\i in {4,...,12} {\hex(\i,-1)\black(\i,-1)}
  \hex(13,-1)\black(13,-1)\sflabel{1}
  \hex(1,2)\white(1,2)
  \hex(1,3)\white(1,3)
  \hex(1,4)\black(1,4)\sflabel{3}
  \hex(1,5)\white(1,5)
  \hex(2,1)\white(2,1)
  \hex(2,2)\black(2,2)
  \hex(2,3)
  \hex(2,4)\black(2,4)
  \hex(2,5)\white(2,5)
  \hex(3,0)\white(3,0)
  \hex(3,1)\black(3,1)\sflabel{4}
  \hex(3,2)
  \hex(3,3)
  \hex(3,4)\white(3,4)
  \hex(3,5)\white(3,5)
  \hex(4,0)\white(4,0)
  \hex(4,1)\black(4,1)
  \hex(4,2)\white(4,2)
  \node[fill=black, inner sep=0mm, minimum height=3.5mm] at \coord(3.5,3) {{\footnotesize\color{white}\sf 5}};
  \hex(4,3)
  \hex(4,4)
  \hex(4,5)\white(4,5)
  \hex(5,0)\white(5,0)
  \hex(5,1)\black(5,1)
  \hex(5,2)
  \hex(5,3)
  \hex(5,4)\black(5,4)\sflabel{2}
  \hex(5,5)\white(5,5)
  \hex(6,0)\white(6,0)
  \hex(6,1)\black(6,1)
  \hex(6,2)\black(6,2)\sflabel{4}
  \hex(6,3)\white(6,3)
  \hex(6,4)\black(6,4)
  \hex(6,5)\white(6,5)
  \hex(7,0)\white(7,0)
  \hex(7,1)
  \hex(7,2)
  \hex(7,3)\black(7,3)\sflabel{2}
  \hex(7,4)\black(7,4)
  \hex(7,5)\white(7,5)
  \hex(8,0)\white(8,0)
  \hex(8,1)\black(8,1)\sflabel{6}
  \hex(8,2)
  \hex(8,3)\white(8,3)
  \hex(8,4)\black(8,4)
  \hex(8,5)\white(8,5)
  \hex(9,0)\white(9,0)
  \hex(9,1)
  \hex(9,2)
  \hex(9,3)\white(9,3)
  \hex(9,4)\black(9,4)
  \hex(9,5)\white(9,5)
  \hex(10,0)\white(10,0)
  \hex(10,1)\white(10,1)
  \hex(10,2)\black(10,2)\sflabel{1}
  \hex(10,3)\white(10,3)
  \hex(10,4)\black(10,4)\sflabel{2}
  \hex(10,5)\white(10,5)
  \hex(11,0)\white(11,0)
  \hex(11,1)\black(11,1)
  \hex(12,0)\black(12,0)
\end{hexboard}
\]
As another example, consider the realization of $\g{a|\g{\top|b}|a}$.
Figure~\ref{fig:realizable} states that this value is realized by the
Hex position $G(1,2,3)=P\opp_{25}(2,1,3)\xx$. Looking up the primary
position $P_{25}$ in Figure~\ref{fig:primaries} and applying the
indicated transformations, we obtain the realization
\[
G(1,2,3) = P\opp_{25}(2,1,3)\xx = ~~
\begin{hexboard}[baseline={($(current bounding box.center)-(0,1ex)$)},scale=0.65,xscale=-1]
  \rotation{90}
  \hex(0,3)\black(0,3)\sflabel{3}
  \hex(0,4)\black(0,4)
  \hex(1,1)\black(1,1)
  \hex(1,2)\black(1,2)
  \hex(1,3)
  \hex(1,4)\white(1,4)
  \hex(2,0)\white(2,0)\
  \hex(2,1)
  \hex(2,2)
  \hex(2,3)\black(2,3)
  \hex(2,4)\white(2,4)
  \hex(3,0)\black(3,0)\sflabel{2}
  \hex(3,1)\black(3,1)
  \hex(3,2)
  \hex(3,3)
  \hex(3,4)\white(3,4)
  \hex(4,1)\white(4,1)
  \hex(4,2)
  \hex(4,3)\white(4,3)
  \hex(5,1)\black(5,1)
  \hex(5,2)\black(5,2)\sflabel{1}
\end{hexboard}
\]
Modulo some padding, this is exactly the simpleswitch of
Section~\ref{ssec:simpleswitch}.

\begin{remark}
  Because of Theorem~\ref{thm:hex-universal}, we know that every Hex
  position can be decomposed into the single-cell primaries $P_1$ and
  $P_2$. Therefore, the remaining 120 primary positions are not
  strictly speaking necessary. However, using these additional
  primaries greatly decreases the size of the database, while also
  making it more useable.
\end{remark}

\begin{remark}\label{rem:corners-and-forks}
  From Section~\ref{ssec:corners-and-forks}, we know that the outcome
  posets for some kinds of 3-terminal regions, such as corners and
  forks, are proper quotients of the outcome poset of a generic
  3-terminal region. This means that distinct 3-terminal positions may
  become equivalent when regarded as corners or forks. We can
  therefore take our database of Hex-realizable 3-terminal values and
  extract smaller databases of Hex-realizable corners and forks. The
  1382388 distinct Hex-realizable 3-terminal values in our database
  reduce to 369291 Hex-realizable corners and 40310 Hex-realizable
  forks.
\end{remark}

\section{Application: Verifying pivoting templates}
\label{sec:pivoting}

In addition to edge templates, which guarantee that one or more stones
can be connected to the relevant edge, one can also consider regions
with other related properties. An example of this is a \emph{pivoting
template}. Informally, a pivoting template is a region that includes a
distinguished black stone $A$ and an empty cell $B$, such that Black
can continuously threaten to connect $A$ to the edge until the point
where Black either connects $A$ to the edge or occupies $B$ and
connects $B$ to the edge. Moreover, to be considered a template, the
region should also be minimal with this property.

The following is an example of a pivoting template:
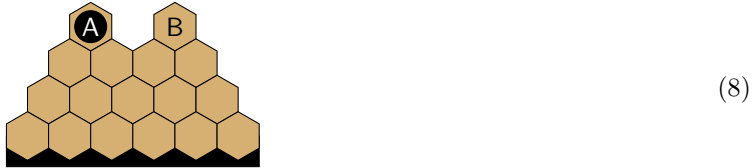
\begin{equation}\label{eqn:pivoting-example}
  \begin{hexboard}[baseline={($(current bounding box.center)-(0,1ex)$)},scale=0.8]
    \rotation{-30}
    \template(6,4)
    \foreach \i in {4,6} {\hex(\i,1)}
    \foreach \i in {3,...,6} {\hex(\i,2)}
    \foreach \i in {2,...,6} {\hex(\i,3)}
    \foreach \i in {1,...,6} {\hex(\i,4)}
    \black(4,1)\sflabel{A}
    \cell(6,1)\sflabel{B}
  \end{hexboard}
\end{equation}
Depending on White's moves, Black may play as follows:
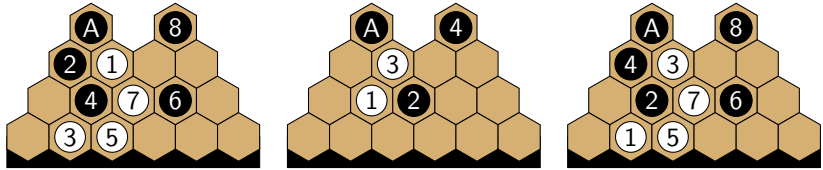
\[
\begin{hexboard}[scale=0.8]
  \rotation{-30}
  \template(6,4)
  \foreach \i in {4,6} {\hex(\i,1)}
  \foreach \i in {3,...,6} {\hex(\i,2)}
  \foreach \i in {2,...,6} {\hex(\i,3)}
  \foreach \i in {1,...,6} {\hex(\i,4)}
  \black(4,1)\sflabel{A}
  \cell(6,1)
  \white(4,2)\sflabel{1}
  \black(3,2)\sflabel{2}
  \white(2,4)\sflabel{3}
  \black(3,3)\sflabel{4}
  \white(3,4)\sflabel{5}
  \black(5,3)\sflabel{6}
  \white(4,3)\sflabel{7}
  \black(6,1)\sflabel{8}
\end{hexboard}
\quad
\begin{hexboard}[scale=0.8]
  \rotation{-30}
  \template(6,4)
  \foreach \i in {4,6} {\hex(\i,1)}
  \foreach \i in {3,...,6} {\hex(\i,2)}
  \foreach \i in {2,...,6} {\hex(\i,3)}
  \foreach \i in {1,...,6} {\hex(\i,4)}
  \black(4,1)\sflabel{A}
  \cell(6,1)
  \white(3,3)\sflabel{1}
  \black(4,3)\sflabel{2}
  \white(4,2)\sflabel{3}
  \black(6,1)\sflabel{4}
\end{hexboard}
\quad
\begin{hexboard}[scale=0.8]
  \rotation{-30}
  \template(6,4)
  \foreach \i in {4,6} {\hex(\i,1)}
  \foreach \i in {3,...,6} {\hex(\i,2)}
  \foreach \i in {2,...,6} {\hex(\i,3)}
  \foreach \i in {1,...,6} {\hex(\i,4)}
  \black(4,1)\sflabel{A}
  \cell(6,1)
  \white(2,4)\sflabel{1}
  \black(3,3)\sflabel{2}
  \white(4,2)\sflabel{3}
  \black(3,2)\sflabel{4}
  \white(3,4)\sflabel{5}
  \black(5,3)\sflabel{6}
  \white(4,3)\sflabel{7}
  \black(6,1)\sflabel{8}
\end{hexboard}
\]
In all cases, Black is threatening to connect $A$ to the edge (meaning
that if White passes, Black could connect $A$ to the edge in a single
move), until Black occupies $B$.

\subsection{Pivoting forks}\label{ssec:pivoting-forks}

We can use combinatorial game theory to formalize the concept of
pivoting. Consider the outcome poset $P_2 = \s{\bot,a,b,\top}$, where
$a$ and $b$ are incomparable. Slightly generalizing the terminology
from Section~\ref{ssec:corners-and-forks}, we will use the term
\emph{fork} for any game over this outcome poset (even for regions
that are not connected to white board edges). Informally, we want to
say that a fork is \emph{pivoting} if Left can continuously threaten
$a$ until the point where either $a$ or $b$ is achieved. To make this
more formal, we need to interpret the phrases ``achieve'',
``threaten'', and ``continuously until'' in combinatorial game theory
terms. We say that a game $G$ \emph{achieves} $b$ if $b\leq G$, and
that $G$ \emph{threatens} $a$ if $a\tri G$. The former means that Left
can guarantee outcome at least $b$, going second. The latter means
that Left can guarantee outcome at least $a$, going first. By
``continuously threaten until'', we mean that $G$ threatens $a$, and
for every right move there is a left response that reestablishes the
threat, until $b$ is achieved. The following recursive definition
makes this precise. Recall from Section~\ref{sec:background-cgt} that
$G^{R(L)}$ means a left option of $G^R$ when $G^R$ is composite, or
$G^R$ itself when $G^R$ is atomic.

\begin{definition}[Pivoting fork]\label{def:pivoting-formal}
  A fork $G$ is called \emph{pivoting} if
  \begin{itemize}
  \item $b\leq G$, or
  \item $a\tri G$ and for all $G^R$, there exists some $G^{R(L)}$ that
    is pivoting.
  \end{itemize}
\end{definition}

\begin{lemma}\label{lem:pivoting-monotone}
  If $G$ is pivoting and $G\leq H$, then $H$ is also pivoting.
\end{lemma}

\begin{proof}
  For the proof, it is convenient to say that $G$ is
  \emph{pre-pivoting} if there exists $G^{(L)}$ such that $G^{(L)}$ is
  pivoting. We then prove the following three statements by
  simultaneous induction: (a) $G$ pivoting and $G\leq H$ implies $H$
  pivoting; (b) $G$ pivoting and $G\tri H$ implies $H$ pre-pivoting;
  (c) $G$ pre-pivoting and $G\leq H$ implies $H$ pre-pivoting. With
  this setup, the rest of the proof is then routine.
\end{proof}

The following proposition characterizes pivoting forks in two
different ways. Recall that $G\plusj H$ denotes the juxtaposition
operation, which was defined in Section~\ref{ssec:distinguish} for
3-terminal positions, but also makes sense for forks. It has the
following addition table:
\[
\begin{array}{c|cccc}
  \plusj & \bot & a & b & \top \\\hline
  \bot & \bot & \bot & \bot & \bot \\
  a & \bot & \top & \bot & \top \\
  b & \bot & \bot & \top & \top \\
  \top & \bot & \top & \top & \top \\
\end{array}
\]

\begin{proposition}\label{prop:pivoting}
  For a passable fork $G$, the following are equivalent:
  \begin{enumerate}\alphalabels
  \item $G$ is pivoting;
  \item $\top \leq G \plusj \g{a,b|\g{a,b|\g{a|\bot}}}$;
  \item $b\leq G$ or $\g{a|a,\g{b|\bot}}\leq G$.
  \end{enumerate}
\end{proposition}

\begin{proof}
  $(a)\imp (b)$: Assume $G$ is pivoting. Let $C =
  \g{a,b|\g{a,b|\g{a|\bot}}}$. We must show that Left has a second
  player winning strategy in $G \plusj C$.  We show this by induction
  on $G$. If $b\leq G$, the claim holds because we can verify by
  direct calculation that $\top\leq b\plusj C$. Otherwise, suppose
  Right moves to $G^R\plusj C$. By the pivoting assumption, Left has a
  response $G^{R(L)}$ (possibly passing) that is again pivoting. Then
  Left wins the game by the induction hypothesis. Now suppose Right
  moves to $G\plusj C^R = G\plusj \g{a,b|\g{a|\bot}}$.  By the
  pivoting assumption, either $b\leq G$, in which case Left wins by
  moving to $G\plusj b$, or $a\tri G$, in which case $a\plusj C^R \tri
  G\plusj C^R$. Since $a\plusj C^R = a \plusj \g{a,b|\g{a|\bot}}$ is
  easily seen to be a second-player win for Left, it follows that
  $G\plusj C^R$ is a first-player win.

  $(b)\imp (c)$: Let $C=\g{a,b|\g{a,b|\g{a|\bot}}}$, and assume $\top
  \leq G \plusj C$.  By the fundamental theorem of monotone games, we
  can assume without loss of generality that $G$ is monotone.

  We proceed by induction on $G$. We first claim that $a\tri G$ or
  $b\leq G$. Let $C^R = \g{a,b|\g{a|\bot}}$. By definition of $\leq$,
  we have $\top\tri G\plusj C^R$. By definition of $\tri$, there is
  some $(G\plusj C^R)^L$ with $\top\leq (G\plusj C^R)^L$. By
  definition of the sum, there are two cases: $\top\leq G^L \plusj
  C^R$ or $\top\leq G\plusj C^{RL}$.  In the first case, we have
  $\top\tri G^L\plusj C^{RR}$, i.e., $\top\tri G^L \plusj \g{a|\bot}$,
  i.e., Left has a winning move in $G^L\plusj \g{a|\bot}$. This move
  must be to $a$ in the right component, since any other move is
  losing due to Right's response of $\bot$. Therefore $\top\leq
  G^L\plusj a$. This implies $a\leq G^L$, hence $a\tri G$. In the
  second case, we have $\top\leq G\plusj C^{RL}$. Since $C^{RL}$ is
  either $a$ or $b$, it follows that $\top\leq G\plusj a$ or $\top\leq
  G\plusj b$. In the first case, $a\leq G$, which implies $a\tri
  G$. In the second case, $b\leq G$. In all cases, we are done proving
  the first claim.

  To finish proving $(c)$, we must prove $b\leq G$ or
  $\g{a|a,\g{b|\bot}}\leq G$. In case $b\leq G$, we are done, so
  assume $b\nleq G$. We must show $\g{a|a,\g{b|\bot}}\leq G$.  By
  definition of $\leq$, we must show two things.
  \begin{itemize}
  \item We must show that $a\tri G$, but this holds by the first claim
    (and our assumption that $b\nleq G$).
  \item We must also show that every $G^{(R)}$ satisfies
    $\g{a|a,\g{b|\bot}}\tri G^{(R)}$. So consider some $G^{(R)}$.  From
    assumption $(b)$, we get $\top\tri G^{(R)}\plusj C$. By definition of
    $\tri$, there exists some $(G^{(R)}\plusj C)^L$ such that $\top\leq
    (G^{(R)}\plusj C)^L$. By definition of the sum, $(G^{(R)}\plusj C)^L$ is
    either $G^{(R)}\plusj C^L$ or $G^{{(R)}L}\plusj C$.
    \begin{itemize}
    \item Case 1: $\top\leq G^{(R)}\plusj C^L$. Note that $C^L$ is
      either $a$ or $b$, and therefore either $\top\leq
      G^{(R)}\plusj a$ or $\top\leq G^{(R)}\plusj b$. This implies
      $a\leq G^{(R)}$ or $b\leq G^{(R)}$.
      \begin{itemize}
      \item Case 1.1: $a\leq G^{(R)}$, therefore
        $\g{a|a,\g{b|\bot}}\tri G^{(R)}$ directly from the definition
        of $\tri$.
      \item Case 1.2: $b\leq G^{(R)}$. But we had assumed that $G$ is
        monotone, therefore $b\leq G^{(R)}\leq G$, contradicting our
        assumption that $b\nleq G$.
      \end{itemize}
    \item Case 2: $\top\leq G^{{(R)}L}\plusj C$. Then by the
      induction hypothesis, we have $b\leq G^{{(R)}L}$ or
      $\g{a|a,\g{b|\bot}}\leq G^{{(R)}L}$.
      \begin{itemize}
      \item Case 2.1: $b\leq G^{{(R)}L}$. This implies $b\tri G^{(R)}$,
        hence $\g{b|\bot}\leq G^{(R)}$, hence
        $\g{a|a,\g{b|\bot}}\tri G^{(R)}$ as claimed.
      \item Case 2.2: $\g{a|a,\g{b|\bot}}\leq G^{{(R)}L}$. This
        directly implies $\g{a|a,\g{b|\bot}}\tri G^{{(R)}}$ by the
        definition of $\tri$.
      \end{itemize}
    \end{itemize}
  \end{itemize}

  $(c)\imp (a)$: Clearly both $b$ and $\g{a|a,\g{b|\bot}}$ are
  pivoting, and therefore, by Lemma~\ref{lem:pivoting-monotone}, so is
  any $G$ that is at least as good as one of those two.
\end{proof}

\subsection{Verifying pivoting templates}\label{ssec:verifying-pivoting}

Proposition~\ref{prop:pivoting} suggests a method for verifying the
validity of pivoting templates. To prove that some Hex position $G$
has the pivoting property, by Proposition~\ref{prop:pivoting}(b), all
we need to do is juxtapose it with a Hex position of value $C =
\g{a,b|\g{a,b|\g{a|\bot}}}$ and use a Hex solver to check whether
Black has a second-player winning strategy in $G\plusj C$. But can we
find a Hex fork with value $C$? Here, the database of
Section~\ref{sec:database} is helpful. As shown in
Figure~\ref{fig:realizable}, the value $C$ has a Hex realization with
14 empty cells as $G(1,2,3) = P_{5}\opp(4,5,1), P_{15}\opp(4,6,2)\xx,
P_{25}\opp(6,5,3)\xx$.  Following the procedure of
Section~\ref{ssec:database-use}, we find that this Hex realization is
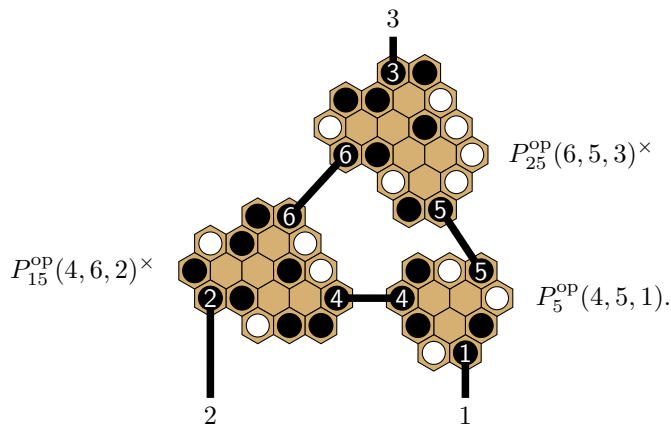
\[
\begin{hexboard}[baseline={($(current bounding box.center)-(0,1ex)$)},scale=0.6,rotate=90]
  \begin{scope}[shift={(0cm,0cm)}]
    \rotation{60}
    \hex(0,2)\white(0,2)
    \hex(0,3)\black(0,3)\sflabel{5}
    \hex(1,1)\black(1,1)
    \hex(1,2)
    \hex(1,3)\white(1,3)
    \hex(2,0)\black(2,0)\sflabel{1}
    \hex(2,1)
    \hex(2,2)
    \hex(2,3)\black(2,3)
    \hex(3,0)\white(3,0)
    \hex(3,1)\black(3,1)
    \hex(3,2)\black(3,2)\sflabel{4}
    \node[circle, minimum size=0.25cm, inner sep=0] (p5-1) at \coord(2,0) {};
    \node[circle, minimum size=0.25cm, inner sep=0] (p5-4) at \coord(3,2) {};
    \node[circle, minimum size=0.25cm, inner sep=0] (p5-5) at \coord(0,3) {};
  \end{scope}
  \begin{scope}[shift={(1.82cm,6cm)},xscale=-1]
    \rotation{-60}
    \hex(0,3)\black(0,3)
    \hex(0,4)\black(0,4)\sflabel{6}
    \hex(1,1)\white(1,1)
    \hex(1,2)\black(1,2)
    \hex(1,3)
    \hex(1,4)\white(1,4)
    \hex(2,0)\black(2,0)
    \hex(2,1)
    \hex(2,2)
    \hex(2,3)\black(2,3)
    \hex(2,4)\white(2,4)
    \hex(3,0)\black(3,0)\sflabel{2}
    \hex(3,1)\black(3,1)
    \hex(3,2)
    \hex(3,3)
    \hex(3,4)\black(3,4)\sflabel{4}
    \hex(4,1)\white(4,1)
    \hex(4,2)\black(4,2)
    \hex(4,3)\black(4,3)
    \node[circle, minimum size=0.25cm, inner sep=0] (p15-2) at \coord(3,0) {};
    \node[circle, minimum size=0.25cm, inner sep=0] (p15-4) at \coord(3,4) {};
    \node[circle, minimum size=0.25cm, inner sep=0] (p15-6) at \coord(0,4) {};
    \node[circle, minimum size=0.25cm, inner sep=0] (p15-l) at \coord(2,0) {};
  \end{scope}
  \begin{scope}[shift={(5cm,3cm)},xscale=-1]
    \rotation{-60}
    \hex(0,3)\black(0,3)\sflabel{3}
    \hex(0,4)\black(0,4)
    \hex(1,1)\black(1,1)
    \hex(1,2)\black(1,2)
    \hex(1,3)
    \hex(1,4)\white(1,4)
    \hex(2,0)\white(2,0)
    \hex(2,1)
    \hex(2,2)
    \hex(2,3)\black(2,3)
    \hex(2,4)\white(2,4)
    \hex(3,0)\black(3,0)\sflabel{6}
    \hex(3,1)\black(3,1)
    \hex(3,2)
    \hex(3,3)
    \hex(3,4)\white(3,4)
    \hex(4,1)\white(4,1)
    \hex(4,2)
    \hex(4,3)\white(4,3)
    \hex(5,1)\black(5,1)
    \hex(5,2)\black(5,2)\sflabel{5}
    \node[circle, minimum size=0.25cm, inner sep=0] (p25-3) at \coord(0,3) {};
    \node[circle, minimum size=0.25cm, inner sep=0] (p25-5) at \coord(5,2) {};
    \node[circle, minimum size=0.25cm, inner sep=0] (p25-6) at \coord(3,0) {};
  \end{scope}
  \node at (0.5cm,-1cm -| p5-4) [right] {$P_{5}\opp(4,5,1)$.};
  \node at (1cm,7cm -| p15-l) [left] {$P_{15}\opp(4,6,2)\xx$};
  \node at (4cm,-0.4cm -| p25-6) [right] {$P_{25}\opp(6,5,3)\xx$};
  \node (t1) at (-2cm,0cm |- p5-1) {$1$};
  \node (t2) at (-2cm,0cm |- p15-2) {$2$};
  \node (t3) at (6.8cm,0cm |- p25-3) {$3$};
  \pgfsetcornersarced{\pgfpoint{0.5cm}{0.5cm}}
  \draw [terminal] (t1) -- (p5-1);
  \draw [terminal] (t2) -- (p15-2);
  \draw [terminal] (t3) -- (p25-3);
  \draw [terminal] (p5-5) -- (p25-5);
  \draw [terminal] (p5-4) -- (p15-4);
  \draw [terminal] (p15-6) -- (p25-6);
\end{hexboard}
\]
The following is a more compact representation of this
position. Note that according to our convention of terminal numbering,
outcome $a$ means that terminals 3 and 2 are connected, and outcome
$b$ means that terminals 3 and 1 are connected.
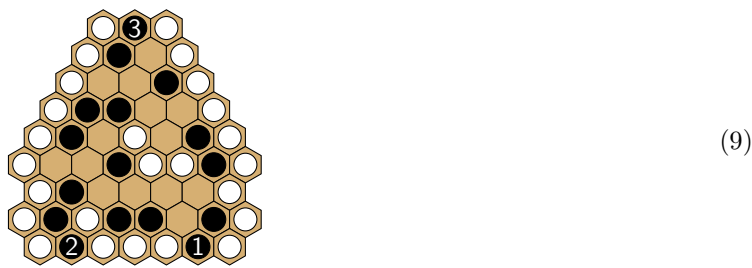
\begin{equation}\label{eqn:co-pivot}
  \begin{hexboard}[baseline={($(current bounding box.center)-(0,1ex)$)},scale=0.6]
    \rotation{-30}
    \foreach\i in {8,...,10} {\hex(\i,0)}
    \foreach\i in {7,...,10} {\hex(\i,1)}
    \foreach\i in {6,...,10} {\hex(\i,2)}
    \foreach\i in {5,...,10} {\hex(\i,3)}
    \foreach\i in {4,...,10} {\hex(\i,4)}
    \foreach\i in {3,...,10} {\hex(\i,5)}
    \foreach\i in {3,...,9} {\hex(\i,6)}
    \foreach\i in {2,...,9} {\hex(\i,7)}
    \foreach\i in {2,...,8} {\hex(\i,8)}
    \white(8,0)
    \black(9,0)\sflabel{3}
    \white(10,0)
    \white(7,1)
    \black(8,1)
    \white(10,1)
    \white(6,2)
    \black(9,2)
    \white(10,2)
    \white(5,3)
    \black(6,3)
    \black(7,3)
    \white(10,3)
    \white(4,4)
    \black(5,4)
    \white(7,4)
    \black(9,4)
    \white(10,4)
    \white(3,5)
    \black(6,5)
    \white(7,5)
    \white(8,5)
    \black(9,5)
    \white(10,5)
    \white(3,6)
    \black(4,6)
    \white(9,6)
    \white(2,7)
    \black(3,7)
    \white(4,7)
    \black(5,7)
    \black(6,7)
    \black(8,7)
    \white(9,7)
    \white(2,8)
    \black(3,8)\sflabel{2}
    \white(4,8)
    \white(5,8)
    \white(6,8)
    \black(7,8)\sflabel{1}
    \white(8,8)
  \end{hexboard}
\end{equation}
To check the validity of a pivoting template, it then suffices to
juxtapose the template with the position {\eqref{eqn:co-pivot}} and
let a Hex solver check that Black has a second-player win. For
example, Figure~\ref{fig:pivot-validity} shows a position that can be
used to verify the validity of the pivoting template
{\eqref{eqn:pivoting-example}}.

As usual, minimality can be checked by placing a white stone on one of
the empty cells and re-running the solver, much like we did for other
types of templates.

\begin{figure}
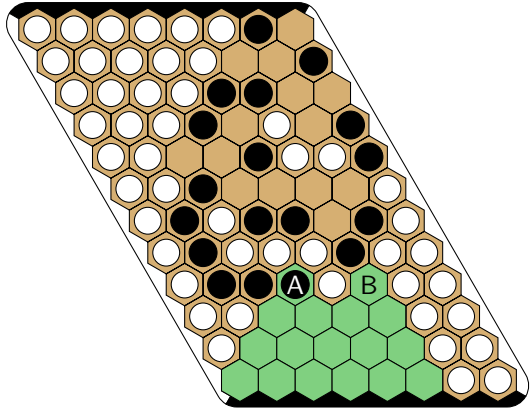

  \[
  \begin{hexboard}[scale=0.7]
    \rotation{-30}
    \board(8,12)
    \white(1,1)
    \white(2,1)
    \white(3,1)
    \white(4,1)
    \white(5,1)
    \white(6,1)
    \black(7,1)
    \white(1,2)
    \white(2,2)
    \white(3,2)
    \white(4,2)
    \white(5,2)
    \black(8,2)
    \white(1,3)
    \white(2,3)
    \white(3,3)
    \white(4,3)
    \black(5,3)
    \black(6,3)
    \white(1,4)
    \white(2,4)
    \white(3,4)
    \black(4,4)
    \white(6,4)
    \black(8,4)
    \white(1,5)
    \white(2,5)
    \black(5,5)
    \white(6,5)
    \white(7,5)
    \black(8,5)
    \white(1,6)
    \white(2,6)
    \black(3,6)
    \white(8,6)
    \white(1,7)
    \black(2,7)
    \white(3,7)
    \black(4,7)
    \black(5,7)
    \black(7,7)
    \white(8,7)
    \white(1,8)
    \black(2,8)
    \white(3,8)
    \white(4,8)
    \white(5,8)
    \black(6,8)
    \white(7,8)
    \white(8,8)
    \white(1,9)
    \black(2,9)
    \black(3,9)
    \black(4,9)\sflabel{A}
    \white(5,9)
    \cell(6,9)\sflabel{B}
    \white(7,9)
    \white(8,9)
    \white(1,10)
    \white(2,10)
    \white(7,10)
    \white(8,10)
    \white(1,11)
    \white(7,11)
    \white(8,11)
    \white(7,12)
    \white(8,12)
    \carry(4,9)
    \carry(6,9)
    \carry(3,10)\carry(4,10)\carry(5,10)\carry(6,10)
    \carry(2,11)\carry(3,11)\carry(4,11)\carry(5,11)\carry(6,11)
    \carry(1,12)\carry(2,12)\carry(3,12)\carry(4,12)\carry(5,12)\carry(6,12)
  \end{hexboard}
  \]  
  \caption{Verifying a pivoting template}
  \label{fig:pivot-validity}
\end{figure}

\subsection{Sente pivoting forks}

\emph{Sente} is a Japanese Go term that roughly means ``keeping the
initiative''. Its opposite is \emph{gote}, which means ``losing the
initiative''. Concretely, suppose there is a goal that a player wants
to achieve. To achieve the goal in sente means to achieve it in such a
way that it's the player's turn immediately afterwards. If instead it
is the opponent's turn after the player achieves the goal, then the
player has achieved it in gote.

Our definition of a pivoting fork in Section~\ref{ssec:pivoting-forks}
is such that when Left achieves the outcome $b$, Left achieves it in
gote. We can also define a stronger version of a pivoting fork where
Left achieves $b$ in sente. The following definition is analogous to
Definition~\ref{def:pivoting-formal}, except that we have moved the
condition $b\leq G$ to when it is Left's turn.

\begin{definition}[Sente pivoting fork]
  A fork $G$ is called \emph{sente pivoting} if
  \begin{itemize}
  \item $a\tri G$ and for all $G^R$, either $b\leq G$ or there exists
    some $G^{R(L)}$ that is sente pivoting.
  \end{itemize}
\end{definition}

The following characterization of sente pivoting forks is analogous to
Proposition~\ref{prop:pivoting}.

\begin{proposition}\label{prop:sente-pivoting}
  For a passable fork $G$, the following are equivalent:
  \begin{enumerate}\alphalabels
  \item $G$ is sente pivoting;
  \item $\top \leq G\plusj \g{a,b|a}$;
  \item $\g{a|a,b}\leq G$.
  \end{enumerate}
\end{proposition}

\begin{proof}
  The proof follows along similar lines as that of
  Proposition~\ref{prop:pivoting}, and we omit the details.
\end{proof}

An example of a sente pivoting template is shown in
Figure~\ref{fig:pivoting-gote-sente}(b). It is obtained from the
(gote) pivoting template in Figure~\ref{fig:pivoting-gote-sente}(a) by
extending the template's carrier with an additional cell between $A$
and $B$. The idea is that even after Black occupies $B$, Black still
threatens to connect to $A$, forcing White to spend one more move
defending against that threat. So when Black finally stops threatening
$A$, it is Black's turn.

\begin{figure}
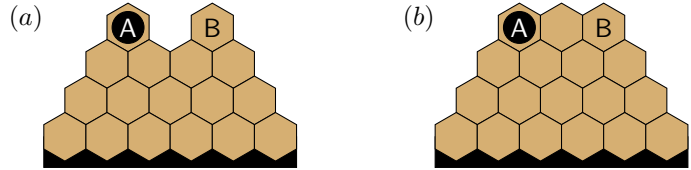

  \[
  (a)
  \begin{hexboard}[baseline={($(current bounding box.north)-(0,2ex)$)},scale=0.8]
    \rotation{-30}
    \template(6,4)
    \foreach \i in {4,6} {\hex(\i,1)}
    \foreach \i in {3,...,6} {\hex(\i,2)}
    \foreach \i in {2,...,6} {\hex(\i,3)}
    \foreach \i in {1,...,6} {\hex(\i,4)}
    \black(4,1)\sflabel{A}
    \cell(6,1)\sflabel{B}
  \end{hexboard}
  \qquad\qquad
  (b)
  \begin{hexboard}[baseline={($(current bounding box.north)-(0,2ex)$)},scale=0.8]
    \rotation{-30}
    \template(6,4)
    \foreach \i in {4,5,6} {\hex(\i,1)}
    \foreach \i in {3,...,6} {\hex(\i,2)}
    \foreach \i in {2,...,6} {\hex(\i,3)}
    \foreach \i in {1,...,6} {\hex(\i,4)}
    \black(4,1)\sflabel{A}
    \cell(6,1)\sflabel{B}
  \end{hexboard}
  \]
  \caption{A gote pivoting template and a sente pivoting template}
  \label{fig:pivoting-gote-sente}
\end{figure}

The following proposition shows that this works in general, i.e.,
extending any (gote) pivoting template with a semi-connection from $A$
to $B$ results in a sente pivoting template.

\begin{proposition}\label{prop:gote-pivoting-sente}
  As before, let $P_2 = \s{\bot,a,b,\top}$ be the outcome poset for a
  fork. Also consider the 2-element poset $\Bool=\s{\top,\bot}$ and
  any addition operation $+:P_2\times\Bool\to P_2$ satisfying
  $b+\top\geq a$ and $x+\bot\geq x$ for all $x\in P_2$. If $G$ is a
  pivoting fork, then $G+\g{\top|\bot}$ is a sente pivoting fork.
\end{proposition}

\begin{proof}
  By Proposition~\ref{prop:pivoting}, we know that $b\leq G$ or
  $\g{a|a,\g{b|\bot}}\leq G$. By
  Proposition~\ref{prop:sente-pivoting}, we must show that
  $\g{a|a,b}\leq G+\g{\top|\bot}$. Therefore, it suffices to show that
  $\g{a|a,b}\leq b+\g{\top|\bot}$ and $\g{a|a,b}\leq
  \g{a|a,\g{b|\bot}}+\g{\top|\bot}$. Both are easy to show from the
  definition of $\leq$ and the assumptions about $+$.
\end{proof}

\subsection{Application: A new handicap strategy for \texorpdfstring{$11\times 11$}{11 × 11} Hex}

In {\cite{Henderson-Hayward-2015}}, Henderson and Hayward described an
explicit winning strategy for Black in $11\times 11$ Hex if Black is
allowed to start by playing two stones. (More generally, they
described a winning strategy on $n\times n$ Hex, provided that Black
is allowed to start with $k$ stones where $6k-1\geq
n$). Figure~\ref{fig:handicap-strategies}(a) shows Henderson and
Hayward's winning opening.

\begin{figure}
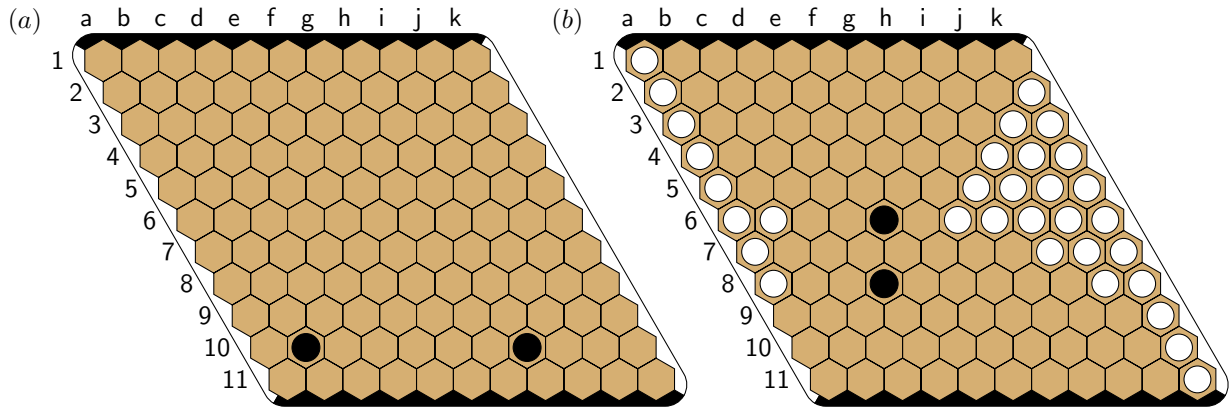

  \[
  (a)
  \begin{hexboard}[baseline={($(current bounding box.north)-(0,2.5ex)$)},scale=0.7]
    \rotation{-30}
    \board(11,11)
    \black(2,10)
    \black(8,10)
    \cell(1.2,-0.3)\label*{\sf a}
    \cell(2.2,-0.3)\label*{\sf b}
    \cell(3.2,-0.3)\label*{\sf c}
    \cell(4.2,-0.3)\label*{\sf d}
    \cell(5.2,-0.3)\label*{\sf e}
    \cell(6.2,-0.3)\label*{\sf f}
    \cell(7.2,-0.3)\label*{\sf g}
    \cell(8.2,-0.3)\label*{\sf h}
    \cell(9.2,-0.3)\label*{\sf i}
    \cell(10.2,-0.3)\label*{\sf j}
    \cell(11.2,-0.3)\label*{\sf k}
    \cell(0.2,1)\leftlabel{\sf 1}
    \cell(0.2,2)\leftlabel{\sf 2}
    \cell(0.2,3)\leftlabel{\sf 3}
    \cell(0.2,4)\leftlabel{\sf 4}
    \cell(0.2,5)\leftlabel{\sf 5}
    \cell(0.2,6)\leftlabel{\sf 6}
    \cell(0.2,7)\leftlabel{\sf 7}
    \cell(0.2,8)\leftlabel{\sf 8}
    \cell(0.2,9)\leftlabel{\sf 9}
    \cell(0.2,10)\leftlabel{\sf 10}
    \cell(0.2,11)\leftlabel{\sf 11}
  \end{hexboard}
  \hspace{-12ex}
  (b)
  \begin{hexboard}[baseline={($(current bounding box.north)-(0,2.5ex)$)},scale=0.7] 
    \rotation{-30}
    \board(11,11)
    \white(1,1)
    \white(1,2)
    \white(11,2)
    \white(1,3)
    \white(10,3)
    \white(11,3)
    \white(1,4)
    \white(9,4)
    \white(10,4)
    \white(11,4)
    \white(1,5)
    \white(8,5)
    \white(9,5)
    \white(10,5)
    \white(11,5)
    \white(1,6)
    \white(2,6)
    \black(5,6)
    \white(7,6)
    \white(8,6)
    \white(9,6)
    \white(10,6)
    \white(11,6)
    \white(1,7)
    \white(9,7)
    \white(10,7)
    \white(11,7)
    \white(1,8)
    \black(4,8)
    \white(10,8)
    \white(11,8)
    \white(11,9)
    \white(11,10)
    \white(11,11)
    \cell(1.2,-0.3)\label*{\sf a}
    \cell(2.2,-0.3)\label*{\sf b}
    \cell(3.2,-0.3)\label*{\sf c}
    \cell(4.2,-0.3)\label*{\sf d}
    \cell(5.2,-0.3)\label*{\sf e}
    \cell(6.2,-0.3)\label*{\sf f}
    \cell(7.2,-0.3)\label*{\sf g}
    \cell(8.2,-0.3)\label*{\sf h}
    \cell(9.2,-0.3)\label*{\sf i}
    \cell(10.2,-0.3)\label*{\sf j}
    \cell(11.2,-0.3)\label*{\sf k}
    \cell(0.2,1)\leftlabel{\sf 1}
    \cell(0.2,2)\leftlabel{\sf 2}
    \cell(0.2,3)\leftlabel{\sf 3}
    \cell(0.2,4)\leftlabel{\sf 4}
    \cell(0.2,5)\leftlabel{\sf 5}
    \cell(0.2,6)\leftlabel{\sf 6}
    \cell(0.2,7)\leftlabel{\sf 7}
    \cell(0.2,8)\leftlabel{\sf 8}
    \cell(0.2,9)\leftlabel{\sf 9}
    \cell(0.2,10)\leftlabel{\sf 10}
    \cell(0.2,11)\leftlabel{\sf 11}
  \end{hexboard}
  \]
  \caption{Two handicap winning positions for Black. (a) From
    Henderson and Hayward {\cite{Henderson-Hayward-2015}}. (b) Using a
    pivoting template.}
  \label{fig:handicap-strategies}
\end{figure}

Interestingly, our characterization of pivoting templates allows us to
give another such handicap strategy for $11\times 11$. Specifically,
we claim that the two black stones shown in
Figure~\ref{fig:handicap-strategies}(b) are a winning opening for
Black. The white stones are of course not required in an actual game,
but we have included them in the figure to show the area that Black needs
to carry out the win (or more precisely, the area that we need to
carry out our \emph{proof} of Black's win).

\begin{proposition}
  The position shown in Figure~\ref{fig:handicap-strategies}(b) is
  winning for Black, with White to move.
\end{proposition}

\begin{proof}
  In principle, we could show that this position is winning for Black
  by directly inputting it into a Hex solver. However, the region is
  too large to be efficiently solvable. Instead, we prove it using a
  divide-and-conquer method.  We divide the board into two regions as
  shown in Figure~\ref{fig:handicap-proof}. We also insert some thin
  white lines, and claim that Black can win even without the winning
  path crossing those lines.

  First, we claim that region 1, minus the cell marked ``$*$'', is a
  pivoting template. As explained in
  Section~\ref{ssec:verifying-pivoting}, this can be checked by
  juxtaposing it with an appropriate context, as in
  Figure~\ref{fig:conquer}(a). The position in
  Figure~\ref{fig:conquer}(a) is simple enough to be solved by Mohex,
  and is a second-player win for Black as claimed.

  Second, we claim that region 1, including the cell marked ``$*$'',
  is a sente pivoting template. This follows by
  Proposition~\ref{prop:gote-pivoting-sente}. Therefore, by
  Proposition~\ref{prop:sente-pivoting}, the value of region 1
  (including the cell marked ``$*$'') is at least $\g{a|a,b}$.

  Third, we claim that region 2, when juxtaposed with a fork of value
  $\g{a|a,b}$, is a second-player win for Black. Since the dual
  superswitch of Section~\ref{ssec:distinguish} is known to have
  exactly value $\g{a|a,b}$, the claim can be checked by juxtaposing
  region 2 with such a dual superswitch, as in
  Figure~\ref{fig:conquer}(b). This position is simple enough to be
  solved by Mohex, and is a second-player win for Black as claimed.

  Putting all claims together, since region 2 is winning in the
  context of $\g{a|a,b}$, and region 1 is at least as good for Black
  as $\g{a|a,b}$, region 2 is also winning in the context of region 1,
  as claimed.
\end{proof}

The idea of using divide-and-conquer methods for finding winning
strategies in Hex is not new. Yang et al.~\cite{YLP-decomp} decomposed
the board into smaller regions to identify winning moves on boards up
to size $9\times 9$. However, they only considered decompositions into
simple (2-terminal) templates. By contrast, our regions have much more
intricate combinatorial properties.

\subsection{Corollary: A winning strategy for 1-move handicap}\label{ssec:handicap-strategy}

A system for measuring the strength of handicap in Hex was proposed by
one of the authors, and has found some acceptance in the Hex community
{\cite{Demer-handicap}}. The idea is to measure handicap as a nominal
``number of extra moves'' given to Black at the start of the
game. Through selective use of the swap rule, this ``number of moves''
can be adjusted in increments of 0.5. Specifically, when playing Hex
with the swap rule, the game is approximately fair (we ignore the
second player's theoretical advantage because it is very small in
practice). This is considered a handicap of 0 moves. On the other
hand, when playing without the swap rule, the difference between going
first and going second is exactly one extra move at the start of the
game. So compared to a theoretically fair game, it makes sense to say
that Black, who goes first, has an advantage worth 0.5 moves when
playing without swap. An advantage of exactly 1 move can be achieved
by playing with swap, but giving Black one extra move at the earliest
opportunity (i.e., right after White decides to swap or not). This
means: Black plays, White swaps, and Black makes two consecutive
moves, or: Black plays, White declines to swap, and Black makes an
additional move. A handicap of 1.5 moves can be achieved by playing
without swap and letting Black start with two moves, and so on.

In this terminology, both Henderson and Hayward's strategy and our
pivoting strategy (Figure~\ref{fig:handicap-strategies}(a) and (b))
are for 1.5-move handicap: the swap rule is not used and Black gets to
start with one extra move.

Interestingly, the two strategies can be combined to yield a
guaranteed winning strategy for Black with 1-move handicap. The strategy is as
follows: Black opens at \move{h10}. If White doesn't swap, Black plays
a free move at \move{b10} and then follows the Henderson-Hayward
strategy (see Figure~\ref{fig:handicap-strategies}(a)). If White does
swap, the black stone at \move{h10} effectively becomes a white stone
at \move{j8}, which is outside of the carrier of our pivoting
strategy. Black plays two stones at \move{e6} and \move{d8} and then
follows the pivoting strategy (see
Figure~\ref{fig:handicap-strategies}(b)).

\begin{figure}
  \[
  \begin{hexboard}[baseline={($(current bounding box.north)-(0,2.5ex)$)},scale=0.7] 
    \rotation{-30}
    \board(11,11)
    \white(1,1)
    \white(1,2)
    \white(11,2)
    \white(1,3)
    \white(10,3)
    \white(11,3)
    \white(1,4)
    \white(9,4)
    \white(10,4)
    \white(11,4)
    \white(1,5)
    \white(8,5)
    \white(9,5)
    \white(10,5)
    \white(11,5)
    \white(1,6)
    \white(2,6)
    \white(7,6)
    \white(8,6)
    \white(9,6)
    \white(10,6)
    \white(11,6)
    \white(1,7)
    \white(9,7)
    \white(10,7)
    \white(11,7)
    \white(1,8)
    \white(10,8)
    \white(11,8)
    \white(11,9)
    \white(11,10)
    \white(11,11)
    \cell(3,6)\sflabel{B}
    \cell(4,6)\label{$*$}
    \black(5,6)\sflabel{A}
    \black(4,8)
    \foreach\i in {2,...,11} {\shaded{FF8080}(\i,1)}
    \foreach\i in {2,...,10} {\shaded{FF8080}(\i,2)}
    \foreach\i in {2,...,9} {\shaded{FF8080}(\i,3)}
    \foreach\i in {2,...,8} {\shaded{FF8080}(\i,4)}
    \foreach\i in {2,...,7} {\shaded{FF8080}(\i,5)}
    \foreach\i in {3,...,6} {\shaded{FF8080}(\i,6)}
    \foreach\i in {2,...,8} {\shaded{80D080}(\i,7)}
    \foreach\i in {2,...,9} {\shaded{80D080}(\i,8)}
    \foreach\i in {1,...,10} {\shaded{80D080}(\i,9)}
    \foreach\i in {1,...,10} {\shaded{80D080}(\i,10)}
    \foreach\i in {1,...,10} {\shaded{80D080}(\i,11)}
    \begin{scope}
      \clip \coord(3,7)--\coord(4,5)--\coord(5,5)--\coord(4,7)--\coord(5,7)--\coord(6,5)--\coord(7,5)--\coord(6,7)--cycle;
      \draw[white, line width=0.8mm] \coord(3.3,6.3)--\coord(3.666,6.666)--\coord(4.4,6.3);
      \draw[white, line width=0.8mm] \coord(5.3,6.3)--\coord(5.666,6.666)--\coord(6.4,6.3);
    \end{scope}
    \node at \coord(5.333,3.333) {\sf\color{white}{\Huge\textbf{1}}};
    \node at \coord(6,9) {\sf\color{white}{\Huge\textbf{2}}};
  \end{hexboard}
  \]
  \caption{Divide\ldots}
  \label{fig:handicap-proof}
\end{figure}

\begin{figure}
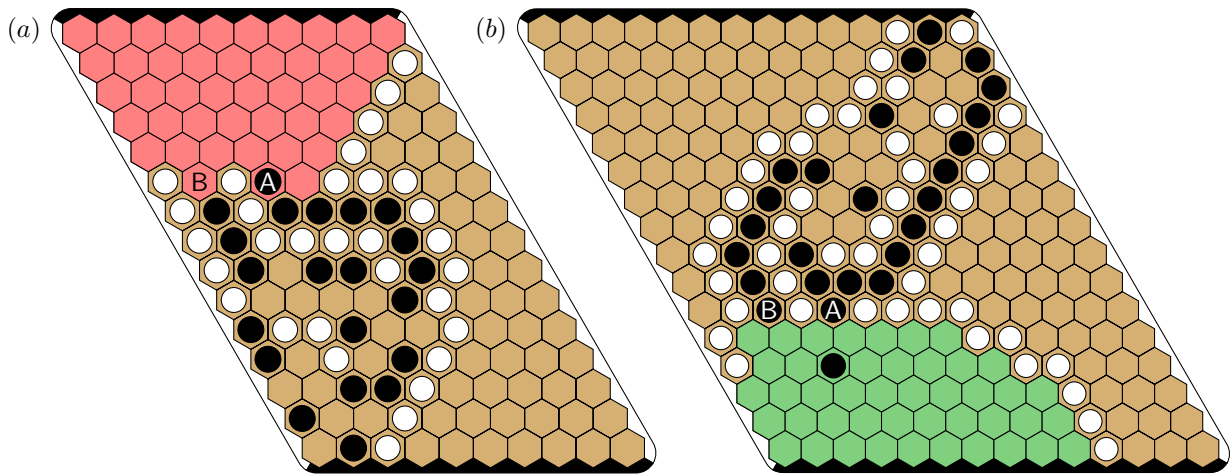

  \[
  (a)~
  \begin{hexboard}[scale=0.65,rotate=180,baseline={($(current bounding box.north)-(0,2.5ex)$)}]
    \rotation{-30}
    \board(10,15)
    \foreach\i in {6,7,9} {\shaded{FF8080}(\i,10)}
    \foreach\i in {5,...,10} {\shaded{FF8080}(\i,11)}
    \foreach\i in {4,...,10} {\shaded{FF8080}(\i,12)}
    \foreach\i in {3,...,10} {\shaded{FF8080}(\i,13)}
    \foreach\i in {2,...,10} {\shaded{FF8080}(\i,14)}
    \foreach\i in {1,...,10} {\shaded{FF8080}(\i,15)}
    \white(8,1)
    \black(9,1)
    \white(7,2)
    \black(10,2)
    \white(6,3)
    \black(7,3)
    \black(8,3)
    \white(5,4)
    \black(6,4)
    \white(8,4)
    \black(10,4)
    \white(4,5)
    \black(7,5)
    \white(8,5)
    \white(9,5)
    \black(10,5)
    \white(4,6)
    \black(5,6)
    \white(10,6)
    \white(3,7)
    \black(4,7)
    \white(5,7)
    \black(6,7)
    \black(7,7)
    \black(9,7)
    \white(10,7)
    \white(3,8)
    \black(4,8)
    \white(5,8)
    \white(6,8)
    \white(7,8)
    \white(8,8)
    \black(9,8)
    \white(10,8)
    \white(3,9)
    \black(4,9)
    \black(5,9)
    \black(6,9)
    \black(7,9)
    \white(8,9)
    \black(9,9)
    \white(10,9)
    \white(3,10)
    \white(4,10)
    \white(5,10)
    \black(7,10)\sflabel{A}
    \white(8,10)
    \cell(9,10)\sflabel{B}
    \white(10,10)
    \white(4,11)
    \white(3,12)
    \white(2,13)
    \white(1,14)
  \end{hexboard}
  \hspace{-16ex}
  (b)~
  \begin{hexboard}[scale=0.61,baseline={($(current bounding box.north)-(0,2.5ex)$)}]
    \rotation{-30}
    \board(14,16)
    \foreach\i in {2,...,8} {\shaded{80D080}(\i,12)}
    \foreach\i in {2,...,9} {\shaded{80D080}(\i,13)}
    \foreach\i in {1,...,10} {\shaded{80D080}(\i,14)}
    \foreach\i in {1,...,10} {\shaded{80D080}(\i,15)}
    \foreach\i in {1,...,10} {\shaded{80D080}(\i,16)}
    \white(12,1)
    \black(13,1)
    \white(14,1)
    \white(11,2)
    \black(12,2)
    \black(14,2)
    \white(10,3)
    \white(11,3)
    \black(14,3)
    \white(8,4)
    \white(9,4)
    \black(10,4)
    \white(12,4)
    \black(13,4)
    \white(14,4)
    \white(6,5)
    \white(7,5)
    \white(10,5)
    \black(12,5)
    \white(13,5)
    \white(5,6)
    \black(6,6)
    \black(7,6)
    \white(10,6)
    \black(11,6)
    \white(12,6)
    \white(4,7)
    \black(5,7)
    \white(6,7)
    \black(8,7)
    \white(9,7)
    \black(10,7)
    \white(11,7)
    \white(3,8)
    \black(4,8)
    \white(5,8)
    \white(8,8)
    \black(9,8)
    \white(10,8)
    \white(2,9)
    \black(3,9)
    \white(4,9)
    \black(5,9)
    \white(6,9)
    \white(7,9)
    \black(8,9)
    \white(9,9)
    \white(2,10)
    \black(3,10)
    \white(4,10)
    \black(5,10)
    \black(6,10)
    \black(7,10)
    \white(8,10)
    \white(2,11)
    \white(4,11)
    \white(6,11)
    \white(7,11)
    \white(8,11)
    \white(9,11)
    \white(1,12)
    \white(9,12)
    \white(10,12)
    \white(1,13)
    \black(4,13)
    \white(10,13)
    \white(11,13)
    \white(11,14)
    \white(11,15)
    \white(11,16)
    \black(3,11)\sflabel{B}
    \black(5,11)\sflabel{A}
  \end{hexboard}
  \]
  \caption{\ldots and conquer. Both positions are second-player wins for
    Black and can be efficiently solved by Mohex.}
  \label{fig:conquer}
\end{figure}

\section{Application: Non-inferiority of probes in Hex templates}
\label{sec:witnesses}

\subsection{Inferiority of probes}

Consider the following edge template, which is called the
\emph{ziggurat} or the \emph{4-3-2 template} {\cite{Seymour,Henderson-Hayward}}:
\begin{equation}\label{eqn:ziggurat}
\begin{hexboard}[baseline={($(current bounding box.center)-(0,1ex)$)},scale=0.8]
  \template(4,3)
  \foreach \i in {3,...,4} {\hex(\i,1)}
  \foreach \i in {2,...,4} {\hex(\i,2)}
  \foreach \i in {1,...,4} {\hex(\i,3)}
  \black(3,1)
  \cell(4,1)\sflabel{1}
  \cell(2,2)\sflabel{2}
  \cell(3,2)\sflabel{3}
  \cell(4,2)\sflabel{4}
  \cell(1,3)\sflabel{5}
  \cell(2,3)\sflabel{6}
  \cell(3,3)\sflabel{7}
  \cell(4,3)\sflabel{8}
\end{hexboard}
\end{equation}
It is easy to verify that Black can indeed connect the stone to the
edge. One way for Black to accomplish this is to play the pairing
strategy $\s{\s{1,3},\s{2,4},\s{5,6},\s{7,8}}$: if White plays in any
numbered cell in the template, Black plays in the other numbered cell
of the same pair.

Nevertheless, the fact that Black can defend the connection does not
mean that it is never useful for White to play in the template. By
judiciously playing in the template, White can force Black to respond
in a way that gives White an advantage. For example, consider the
following situation, with White to move. It is easy to see that
White's move at $4$ is winning, and every other move is losing:
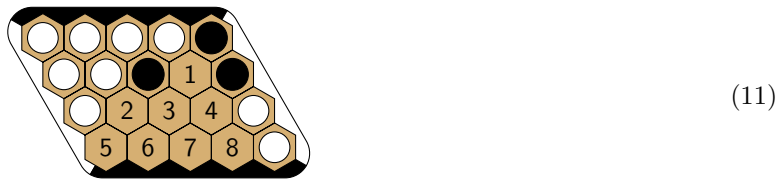
\begin{equation}\label{eqn:witness-ziggurat-4-ex}
\begin{hexboard}[baseline={($(current bounding box.center)-(0,1ex)$)},scale=0.8]
  \rotation{-30}
  \board(5,4)
  \white(1,1)
  \white(2,1)
  \white(3,1)
  \white(4,1)
  \white(1,2)
  \white(2,2)
  \white(1,3)
  \black(5,1)
  \black(3,2)
  \black(5,2)
  \white(5,3)
  \white(5,4)
  \cell(4,2)\sflabel{1}
  \cell(2,3)\sflabel{2}
  \cell(3,3)\sflabel{3}
  \cell(4,3)\sflabel{4}
  \cell(1,4)\sflabel{5}
  \cell(2,4)\sflabel{6}
  \cell(3,4)\sflabel{7}
  \cell(4,4)\sflabel{8}
\end{hexboard}
\end{equation}
A white move in Black's template is called an \emph{intrusion} or a
\emph{probe} {\cite{Seymour,Henderson-Hayward}}. A probe $x$ is
\emph{inferior} if it can never be the unique winning move. More
precisely, for every way of embedding the template in a Hex board, if
playing at $x$ is a winning first move for White, then there also
exists some other winning first move (inside or outside the
template). Conversely, a probe is \emph{non-inferior} if there exists
some board position where $x$ is the unique winning move. Such a board
position is called a \emph{witness} of the non-inferiority of the
probe. Thus, {\eqref{eqn:witness-ziggurat-4-ex}} is a witness of the
non-inferiority of probe~$4$ in the ziggurat.

Henderson and Hayward conjectured that in the ziggurat, probes $3$,
$5$, $6$, $7$ and $8$ are inferior
{\cite[Conj.~1]{Henderson-Hayward}}. Here, we show that the conjecture
is false: in fact, we prove that none of the ziggurat's probes are
inferior. (Impatient readers can skip directly to
Figure~\ref{fig:witnesses-ziggurat} to see the witnesses).

Proving the non-inferiority of some probe in a template, or more
generally, of some move in a Hex region, is not an easy task, because
it requires finding a witnessing context in which that move, and only
that move, is winning. It is not usually possible to find such
witnesses by trial and error or by a brute force search, because the
witnesses can be exceedingly rare and subtle. Instead, we need a
finely-tuned tool. It turns out that combinatorial game theory, along
with our database of Hex-realizable 3-terminal values from
Section~\ref{sec:database}, is the right tool for the job.

\subsection{Abstract probes}\label{ssec:abstract-probes}

Our method for constructing witnesses of non-inferiority is best
illustrated in an example. Our running example will be probe~$7$ in
the ziggurat. We introduce the required machinery in this section, and
then give a step-by-step demonstration of the method in
Section~\ref{ssec:step-by-step}.

First, some terminology. Consider the ziggurat in
{\eqref{eqn:ziggurat}}. Let $A$ be the outcome poset for the ziggurat,
as defined in Section~\ref{sec:background-cgt}. Note that the ziggurat
has a complicated boundary that is not an $n$-terminal region, and so
its outcome poset is likely to be complicated. To carry out the
computations below, we must be able to do computations in this outcome
poset, but fortunately, we do not need to explicitly understand what
its elements are.

Let $G$ be the game form of the ziggurat {\eqref{eqn:ziggurat}} over
the poset $A$. We have
\begin{equation}\label{eqn:ziggurat-value}
  G \eq \g{G^L_1,\ldots,G^L_8|G^R_1,\ldots,G^R_8},
\end{equation}
where $G^L_i$ is the value of the position obtained by placing a black
stone on cell $i$ of the ziggurat, and $G^R_i$ is the value of the
analogous position with a white stone on cell $i$.  Since we are
interested in situations where $G^R_7$ is the only winning move for
White, we can ask what happens if we modify $G$ so that White cannot
play $7$ as the first move. Let $H$ be the game form that is just like
$G$, except that the right option $G^R_7$ has been removed:
\begin{equation}\label{eqn:ziggurat-probe}
  H = \g{G^L_1,\ldots,G^L_8|G^R_1,\ldots,G^R_6,G^R_8}.
\end{equation}
(Note that we only remove the white move at $7$ from $G$ itself, not
from any proper followers of $G$; in other words, White will be
allowed to play in cell $7$ as long as it is not the first move in the
region). We call the pair $(G,H)$ an \emph{abstract probe}. More
generally:

\begin{definition}[Abstract probe]
  An \emph{abstract probe} over an outcome poset $A$ is a pair $(G,H)$
  of games over $A$ such that $G\leq H$. We say that the abstract
  probe is \emph{viable} if $G<H$.
\end{definition}

When an abstract probe is viable, it means that there
\emph{potentially} exists some context in which White is winning in
$G$ and losing in $H$. In particular, in the intended situation where
$H$ is obtained from $G$ by removing a single white option, it means
that there potentially exists a context where the removed option was
the only winning move for White. On the other hand, when an abstract
probe is not viable, there is no such context.

Our general strategy is to start with a viable abstract probe, and
then gradually refine and extend the context while keeping the probe
viable at every step. To carry out this refinement, we must be able to
add a context to a probe. We do this in a componentwise way, by
defining $(G,H)\plusf X = (G\plusf X, H\plusf X)$.

Our refinement process ends when we reach a 2-terminal position, i.e.,
a probe $(G,H)\plusf X$ over the boolean poset $\Bool$. At this point,
the following lemma will give us a witness of non-inferiority.

\begin{lemma}\label{lem:viable-follower}
  Let $(G,H)$ be a probe over a poset $A$, such that $G$ and $H$ have
  the same left options, and such that $G$ has exactly one right
  option $K$ that is not a right option of $H$. Let $X$ be a game over
  a poset $B$, and let $f:A\times B\to \Bool$ be a monotone
  function. If the probe $(G,H)\plusf X$ is viable, then there exists
  some follower $Y$ of $X$ such that $K\plusf Y$ is the unique
  winning move for White in $G\plusf Y$.
\end{lemma}

\begin{proof}
  By induction on $X$. Since $G\plusf X<H\plusf X$ and both games are
  over $\Bool$, we know that either as first player or as second
  player, White has a winning strategy in $G\plusf X$ but not in
  $H\plusf X$. Case 1: White has a first-player winning strategy in
  $G\plusf X$ but not in $H\plusf X$, and none of White's winning
  moves in $G\plusf X$ are in $X$. In this case, White must have a
  winning move of the form $G^R\plusf X$. Since White does not have a
  winning move in $H\plusf X$, $G^R$ cannot be a right option of $H$,
  and therefore, $G^R$ must be $K$, proving the claim of the lemma.
  Case 2: White has a first-player winning strategy in $G\plusf X$ but
  not in $H\plusf X$, and at least one of White's winning moves in
  $G\plusf X$ is of the form $G\plusf X^R$. By assumption, $H\plusf
  X^R$ is not winning, so $G\plusf X^R<H\plusf X^R$, and the claim
  follows by the induction hypothesis. Case 3: White has a
  second-player winning strategy in $G\plusf X$ but not in $H\plusf
  X$. Then Left has some winning move in $H\plusf X$.  Since $G$ and
  $H$ have the same left options and Left has no winning move in
  $G\plusf X$, Left's winning move in $H\plusf X$ must be some
  $H\plusf X^L$, and the corresponding $G\plusf X^L$ is not winning
  for Left. Then $G\plusf X^L<H\plusf X^L$, and the claim
  follows by the induction hypothesis.
\end{proof}

\subsection{Step-by-step computation of witnesses}
\label{ssec:step-by-step}

Let us now apply this method to probe~$7$ in the ziggurat. Note that
each of the below steps is explorative; if any step fails, one must
redo the previous step.

\paragraph{Step 1.}

Let $P=(G,H)$ be the abstract probe defined by
{\eqref{eqn:ziggurat-value}} and {\eqref{eqn:ziggurat-probe}} above,
i.e., the one corresponding to probe~$7$ in the ziggurat. Since the
ziggurat only has $8$ empty cells, the values $G$ and $H$ can be
computed reasonably efficiently by traversing the game tree. The first
thing we must check is that the probe $P$ is viable, i.e., that
$G<H$. This can be verified by a computation.

\paragraph{Step 2.}

Next, we try to convert $P$ into an $n$-terminal probe by
surrounding it with black and white stones. There are potentially 
many ways to do so. By trial and error, we find that the following
boundary keeps the probe viable.
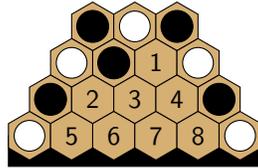
\[
\begin{hexboard}[baseline={($(current bounding box.center)-(0,1ex)$)},scale=0.8]
  \template(6,4)
  \foreach \i in {4,...,6} {\hex(\i,1)}
  \foreach \i in {3,...,6} {\hex(\i,2)}
  \foreach \i in {2,...,6} {\hex(\i,3)}
  \foreach \i in {1,...,6} {\hex(\i,4)}
  \black(4,2)
  \cell(5,2)\sflabel{1}
  \cell(3,3)\sflabel{2}
  \cell(4,3)\sflabel{3}
  \cell(5,3)\sflabel{4}
  \cell(2,4)\sflabel{5}
  \cell(3,4)\sflabel{6}
  \cell(4,4)\sflabel{7}
  \cell(5,4)\sflabel{8}
  \white(1,4)
  \black(2,3)
  \white(3,2)
  \black(4,1)
  \white(5,1)
  \black(6,1)
  \white(6,2)
  \black(6,3)
  \white(6,4)
\end{hexboard}
\]
Let $P'$ be the value of this new probe over the 5-terminal outcome
poset. We can schematically embed the probe $P'$ inside a Hex board as
follows:
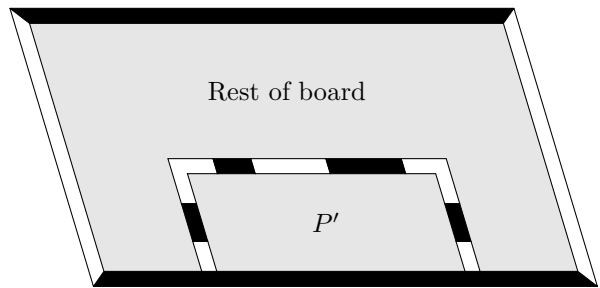
\[
\begin{tikzpicture}[xscale=0.5, yscale=0.5, xslant=-0.3]
  \pgfdeclarelayer{black}
  \pgfdeclarelayer{white}
  \pgfsetlayers{black,white,main}
  \def\r{0.2}

  \fill[black!10]
  (0,7) -- (13,7) -- (13,0) -- (0,0) -- cycle;

  \draw[fill=black,line join=round]
  (-\r,-\r) -- (\r,\r) -- (13-\r,\r) -- (13+\r,-\r) -- cycle
  (-\r,7+\r) -- (\r,7-\r) -- (13-\r,7-\r) -- (13+\r,7+\r) -- cycle;
  \draw[fill=white,line join=round]
  (-\r,-\r) -- (\r,\r) -- (\r,7-\r) -- (-\r,7+\r) -- cycle
  (13-\r,\r) -- (13+\r,-\r) -- (13+\r,7+\r) -- (13-\r,7-\r) -- cycle;

  \draw[fill=white]
  (3-\r,\r) -- (3-\r,1) -- (3+\r,1) -- (3+\r,\r) -- cycle
  (3-\r,2) -- (3-\r,3+\r) -- (4,3+\r) -- (4,3-\r) -- (3+\r,3-\r) -- (3+\r,2) -- cycle
  (5,3+\r) -- (7,3+\r) -- (7,3-\r) -- (5,3-\r) -- cycle
  (10+\r,2) -- (10+\r,3+\r) -- (9,3+\r) -- (9,3-\r) -- (10-\r,3-\r) -- (10-\r,2) -- cycle
  (10+\r,\r) -- (10+\r,1) -- (10-\r,1) -- (10-\r,\r) -- cycle
  ;
  \draw[fill=black]
  (3-\r,1) -- (3+\r,1) -- (3+\r,2) -- (3-\r,2) -- cycle
  (4,3+\r) -- (4,3-\r) -- (5,3-\r) -- (5,3+\r) -- cycle
  (7,3+\r) -- (7,3-\r) -- (9,3-\r) -- (9,3+\r) -- cycle
  (10-\r,1) -- (10+\r,1) -- (10+\r,2) -- (10-\r,2) -- cycle
  ;
  \node at (6.5,1.5) {$P'$};
  \node at (6.5,5) {Rest of board};
\end{tikzpicture}
\]
Our goal is to find a Hex-realizable value for ``Rest of board'' that
keeps the probe viable.

\paragraph{Step 3.}

We divide the rest of the board into simpler regions. By trial and
error, we find that a good first step is to slice off two corners
$A$ and $B$, as in the following diagram:
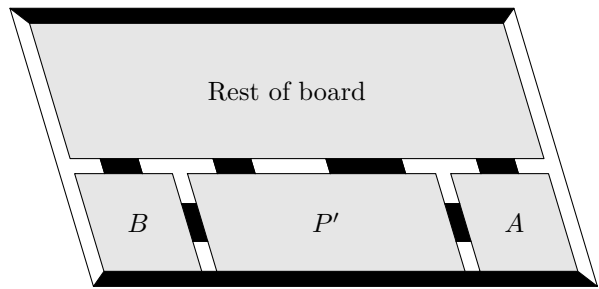
\[
\begin{tikzpicture}[xscale=0.5, yscale=0.5, xslant=-0.3]
  \pgfdeclarelayer{black}
  \pgfdeclarelayer{white}
  \pgfsetlayers{black,white,main}
  \def\r{0.2}

  \fill[black!10]
  (0,7) -- (13,7) -- (13,0) -- (0,0) -- cycle;

  \draw[fill=black,line join=round]
  (-\r,-\r) -- (\r,\r) -- (13-\r,\r) -- (13+\r,-\r) -- cycle
  (-\r,7+\r) -- (\r,7-\r) -- (13-\r,7-\r) -- (13+\r,7+\r) -- cycle;
  \draw[fill=white,line join=round]
  (-\r,-\r) -- (\r,\r) -- (\r,3-\r) -- (1,3-\r) -- (1,3+\r) -- (\r,3+\r) -- (\r,7-\r) -- (-\r,7+\r) -- cycle
  (13-\r,\r) -- (13+\r,-\r) -- (13+\r,7+\r) -- (13-\r,7-\r) -- (13-\r,3+\r) -- (12,3+\r) -- (12,3-\r) -- (13-\r,3-\r) -- cycle;

  \draw[fill=white]
  (3-\r,\r) -- (3-\r,1) -- (3+\r,1) -- (3+\r,\r) -- cycle
  (3-\r,2) -- (3-\r,3-\r) -- (2,3-\r) -- (2,3+\r) -- (4,3+\r) -- (4,3-\r) -- (3+\r,3-\r) -- (3+\r,2) -- cycle
  (5,3+\r) -- (7,3+\r) -- (7,3-\r) -- (5,3-\r) -- cycle
  (10+\r,2) -- (10+\r,3-\r) -- (11,3-\r) -- (11,3+\r) -- (9,3+\r) -- (9,3-\r) -- (10-\r,3-\r) -- (10-\r,2) -- cycle
  (10+\r,\r) -- (10+\r,1) -- (10-\r,1) -- (10-\r,\r) -- cycle
  ;
  \draw[fill=black]
  (3-\r,1) -- (3+\r,1) -- (3+\r,2) -- (3-\r,2) -- cycle
  (4,3+\r) -- (4,3-\r) -- (5,3-\r) -- (5,3+\r) -- cycle
  (7,3+\r) -- (7,3-\r) -- (9,3-\r) -- (9,3+\r) -- cycle
  (10-\r,1) -- (10+\r,1) -- (10+\r,2) -- (10-\r,2) -- cycle
  ;
  \node at (6.5,1.5) {$P'$};

  \draw[fill=black]
  (1,3-\r) -- (2,3-\r) -- (2,3+\r) -- (1,3+\r) -- cycle
  (11,3-\r) -- (12,3-\r) -- (12,3+\r) -- (11,3+\r) -- cycle
  ;
  \node at (11.5,1.5) {$A$};
  \node at (1.5,1.5) {$B$};
  \node at (6.5,5) {Rest of board};
\end{tikzpicture}
\]
Using the database from Remark~\ref{rem:corners-and-forks}, we make a
list of Hex-realizable corner values $A$ such that the probe $P'+A$
remains viable. Anecdotally, we find that the probe is kept viable by
about 1.5 percent of the possible values of $A$. Similarly, we make a
list of Hex-realizable corners $B$ such that the probe $P'+B$ remains
viable. Here, we find that about 64 percent of the candidate values of
$B$ work. Next, we combine the two separate lists of viable values for
$A$ and $B$ into a single list of pairs $(A,B)$ such that the probe
$P'+A+B$ is viable. The vast majority of such pairs $(A,B)$ work.

\paragraph{Step 4.}

We divide the rest of the board into smaller regions again. Consider
the regions $C$ and $D$ shown in the following diagram:
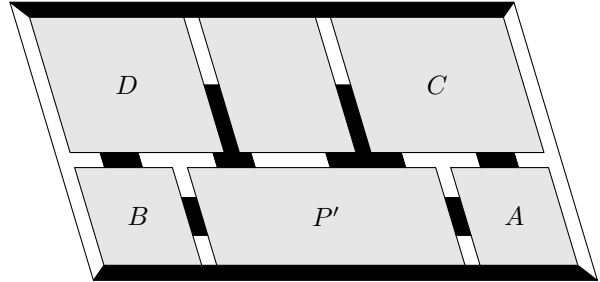
\[
\begin{tikzpicture}[xscale=0.5, yscale=0.5, xslant=-0.3]
  \pgfdeclarelayer{black}
  \pgfdeclarelayer{white}
  \pgfsetlayers{black,white,main}
  \def\r{0.2}

  \fill[black!10]
  (0,7) -- (13,7) -- (13,0) -- (0,0) -- cycle;

  \draw[fill=black,line join=round]
  (-\r,-\r) -- (\r,\r) -- (13-\r,\r) -- (13+\r,-\r) -- cycle
  (-\r,7+\r) -- (\r,7-\r) -- (13-\r,7-\r) -- (13+\r,7+\r) -- cycle;
  \draw[fill=white,line join=round]
  (-\r,-\r) -- (\r,\r) -- (\r,3-\r) -- (1,3-\r) -- (1,3+\r) -- (\r,3+\r) -- (\r,7-\r) -- (-\r,7+\r) -- cycle
  (13-\r,\r) -- (13+\r,-\r) -- (13+\r,7+\r) -- (13-\r,7-\r) -- (13-\r,3+\r) -- (12,3+\r) -- (12,3-\r) -- (13-\r,3-\r) -- cycle;

  \draw[fill=white]
  (3-\r,\r) -- (3-\r,1) -- (3+\r,1) -- (3+\r,\r) -- cycle
  (3-\r,2) -- (3-\r,3-\r) -- (2,3-\r) -- (2,3+\r) -- (4,3+\r) -- (4,3-\r) -- (3+\r,3-\r) -- (3+\r,2) -- cycle
  (5,3+\r) -- (7,3+\r) -- (7,3-\r) -- (5,3-\r) -- cycle
  (10+\r,2) -- (10+\r,3-\r) -- (11,3-\r) -- (11,3+\r) -- (9,3+\r) -- (9,3-\r) -- (10-\r,3-\r) -- (10-\r,2) -- cycle
  (10+\r,\r) -- (10+\r,1) -- (10-\r,1) -- (10-\r,\r) -- cycle
  ;
  \draw[fill=black]
  (3-\r,1) -- (3+\r,1) -- (3+\r,2) -- (3-\r,2) -- cycle
  (4,3+\r) -- (4.5-\r,3+\r) -- (4.5-\r,5) -- (4.5+\r,5) -- (4.5+\r,3+\r) -- (5,3+\r) -- (5,3-\r) -- (4,3-\r) -- cycle
  (7,3+\r) -- (8-\r,3+\r) -- (8-\r,5) -- (8+\r,5) -- (8+\r,3+\r) -- (9,3+\r) -- (9,3-\r) -- (7,3-\r) -- cycle
  (10-\r,1) -- (10+\r,1) -- (10+\r,2) -- (10-\r,2) -- cycle
  ;
  \node at (6.5,1.5) {$P'$};

  \draw[fill=black]
  (1,3-\r) -- (2,3-\r) -- (2,3+\r) -- (1,3+\r) -- cycle
  (11,3-\r) -- (12,3-\r) -- (12,3+\r) -- (11,3+\r) -- cycle
  ;
  \node at (11.5,1.5) {$A$};
  \node at (1.5,1.5) {$B$};

  \draw[fill=white]
  (4.5-\r,5) -- (4.5-\r,7-\r) -- (4.5+\r,7-\r) -- (4.5+\r,5) -- cycle
  (8-\r,5) -- (8-\r,7-\r) -- (8+\r,7-\r) -- (8+\r,5) -- cycle
  ;
  \node at (10.5,5) {$C$};
  \node at (2.25,5) {$D$};
\end{tikzpicture}
\]
Similarly to what we did in the previous step, we make a list of pairs
$(A,C)$ keeping the probe $P'+A+C$ viable (about 24 percent of the
pairs work). We also make a list of pairs $(B,D)$ keeping the probe
$P'+B+D$ viable (only about 0.8 of such pairs work). Next, we look for
triples $(A,B,D)$ keeping the probe $P'+A+B+D$ viable.

Here, we hit a snag. We do not find any viable triples $(A,B,D)$. 

\paragraph{Step 5.}

Since we were not able to complete the previous step, we try a
different subdivision of the board. Using more trial and error, we
arrive at the following subdivision:
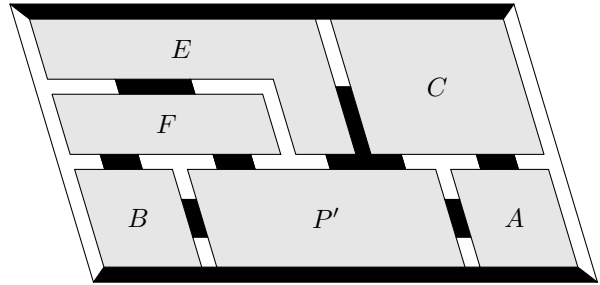
\[
\begin{tikzpicture}[xscale=0.5, yscale=0.5, xslant=-0.3]
  \pgfdeclarelayer{black}
  \pgfdeclarelayer{white}
  \pgfsetlayers{black,white,main}
  \def\r{0.2}

  \fill[black!10]
  (0,7) -- (13,7) -- (13,0) -- (0,0) -- cycle;

  \draw[fill=black,line join=round]
  (-\r,-\r) -- (\r,\r) -- (13-\r,\r) -- (13+\r,-\r) -- cycle
  (-\r,7+\r) -- (\r,7-\r) -- (13-\r,7-\r) -- (13+\r,7+\r) -- cycle;
  \draw[fill=white,line join=round]
  (-\r,-\r) -- (\r,\r) -- (\r,3-\r) -- (1,3-\r) -- (1,3+\r) -- (\r,3+\r) -- (\r,5-\r) -- (2,5-\r) -- (2,5+\r) -- (\r,5+\r) -- (\r,7-\r) -- (-\r,7+\r) -- cycle
  (13-\r,\r) -- (13+\r,-\r) -- (13+\r,7+\r) -- (13-\r,7-\r) -- (13-\r,3+\r) -- (12,3+\r) -- (12,3-\r) -- (13-\r,3-\r) -- cycle;

  \draw[fill=white]
  (3-\r,\r) -- (3-\r,1) -- (3+\r,1) -- (3+\r,\r) -- cycle
  (3-\r,2) -- (3-\r,3-\r) -- (2,3-\r) -- (2,3+\r) -- (4,3+\r) -- (4,3-\r) -- (3+\r,3-\r) -- (3+\r,2) -- cycle
  (5,3+\r) -- (6-\r,3+\r) -- (6-\r,5-\r) -- (4,5-\r) -- (4,5+\r) -- (6+\r,5+\r) -- (6+\r,3+\r) -- (7,3+\r) -- (7,3-\r) -- (5,3-\r) -- cycle
  (10+\r,2) -- (10+\r,3-\r) -- (11,3-\r) -- (11,3+\r) -- (9,3+\r) -- (9,3-\r) -- (10-\r,3-\r) -- (10-\r,2) -- cycle
  (10+\r,\r) -- (10+\r,1) -- (10-\r,1) -- (10-\r,\r) -- cycle
  ;
  \draw[fill=black]
  (3-\r,1) -- (3+\r,1) -- (3+\r,2) -- (3-\r,2) -- cycle
  (4,3+\r) -- (4,3-\r) -- (5,3-\r) -- (5,3+\r) -- cycle
  (7,3+\r) -- (8-\r,3+\r) -- (8-\r,5) -- (8+\r,5) -- (8+\r,3+\r) -- (9,3+\r) -- (9,3-\r) -- (7,3-\r) -- cycle
  (10-\r,1) -- (10+\r,1) -- (10+\r,2) -- (10-\r,2) -- cycle
  ;
  \node at (6.5,1.5) {$P'$};

  \draw[fill=black]
  (1,3-\r) -- (2,3-\r) -- (2,3+\r) -- (1,3+\r) -- cycle
  (11,3-\r) -- (12,3-\r) -- (12,3+\r) -- (11,3+\r) -- cycle
  ;
  \node at (11.5,1.5) {$A$};
  \node at (1.5,1.5) {$B$};

  \draw[fill=white]
  (8-\r,5) -- (8-\r,7-\r) -- (8+\r,7-\r) -- (8+\r,5) -- cycle
  ;
  \node at (10.5,5) {$C$};

  \draw[fill=black]
  (2,5-\r) -- (2,5+\r) -- (4,5+\r) -- (4,5-\r) -- cycle;
  \node at (4,6) {$E$};
  \node at (3,4) {$F$};
\end{tikzpicture}
\]
Note that regions $A$, $B$, $C$, and $E$ are corners, but $F$ is a
general 3-terminal position. Also note that region $C$ is the same as
in step~4, so we can re-use the list of viable pairs $(C,A)$ we have
already computed.  Continuing as above, we eventually find a tuple
$(A,B,C,E,F)$ of Hex-realizable values for these regions, such that
the probe $P'+A+B+C+E+F$ is viable. The simplest such tuple we found,
as measured by the total number of empty cells in its Hex realization,
is
\[
\begin{array}{ccl}
  A &=& \g{a,c|\g{a|\bot}},\\
  B &=& \g{a,c|\g{a|\bot}},\\
  C &=& \g{\g{\top|a},\g{\top|c,\g{b|\bot}}|\bot},\\
  E &=& \g{\g{\top|b}|\bot},\\
  F &=& \g{\g{\top|b}|b,\g{c|\bot}}.
\end{array}
\]

\paragraph{Step 6.}

The trial-and-error phase of the method is now completed. What remains
to be done is to assemble the witness position.  Looking up the Hex
realizations of $A$, $B$, $C$, $E$, and $F$ in the database of
Section~\ref{sec:database}, and connecting the appropriate terminals,
we obtain the following position:
\begin{equation}\label{eqn:witness-ziggurat-7}
\begin{hexboard}[baseline={($(current bounding box.center)-(0,1ex)$)},scale=0.8]
  \rotation{-30}
  \board(11,9)
  \white(1,1)\white(2,1)\white(3,1)\white(4,1)\white(5,1)\white(6,1)\white(7,1)\white(8,1)\white(9,1)\black(10,1)\black(11,1)
  \white(1,2)\white(2,2)\black(3,2)\black(4,2)\white(5,2)\black(6,2)\nofil(7,2)\black(8,2)\nofil(9,2)\white(10,2)\black(11,2)
  \white(1,3)\nofil(2,3)\nofil(3,3)\nofil(4,3)\black(5,3)\nofil(6,3)\nofil(7,3)\white(8,3)\nofil(9,3)\white(10,3)\nofil(11,3)
  \black(1,4)\white(2,4)\black(3,4)\nofil(4,4)\white(5,4)\white(6,4)\nofil(7,4)\white(8,4)\black(9,4)\nofil(10,4)\nofil(11,4)
  \black(1,5)\white(2,5)\nofil(3,5)\black(4,5)\white(5,5)\black(6,5)\white(7,5)\black(8,5)\white(9,5)\nofil(10,5)\black(11,5)
  \black(1,6)\white(2,6)\black(3,6)\white(4,6)\black(5,6)\white(6,6)\black(7,6)\white(8,6)\white(9,6)\black(10,6)\white(11,6)
  \black(1,7)\white(2,7)\black(3,7)\white(4,7)\black(5,7)\carry(6,7)\white(7,7)\black(8,7)\nofil(9,7)\nofil(10,7)\black(11,7)
  \black(1,8)\white(2,8)\black(3,8)\carry(4,8)\carry(5,8)\carry(6,8)\black(7,8)\nofil(8,8)\black(9,8)\nofil(10,8)\black(11,8)
  \black(1,9)\white(2,9)\carry(3,9)\carry(4,9)\carry(5,9)\carry(6,9)\white(7,9)\white(8,9)\white(9,9)\nofil(10,9)\white(11,9)
  \carry(5,7)
  \cell(6,7)\sflabel{1}
  \cell(4,8)\sflabel{2}
  \cell(5,8)\sflabel{3}
  \cell(6,8)\sflabel{4}
  \cell(3,9)\sflabel{5}
  \cell(4,9)\sflabel{6}
  \cell(5,9)\sflabel{7}
  \cell(6,9)\sflabel{8}
  \cell(10,7) \sflabel{$A$}
  \shaded{FF8080}(8,8)
  \shaded{FF8080}(9,7)
  \shaded{FF8080}(10,7)
  \shaded{FF8080}(10,8)
  \shaded{FF8080}(10,9)
  \cell(4,3) \sflabel{$B$}
  \shaded{FF8080}(2,3)
  \shaded{FF8080}(3,3)
  \shaded{FF8080}(4,3)
  \shaded{FF8080}(4,4)
  \shaded{FF8080}(3,5)
  \cell(10,4) \sflabel{$C$}
  \shaded{FF8080}(10,4)
  \shaded{FF8080}(10,5)
  \shaded{FF8080}(11,3)
  \shaded{FF8080}(11,4)
  \cell(9,2) \sflabel{$E$}
  \shaded{FF8080}(9,2)
  \shaded{FF8080}(9,3)
  \cell(7,3) \sflabel{$F$}
  \shaded{FF8080}(6,3)
  \shaded{FF8080}(7,2)
  \shaded{FF8080}(7,3)
  \shaded{FF8080}(7,4)
\end{hexboard}
\end{equation}
Here, we have labelled the regions $A$--$F$ to make them more easily
recognizable. By construction, the probe $P'+A+B+C+E+F$ is viable;
therefore, Lemma~\ref{lem:viable-follower} guarantees that some
follower of $A+B+C+E+F$ is the desired witness of non-inferiority, in
which probe~7 is the unique winning move for White. In practice, since
we typically enumerate the possible contexts in order of increasing
depth, it is almost always the case that $A+B+C+E+F$ itself is the
witness.

But can one trust this to work? The required computations took several
days, and we could have easily made a mistake somewhere. Luckily, one
does not have to take our word for it; it is easy to check the final
answer using a Hex solver such as Mohex. Indeed, Mohex easily confirms
that probe~7 is the only winning move in
{\eqref{eqn:witness-ziggurat-7}}.

\paragraph{Discussion.}

The method described in this section requires both heavy computation
and significant human guidance. Computations are used for such tasks
as searching large databases of Hex-realizable values, computing
outcomes and sums, and determining the viability of probes. Human
judgement is needed for tasks such as proposing subdivisions of the
board, deciding in which order to fill them, when to backtrack, etc.
On the whole, this is a labor intensive process. It may be possible to
better automate this process in the future, but we have not attempted
to do so.

We also note that the witness position
{\eqref{eqn:witness-ziggurat-7}} is very complex (certainly much more
complex than any position that is likely to arise naturally during
game play). While one can computationally verify that probe 7 is the
only winning move for White, we do not have a human-level explanation
for why this is so. Indeed, finding the winning move in this position
is a difficult puzzle that would likely stump most Hex masters.

\subsection{Results}

\subsubsection{Probes in some common Hex templates}

We investigated the inferiority or non-inferiority of probes into many
common Hex templates. Figure~\ref{fig:named-templates} shows a number
of named templates, and we have numbered the probes in each
template.
\begin{figure}
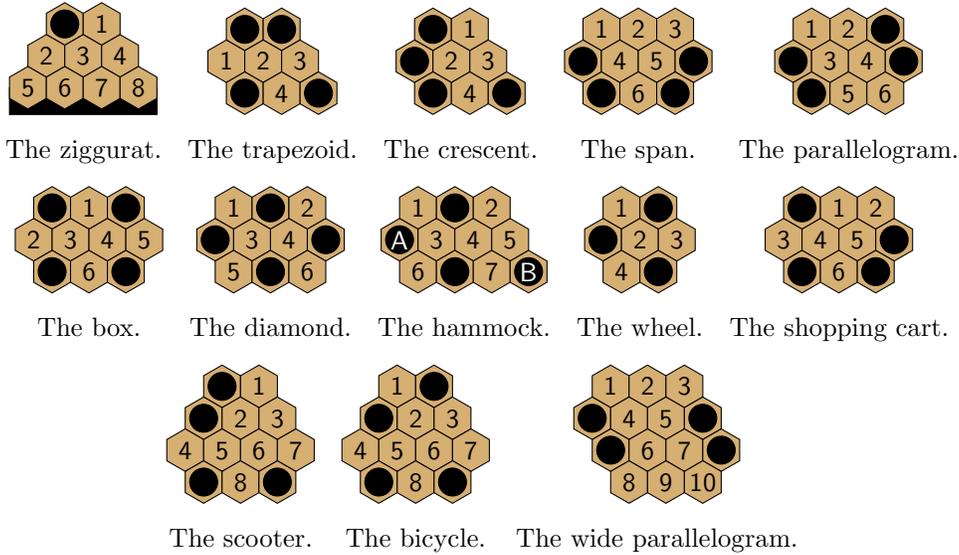

  \def\scale{0.7}
  \setlength{\increment}{0.175cm}
  \newcommand{\offset}{\hspace{2em}}
  \[
  \begin{array}[b]{c}
    \begin{hexboard}[baseline={($(current bounding box.south)-(0,1ex)$)},scale=\scale]
      \template(4,3)
      \foreach \i in {3,...,4} {\hex(\i,1)}
      \foreach \i in {2,...,4} {\hex(\i,2)}
      \foreach \i in {1,...,4} {\hex(\i,3)}
      \black(3,1)
      \cell(4,1)\sflabel{1}
      \cell(2,2)\sflabel{2}
      \cell(3,2)\sflabel{3}
      \cell(4,2)\sflabel{4}
      \cell(1,3)\sflabel{5}
      \cell(2,3)\sflabel{6}
      \cell(3,3)\sflabel{7}
      \cell(4,3)\sflabel{8}
    \end{hexboard}
    \\
    \mbox{The ziggurat.}
  \end{array}
  \begin{array}[b]{c}
    \begin{hexboard}[baseline={($(current bounding box.south)-(0,1ex)$)},scale=\scale]
      \rotation{-30}
      \foreach\i in {2,...,3} {\hex(\i,1)}
      \foreach\i in {1,...,3} {\hex(\i,2)}
      \foreach\i in {1,...,3} {\hex(\i,3)}
      \black(2,1)
      \black(3,1)
      \black(1,3)
      \black(3,3)
      \cell(1,2)\sflabel{1}
      \cell(2,2)\sflabel{2}
      \cell(3,2)\sflabel{3}
      \cell(2,3)\sflabel{4}
    \end{hexboard}
    \\
    \mbox{The trapezoid.}
  \end{array}
  \begin{array}[b]{c}
    \begin{hexboard}[baseline={($(current bounding box.south)-(0,1ex)$)},scale=\scale]
      \rotation{-30}
      \foreach\i in {2,...,3} {\hex(\i,1)}
      \foreach\i in {1,...,3} {\hex(\i,2)}
      \foreach\i in {1,...,3} {\hex(\i,3)}
      \black(2,1)
      \black(1,2)
      \black(1,3)
      \black(3,3)
      \cell(3,1)\sflabel{1}
      \cell(2,2)\sflabel{2}
      \cell(3,2)\sflabel{3}
      \cell(2,3)\sflabel{4}
    \end{hexboard}
    \\
    \mbox{The crescent.}
  \end{array}
  \begin{array}[b]{c}
    \begin{hexboard}[baseline={($(current bounding box.south)-(0,1ex)$)},scale=\scale]
      \rotation{-30}
      \foreach\i in {2,...,4} {\hex(\i,1)}
      \foreach\i in {1,...,4} {\hex(\i,2)}
      \foreach\i in {1,...,3} {\hex(\i,3)}
      \black(1,2)
      \black(1,3)
      \black(3,3)
      \black(4,2)
      \cell(2,1)\sflabel{1}
      \cell(3,1)\sflabel{2}
      \cell(4,1)\sflabel{3}
      \cell(2,2)\sflabel{4}
      \cell(3,2)\sflabel{5}
      \cell(2,3)\sflabel{6}
    \end{hexboard}
    \\
    \mbox{The span.}
  \end{array}
  \begin{array}[b]{c}
    \begin{hexboard}[baseline={($(current bounding box.south)-(0,1ex)$)},scale=\scale]
      \rotation{-30}
      \foreach\i in {2,...,4} {\hex(\i,1)}
      \foreach\i in {1,...,4} {\hex(\i,2)}
      \foreach\i in {1,...,3} {\hex(\i,3)}
      \black(1,2)
      \black(1,3)
      \black(4,1)
      \black(4,2)
      \cell(2,1)\sflabel{1}
      \cell(3,1)\sflabel{2}
      \cell(2,2)\sflabel{3}
      \cell(3,2)\sflabel{4}
      \cell(2,3)\sflabel{5}
      \cell(3,3)\sflabel{6}
    \end{hexboard}
    \\
    \mbox{The parallelogram.}
  \end{array}
  \]
  \[
  \begin{array}[b]{c}
    \begin{hexboard}[baseline={($(current bounding box.south)-(0,1ex)$)},scale=\scale]
      \rotation{-30}
      \foreach\i in {2,...,4} {\hex(\i,1)}
      \foreach\i in {1,...,4} {\hex(\i,2)}
      \foreach\i in {1,...,3} {\hex(\i,3)}
      \black(2,1)
      \black(1,3)
      \black(4,1)
      \black(3,3)
      \cell(3,1)\sflabel{1}
      \cell(1,2)\sflabel{2}
      \cell(2,2)\sflabel{3}
      \cell(3,2)\sflabel{4}
      \cell(4,2)\sflabel{5}
      \cell(2,3)\sflabel{6}
    \end{hexboard}
    \\
    \mbox{The box.}
  \end{array}
  \begin{array}[b]{c}
    \begin{hexboard}[baseline={($(current bounding box.south)-(0,1ex)$)},scale=\scale]
      \rotation{-30}
      \foreach\i in {2,...,4} {\hex(\i,1)}
      \foreach\i in {1,...,4} {\hex(\i,2)}
      \foreach\i in {1,...,3} {\hex(\i,3)}
      \black(1,2)
      \black(3,1)
      \black(2,3)
      \black(4,2)
      \cell(2,1)\sflabel{1}
      \cell(4,1)\sflabel{2}
      \cell(2,2)\sflabel{3}
      \cell(3,2)\sflabel{4}
      \cell(1,3)\sflabel{5}
      \cell(3,3)\sflabel{6}
    \end{hexboard}
    \\
    \mbox{The diamond.}
  \end{array}
  \begin{array}[b]{c}
    \begin{hexboard}[baseline={($(current bounding box.south)-(0,1ex)$)},scale=\scale]
      \rotation{-30}
      \foreach\i in {2,...,4} {\hex(\i,1)}
      \foreach\i in {1,...,4} {\hex(\i,2)}
      \foreach\i in {1,...,4} {\hex(\i,3)}
      \black(1,2)\sflabel{A}
      \black(2,3)
      \black(3,1)
      \black(4,3)\sflabel{B}
      \cell(2,1)\sflabel{1}
      \cell(4,1)\sflabel{2}
      \cell(2,2)\sflabel{3}
      \cell(3,2)\sflabel{4}
      \cell(4,2)\sflabel{5}
      \cell(1,3)\sflabel{6}
      \cell(3,3)\sflabel{7}
    \end{hexboard}
    \\
    \mbox{The hammock.}
  \end{array}
  \begin{array}[b]{c}
    \begin{hexboard}[baseline={($(current bounding box.south)-(0,1ex)$)},scale=\scale]
      \rotation{-30}
      \foreach\i in {2,...,3} {\hex(\i,1)}
      \foreach\i in {1,...,3} {\hex(\i,2)}
      \foreach\i in {1,...,2} {\hex(\i,3)}
      \black(1,2)
      \black(3,1)
      \black(2,3)
      \cell(2,1)\sflabel{1}
      \cell(2,2)\sflabel{2}
      \cell(3,2)\sflabel{3}
      \cell(1,3)\sflabel{4}
    \end{hexboard}
    \\
    \mbox{The wheel.}
  \end{array}
  \begin{array}[b]{c}
    \begin{hexboard}[baseline={($(current bounding box.south)-(0,1ex)$)},scale=\scale]
      \rotation{-30}
      \foreach\i in {2,...,4} {\hex(\i,1)}
      \foreach\i in {1,...,4} {\hex(\i,2)}
      \foreach\i in {1,...,3} {\hex(\i,3)}
      \black(2,1)
      \black(1,3)
      \black(4,2)
      \black(3,3)
      \cell(3,1)\sflabel{1}
      \cell(4,1)\sflabel{2}
      \cell(1,2)\sflabel{3}
      \cell(2,2)\sflabel{4}
      \cell(3,2)\sflabel{5}
      \cell(2,3)\sflabel{6}
    \end{hexboard}
    \\
    \mbox{The shopping cart.}
  \end{array}
  \]
  \[
  \begin{array}[b]{c}
    \begin{hexboard}[baseline={($(current bounding box.south)-(0,1ex)$)},scale=\scale]
      \rotation{-30}
      \foreach\i in {3,...,4} {\hex(\i,1)}
      \foreach\i in {2,...,4} {\hex(\i,2)}
      \foreach\i in {1,...,4} {\hex(\i,3)}
      \foreach\i in {1,...,3} {\hex(\i,4)}
      \black(3,1)
      \black(2,2)
      \black(1,4)
      \black(3,4)
      \cell(4,1)\sflabel{1}
      \cell(3,2)\sflabel{2}
      \cell(4,2)\sflabel{3}
      \cell(1,3)\sflabel{4}
      \cell(2,3)\sflabel{5}
      \cell(3,3)\sflabel{6}
      \cell(4,3)\sflabel{7}
      \cell(2,4)\sflabel{8}
    \end{hexboard}
    \\
    \mbox{The scooter.}
  \end{array}
  \begin{array}[b]{c}
    \begin{hexboard}[baseline={($(current bounding box.south)-(0,1ex)$)},scale=\scale]
      \rotation{-30}
      \foreach\i in {3,...,4} {\hex(\i,1)}
      \foreach\i in {2,...,4} {\hex(\i,2)}
      \foreach\i in {1,...,4} {\hex(\i,3)}
      \foreach\i in {1,...,3} {\hex(\i,4)}
      \black(4,1)
      \black(2,2)
      \black(1,4)
      \black(3,4)
      \cell(3,1)\sflabel{1}
      \cell(3,2)\sflabel{2}
      \cell(4,2)\sflabel{3}
      \cell(1,3)\sflabel{4}
      \cell(2,3)\sflabel{5}
      \cell(3,3)\sflabel{6}
      \cell(4,3)\sflabel{7}
      \cell(2,4)\sflabel{8}
    \end{hexboard}
    \\
    \mbox{The bicycle.}
  \end{array}
  \begin{array}[b]{c}
    \begin{hexboard}[baseline={($(current bounding box.south)-(0,1ex)$)},scale=\scale]
      \rotation{-30}
      \foreach\i in {2,...,4} {\hex(\i,1)}
      \foreach\i in {1,...,4} {\hex(\i,2)}
      \foreach\i in {1,...,4} {\hex(\i,3)}
      \foreach\i in {1,...,3} {\hex(\i,4)}
      \black(1,2)
      \black(1,3)
      \black(4,2)
      \black(4,3)
      \cell(2,1)\sflabel{1}
      \cell(3,1)\sflabel{2}
      \cell(4,1)\sflabel{3}
      \cell(2,2)\sflabel{4}
      \cell(3,2)\sflabel{5}
      \cell(2,3)\sflabel{6}
      \cell(3,3)\sflabel{7}
      \cell(1,4)\sflabel{8}
      \cell(2,4)\sflabel{9}
      \cell(3,4)\sflabel{10}
    \end{hexboard}
    \\
    \mbox{The wide parallelogram.}
  \end{array}
  \]
  \caption{Some Hex templates. The ziggurat is an edge template, and
    the others are interior templates. In the hammock, Black is
    guaranteed to connect the stones labelled $A$ and $B$ to each
    other. In the remaining interior templates, Black can connect all
    of the black stones.}
  \label{fig:named-templates}
\end{figure}

To prove that a probe is non-inferior, it suffices to provide a
witness position as discussed above. For proving that a probe $x$ is
inferior, there are several methods. The most common such methods, and
the only ones that we need here, are to show that $x$ is
\emph{dominated} by some other move in the template, or to show that
$x$ is \emph{strongly reversible}. Here, a white move $x$ is strongly
reversible if there exists a black response $y$ so that if White has a
winning move in the resulting position, then White already had a
winning move other than $x$ before $x$ and $y$ were played. For a more
detailed discussion of inferiority proofs and additional methods and
examples, see Henderson and Hayward
{\cite{Henderson-Hayward,star-decomposition}}.

We found that in the templates of Figure~\ref{fig:named-templates},
the following probes are inferior:
\begin{itemize}
\item In the trapezoid, probe $2$ is inferior. It is
  strongly reversed by $1$.
\item In the crescent, probes $2$ and $4$
  are inferior. Both are strongly reversed by $3$.
  Probe $2$ is also strongly reversed by $1$.
\item In the span, probes $4$ and $5$ are
  inferior. They are dominated by $2$, using
  star-decomposition domination {\cite{star-decomposition}}.
\item In the parallelogram, probes $3$ and
  $4$ are inferior. They are strongly reversed by
  $5$ and $6$, respectively.
\end{itemize}

All other probes in all templates of Figure~\ref{fig:named-templates}
are non-inferior, with witnesses shown in
Figures~\ref{fig:witnesses-part1} and {\ref{fig:witnesses-part2}}. In
each witness position, the cell marked with a white dot is the unique
winning move for White. In choosing these witness positions, we
generally tried to minimize the number of empty cells (rather than
minimizing the number of stones), because the number of empty cells is
a more meaningful measure of a position's complexity.

\begin{figure}
  \def\scale{0.5}
  \setlength{\increment}{0.175cm}
  \newcommand{\offset}{\hspace{1.7em}}
  Ziggurat:
  \[
  \begin{hexboard}[baseline={($(current bounding box.south)-(0,1ex)$)},scale=\scale]
    \rotation{-30}
    \board(6,4)
    \white(1,1)\white(2,1)\white(3,1)\white(4,1)\white(5,1)\black(6,1)
    \white(1,2)\black(2,2)\black(3,2)\black(4,2)\carry(5,2)\white(6,2)
    \black(1,3)\white(2,3)\carry(3,3)\carry(4,3)\carry(5,3)\white(6,3)
    \black(1,4)\carry(2,4)\carry(3,4)\carry(4,4)\carry(5,4)\white(6,4)
    \carry(4,2)
    \cell(5,2)\winmove
  \end{hexboard}
  \offset\hspace{-4\increment}
  \begin{hexboard}[baseline={($(current bounding box.south)-(0,1ex)$)},scale=\scale]
    \rotation{-30}
    \board(4,4)
    \white(1,1)\black(2,1)\nofil(3,1)\white(4,1)
    \black(1,2)\white(2,2)\black(3,2)\carry(4,2)
    \nofil(1,3)\carry(2,3)\carry(3,3)\carry(4,3)
    \carry(1,4)\carry(2,4)\carry(3,4)\carry(4,4)
    \carry(3,2)
    \cell(2,3)\winmove
  \end{hexboard}
  \offset\hspace{-5\increment}
  \begin{hexboard}[baseline={($(current bounding box.south)-(0,1ex)$)},scale=\scale]
    \rotation{-30}
    \board(7,5)
    \white(1,1)\white(2,1)\white(3,1)\white(4,1)\white(5,1)\white(6,1)\black(7,1)
    \white(1,2)\white(2,2)\white(3,2)\white(4,2)\nofil(5,2)\nofil(6,2)\black(7,2)
    \white(1,3)\white(2,3)\white(3,3)\black(4,3)\carry(5,3)\white(6,3)\black(7,3)
    \white(1,4)\black(2,4)\carry(3,4)\carry(4,4)\carry(5,4)\nofil(6,4)\white(7,4)
    \black(1,5)\carry(2,5)\carry(3,5)\carry(4,5)\carry(5,5)\white(6,5)\white(7,5)
    \carry(4,3)
    \cell(4,4)\winmove
  \end{hexboard}
  \offset\hspace{-4\increment}
  \begin{hexboard}[baseline={($(current bounding box.south)-(0,1ex)$)},scale=\scale]
    \rotation{-30}
    \board(5,4)
    \white(1,1)
    \white(2,1)
    \white(3,1)
    \white(4,1)
    \white(1,2)
    \white(2,2)
    \white(1,3)
    \black(5,1)
    \black(3,2)
    \carry(3,2)
    \black(5,2)
    \white(5,3)
    \white(5,4)
    \carry(4,2)
    \carry(2,3)
    \carry(3,3)
    \carry(4,3)\winmove
    \carry(1,4)
    \carry(2,4)
    \carry(3,4)
    \carry(4,4)
  \end{hexboard}
  \offset\hspace{-7\increment}
  \begin{hexboard}[baseline={($(current bounding box.south)-(0,1ex)$)},scale=\scale]
    \rotation{-30}
    \board(7,7)
    \white(1,1)\white(2,1)\white(3,1)\white(4,1)\black(5,1)\white(6,1)\nofil(7,1)
    \white(1,2)\white(2,2)\nofil(3,2)\nofil(4,2)\nofil(5,2)\nofil(6,2)\nofil(7,2)
    \white(1,3)\black(2,3)\nofil(3,3)\white(4,3)\nofil(5,3)\nofil(6,3)\white(7,3)
    \nofil(1,4)\nofil(2,4)\black(3,4)\black(4,4)\white(5,4)\black(6,4)\white(7,4)
    \black(1,5)\white(2,5)\black(3,5)\carry(4,5)\white(5,5)\black(6,5)\white(7,5)
    \black(1,6)\carry(2,6)\carry(3,6)\carry(4,6)\black(5,6)\white(6,6)\white(7,6)
    \carry(1,7)\carry(2,7)\carry(3,7)\carry(4,7)\white(5,7)\white(6,7)\white(7,7)
    \carry(3,5)
    \cell(1,7)\winmove
  \end{hexboard}
  \]
  \[
  \begin{hexboard}[baseline={($(current bounding box.south)-(0,1ex)$)},scale=\scale]
    \rotation{-30}
    \board(6,5)
    \white(1,1)\white(2,1)\white(3,1)\white(4,1)\black(5,1)\white(6,1)
    \white(1,2)\white(2,2)\white(3,2)\nofil(4,2)\nofil(5,2)\black(6,2)
    \white(1,3)\white(2,3)\black(3,3)\carry(4,3)\white(5,3)\black(6,3)
    \white(1,4)\carry(2,4)\carry(3,4)\carry(4,4)\nofil(5,4)\white(6,4)
    \carry(1,5)\carry(2,5)\carry(3,5)\carry(4,5)\black(5,5)\white(6,5)
    \carry(3,3)
    \cell(2,5)\winmove
  \end{hexboard}
  \offset\hspace{-9\increment}
  \begin{hexboard}[baseline={($(current bounding box.south)-(0,1ex)$)},scale=\scale]
    \rotation{-30}
    \board(11,9)
    \white(1,1)\white(2,1)\white(3,1)\white(4,1)\white(5,1)\white(6,1)\white(7,1)\white(8,1)\white(9,1)\black(10,1)\black(11,1)
    \white(1,2)\white(2,2)\black(3,2)\black(4,2)\white(5,2)\black(6,2)\nofil(7,2)\black(8,2)\nofil(9,2)\white(10,2)\black(11,2)
    \white(1,3)\nofil(2,3)\nofil(3,3)\nofil(4,3)\black(5,3)\nofil(6,3)\nofil(7,3)\white(8,3)\nofil(9,3)\white(10,3)\nofil(11,3)
    \black(1,4)\white(2,4)\black(3,4)\nofil(4,4)\white(5,4)\white(6,4)\nofil(7,4)\white(8,4)\black(9,4)\nofil(10,4)\nofil(11,4)
    \black(1,5)\white(2,5)\nofil(3,5)\black(4,5)\white(5,5)\black(6,5)\white(7,5)\black(8,5)\white(9,5)\nofil(10,5)\black(11,5)
    \black(1,6)\white(2,6)\black(3,6)\white(4,6)\black(5,6)\white(6,6)\black(7,6)\white(8,6)\white(9,6)\black(10,6)\white(11,6)
    \black(1,7)\white(2,7)\black(3,7)\white(4,7)\black(5,7)\carry(6,7)\white(7,7)\black(8,7)\nofil(9,7)\nofil(10,7)\black(11,7)
    \black(1,8)\white(2,8)\black(3,8)\carry(4,8)\carry(5,8)\carry(6,8)\black(7,8)\nofil(8,8)\black(9,8)\nofil(10,8)\black(11,8)
    \black(1,9)\white(2,9)\carry(3,9)\carry(4,9)\carry(5,9)\carry(6,9)\white(7,9)\white(8,9)\white(9,9)\nofil(10,9)\white(11,9)
    \carry(5,7)
    \cell(5,9)\winmove
  \end{hexboard}
  \offset\hspace{-7\increment}
  \begin{hexboard}[baseline={($(current bounding box.south)-(0,1ex)$)},scale=\scale]
    \rotation{-30}
    \board(9,7)
    \white(1,1)\white(2,1)\white(3,1)\white(4,1)\white(5,1)\white(6,1)\black(7,1)\white(8,1)\black(9,1)
    \white(1,2)\white(2,2)\black(3,2)\black(4,2)\black(5,2)\nofil(6,2)\white(7,2)\nofil(8,2)\nofil(9,2)
    \white(1,3)\nofil(2,3)\white(3,3)\nofil(4,3)\white(5,3)\nofil(6,3)\nofil(7,3)\nofil(8,3)\black(9,3)
    \white(1,4)\nofil(2,4)\nofil(3,4)\white(4,4)\white(5,4)\black(6,4)\nofil(7,4)\black(8,4)\white(9,4)
    \nofil(1,5)\white(2,5)\black(3,5)\black(4,5)\carry(5,5)\white(6,5)\black(7,5)\white(8,5)\white(9,5)
    \black(1,6)\white(2,6)\carry(3,6)\carry(4,6)\carry(5,6)\black(6,6)\white(7,6)\white(8,6)\white(9,6)
    \black(1,7)\carry(2,7)\carry(3,7)\carry(4,7)\carry(5,7)\white(6,7)\white(7,7)\white(8,7)\white(9,7)
    \carry(4,5)
    \cell(5,7)\winmove
  \end{hexboard}
  \]
  \begin{minipage}[t]{0.6\textwidth}
  Trapezoid:
  \[
  \begin{hexboard}[baseline={($(current bounding box.south)-(0,1ex)$)},scale=\scale]
    \rotation{-30}
    \board(4,4)
    \black(1,1)\black(2,1)\white(3,1)\white(4,1)
    \black(1,2)\white(2,2)\black(3,2)\black(4,2)
    \black(1,3)\carry(2,3)\carry(3,3)\carry(4,3)
    \white(1,4)\black(2,4)\carry(3,4)\black(4,4)
    \carry(3,2)
    \carry(4,2)
    \carry(2,4)
    \carry(4,4)
    \cell(2,3)\winmove
  \end{hexboard}
  \offset\hspace{-6\increment}
  \begin{hexboard}[baseline={($(current bounding box.south)-(0,1ex)$)},scale=\scale]
    \rotation{-30}
    \board(6,5)
    \white(1,1)\white(2,1)\black(3,1)\white(4,1)\black(5,1)\nofil(6,1)
    \white(1,2)\black(2,2)\nofil(3,2)\white(4,2)\carry(5,2)\black(6,2)
    \black(1,3)\white(2,3)\nofil(3,3)\black(4,3)\carry(5,3)\black(6,3)
    \nofil(1,4)\nofil(2,4)\white(3,4)\carry(4,4)\carry(5,4)\white(6,4)
    \white(1,5)\nofil(2,5)\black(3,5)\black(4,5)\white(5,5)\white(6,5)
    \carry(6,2)
    \carry(6,3)
    \carry(4,3)
    \carry(4,5)
    \cell(5,4)\winmove
  \end{hexboard}
  \offset\hspace{-6\increment}
  \begin{hexboard}[baseline={($(current bounding box.south)-(0,1ex)$)},scale=\scale]
    \rotation{-30}
    \board(6,5)
    \white(1,1)\black(2,1)\white(3,1)\white(4,1)\white(5,1)\black(6,1)
    \nofil(1,2)\carry(2,2)\black(3,2)\white(4,2)\nofil(5,2)\black(6,2)
    \black(1,3)\carry(2,3)\black(3,3)\nofil(4,3)\white(5,3)\black(6,3)
    \carry(1,4)\carry(2,4)\white(3,4)\nofil(4,4)\nofil(5,4)\white(6,4)
    \black(1,5)\white(2,5)\nofil(3,5)\white(4,5)\white(5,5)\white(6,5)
    \carry(3,2)
    \carry(1,3)
    \carry(3,3)
    \carry(1,5)
    \cell(1,4)\winmove
  \end{hexboard}
  \]
  \end{minipage}
  \begin{minipage}[t]{0.35\textwidth}
  Crescent:
  \[
  \begin{hexboard}[baseline={($(current bounding box.south)-(0,1ex)$)},scale=\scale]
    \rotation{-30}
    \board(4,4)
    \white(1,1)\white(2,1)\white(3,1)\black(4,1)
    \white(1,2)\black(2,2)\carry(3,2)\white(4,2)
    \black(1,3)\carry(2,3)\carry(3,3)\white(4,3)
    \black(1,4)\carry(2,4)\black(3,4)\white(4,4)
    \carry(2,2)
    \carry(1,3)
    \carry(1,4)
    \carry(3,4)
    \cell(3,2)\winmove
  \end{hexboard}
  \offset\hspace{-5\increment}
  \begin{hexboard}[baseline={($(current bounding box.south)-(0,1ex)$)},scale=\scale]
    \rotation{-30}
    \board(4,5)
    \white(1,1)\white(2,1)\white(3,1)\black(4,1)
    \white(1,2)\black(2,2)\carry(3,2)\black(4,2)
    \black(1,3)\carry(2,3)\carry(3,3)\white(4,3)
    \black(1,4)\carry(2,4)\black(3,4)\white(4,4)
    \white(1,5)\black(2,5)\white(3,5)\white(4,5)
    \carry(2,2)
    \carry(1,3)
    \carry(1,4)
    \carry(3,4)
    \cell(3,3)\winmove
  \end{hexboard}
  \]
  \end{minipage}

  \begin{minipage}[t]{0.6\textwidth}
  Span:
  \[
  \begin{hexboard}[baseline={($(current bounding box.south)-(0,1ex)$)},scale=\scale]
    \rotation{-30}
    \board(4,4)
    \white(1,1)\black(2,1)\white(3,1)\white(4,1)
    \white(1,2)\carry(2,2)\carry(3,2)\carry(4,2)
    \black(1,3)\carry(2,3)\carry(3,3)\black(4,3)
    \black(1,4)\carry(2,4)\black(3,4)\white(4,4)
    \carry(1,3)
    \carry(1,4)
    \carry(3,4)
    \carry(4,3)
    \cell(2,2)\winmove
  \end{hexboard}
  \offset\hspace{-5\increment}
  \begin{hexboard}[baseline={($(current bounding box.south)-(0,1ex)$)},scale=\scale]
    \rotation{-30}
    \board(4,5)
    \white(1,1)\white(2,1)\white(3,1)\black(4,1)
    \white(1,2)\carry(2,2)\carry(3,2)\carry(4,2)
    \black(1,3)\carry(2,3)\carry(3,3)\black(4,3)
    \black(1,4)\carry(2,4)\black(3,4)\white(4,4)
    \black(1,5)\white(2,5)\white(3,5)\white(4,5)
    \carry(1,3)
    \carry(1,4)
    \carry(3,4)
    \carry(4,3)
    \cell(3,2)\winmove
  \end{hexboard}
  \offset\hspace{-7\increment}
  \begin{hexboard}[baseline={($(current bounding box.south)-(0,1ex)$)},scale=\scale]
    \rotation{-30}
    \board(6,7)
    \white(1,1)\white(2,1)\black(3,1)\white(4,1)\white(5,1)\black(6,1)
    \white(1,2)\black(2,2)\white(3,2)\black(4,2)\carry(5,2)\white(6,2)
    \white(1,3)\nofil(2,3)\black(3,3)\carry(4,3)\carry(5,3)\white(6,3)
    \black(1,4)\white(2,4)\carry(3,4)\carry(4,4)\carry(5,4)\black(6,4)
    \black(1,5)\white(2,5)\black(3,5)\black(4,5)\white(5,5)\black(6,5)
    \black(1,6)\nofil(2,6)\white(3,6)\white(4,6)\white(5,6)\black(6,6)
    \white(1,7)\black(2,7)\white(3,7)\white(4,7)\white(5,7)\black(6,7)
    \carry(4,2)
    \carry(3,3)
    \carry(3,5)
    \carry(4,5)
    \cell(3,4)\winmove
  \end{hexboard}
  \]
  \end{minipage}
  \begin{minipage}[t]{0.35\textwidth}
  Parallelogram:
  \[
  \begin{hexboard}[baseline={($(current bounding box.south)-(0,1ex)$)},scale=\scale]
    \rotation{-30}
    \board(4,4)
    \white(1,1)\black(2,1)\white(3,1)\white(4,1)
    \white(1,2)\carry(2,2)\carry(3,2)\black(4,2)
    \black(1,3)\carry(2,3)\carry(3,3)\black(4,3)
    \black(1,4)\carry(2,4)\carry(3,4)\white(4,4)
    \carry(1,3)
    \carry(1,4)
    \carry(4,2)
    \carry(4,3)
    \cell(2,2)\winmove
  \end{hexboard}
  \offset\hspace{-5\increment}
  \begin{hexboard}[baseline={($(current bounding box.south)-(0,1ex)$)},scale=\scale]
    \rotation{-30}
    \board(4,5)
    \black(1,1)\black(2,1)\black(3,1)\white(4,1)
    \white(1,2)\carry(2,2)\carry(3,2)\black(4,2)
    \black(1,3)\carry(2,3)\carry(3,3)\black(4,3)
    \black(1,4)\carry(2,4)\carry(3,4)\black(4,4)
    \white(1,5)\white(2,5)\white(3,5)\black(4,5)
    \carry(1,3)
    \carry(1,4)
    \carry(4,2)
    \carry(4,3)
    \cell(3,2)\winmove
  \end{hexboard}
  \]
  \end{minipage}

  \begin{minipage}[t]{0.6\textwidth}
  Box:
  \[
  \begin{hexboard}[baseline={($(current bounding box.south)-(0,1ex)$)},scale=\scale]
    \rotation{-30}
    \board(5,7)
    \white(1,1)\white(2,1)\black(3,1)\black(4,1)\white(5,1)
    \white(1,2)\black(2,2)\white(3,2)\nofil(4,2)\black(5,2)
    \white(1,3)\carry(2,3)\black(3,3)\white(4,3)\black(5,3)
    \black(1,4)\carry(2,4)\carry(3,4)\white(4,4)\black(5,4)
    \carry(1,5)\carry(2,5)\black(3,5)\nofil(4,5)\white(5,5)
    \black(1,6)\carry(2,6)\white(3,6)\black(4,6)\white(5,6)
    \white(1,7)\black(2,7)\black(3,7)\white(4,7)\white(5,7)
    \carry(3,3)
    \carry(1,4)
    \carry(3,5)
    \carry(1,6)
    \cell(3,4)\winmove
  \end{hexboard}
  \offset\hspace{-4\increment}
  \begin{hexboard}[baseline={($(current bounding box.south)-(0,1ex)$)},scale=\scale]
    \rotation{-30}
    \board(5,4)
    \black(1,1)\white(2,1)\white(3,1)\white(4,1)\white(5,1)
    \black(1,2)\white(2,2)\black(3,2)\carry(4,2)\black(5,2)
    \black(1,3)\carry(2,3)\carry(3,3)\carry(4,3)\carry(5,3)
    \white(1,4)\black(2,4)\carry(3,4)\black(4,4)\white(5,4)
    \carry(2,4)
    \carry(3,2)
    \carry(4,4)
    \carry(5,2)
    \cell(2,3)\winmove
  \end{hexboard}
  \offset\hspace{-5\increment}
  \begin{hexboard}[baseline={($(current bounding box.south)-(0,1ex)$)},scale=\scale]
    \rotation{-30}
    \board(6,5)
    \white(1,1)\carry(2,1)\black(3,1)\white(4,1)\nofil(5,1)\white(6,1)
    \black(1,2)\carry(2,2)\carry(3,2)\white(4,2)\nofil(5,2)\nofil(6,2)
    \carry(1,3)\carry(2,3)\black(3,3)\nofil(4,3)\white(5,3)\black(6,3)
    \black(1,4)\carry(2,4)\white(3,4)\nofil(4,4)\black(5,4)\white(6,4)
    \nofil(1,5)\black(2,5)\white(3,5)\black(4,5)\white(5,5)\white(6,5)
    \carry(3,1)
    \carry(1,2)
    \carry(3,3)
    \carry(1,4)
    \cell(2,2)\winmove
  \end{hexboard}
  \]
  \end{minipage}
  \begin{minipage}[t]{0.35\textwidth}
  Diamond (and wheel):
  \[
  \begin{hexboard}[baseline={($(current bounding box.south)-(0,1ex)$)},scale=\scale]
    \rotation{-30}
    \board(4,4)
    \white(1,1)\black(2,1)\white(3,1)\white(4,1)
    \white(1,2)\carry(2,2)\black(3,2)\carry(4,2)
    \black(1,3)\carry(2,3)\carry(3,3)\black(4,3)
    \carry(1,4)\black(2,4)\carry(3,4)\black(4,4)
    \carry(1,3)
    \carry(3,2)
    \carry(2,4)
    \carry(4,3)
    \cell(2,2)\winmove
  \end{hexboard}
  \offset\hspace{-8\increment}
  \begin{hexboard}[baseline={($(current bounding box.south)-(0,1ex)$)},scale=\scale]
    \rotation{-30}
    \board(9,8)
    \white(1,1)\white(2,1)\white(3,1)\white(4,1)\white(5,1)\white(6,1)\black(7,1)\carry(8,1)\white(9,1)
    \white(1,2)\white(2,2)\white(3,2)\white(4,2)\black(5,2)\carry(6,2)\carry(7,2)\black(8,2)\white(9,2)
    \white(1,3)\white(2,3)\black(3,3)\nofil(4,3)\white(5,3)\black(6,3)\carry(7,3)\carry(8,3)\black(9,3)
    \white(1,4)\black(2,4)\white(3,4)\nofil(4,4)\white(5,4)\carry(6,4)\black(7,4)\nofil(8,4)\black(9,4)
    \white(1,5)\nofil(2,5)\nofil(3,5)\black(4,5)\black(5,5)\white(6,5)\nofil(7,5)\black(8,5)\white(9,5)
    \black(1,6)\white(2,6)\nofil(3,6)\white(4,6)\white(5,6)\black(6,6)\white(7,6)\nofil(8,6)\white(9,6)
    \black(1,7)\black(2,7)\black(3,7)\black(4,7)\black(5,7)\nofil(6,7)\black(7,7)\white(8,7)\white(9,7)
    \white(1,8)\white(2,8)\white(3,8)\white(4,8)\white(5,8)\black(6,8)\white(7,8)\white(8,8)\white(9,8)
    \carry(7,1)
    \carry(8,2)
    \carry(6,3)
    \carry(7,4)
    \cell(7,2)\winmove
  \end{hexboard}
  \]
  \end{minipage}

  Hammock (and wheel):
  \[
  \begin{hexboard}[baseline={($(current bounding box.south)-(0,1ex)$)},scale=\scale]
    \rotation{-30}
    \board(4,4)
    \white(1,1)\black(2,1)\white(3,1)\white(4,1)
    \white(1,2)\carry(2,2)\black(3,2)\carry(4,2)
    \black(1,3)\carry(2,3)\carry(3,3)\carry(4,3)
    \carry(1,4)\black(2,4)\carry(3,4)\black(4,4)
    \carry(1,3)
    \carry(2,4)
    \carry(3,2)
    \carry(4,4)
    \cell(2,2)\winmove
  \end{hexboard}
  \offset\hspace{-4\increment}
  \begin{hexboard}[baseline={($(current bounding box.south)-(0,1ex)$)},scale=\scale]
    \rotation{-30}
    \board(5,4)
    \white(1,1)\white(2,1)\white(3,1)\white(4,1)\black(5,1)
    \white(1,2)\carry(2,2)\black(3,2)\carry(4,2)\white(5,2)
    \black(1,3)\carry(2,3)\carry(3,3)\carry(4,3)\white(5,3)
    \carry(1,4)\black(2,4)\carry(3,4)\black(4,4)\white(5,4)
    \carry(1,3)
    \carry(2,4)
    \carry(3,2)
    \carry(4,4)
    \cell(4,2)\winmove
  \end{hexboard}
  \offset\hspace{-7\increment}
  \begin{hexboard}[baseline={($(current bounding box.south)-(0,1ex)$)},scale=\scale]
    \rotation{-30}
    \board(7,7)
    \white(1,1)\white(2,1)\white(3,1)\white(4,1)\white(5,1)\black(6,1)\white(7,1)
    \white(1,2)\white(2,2)\carry(3,2)\black(4,2)\carry(5,2)\nofil(6,2)\black(7,2)
    \white(1,3)\carry(2,3)\carry(3,3)\carry(4,3)\black(5,3)\white(6,3)\black(7,3)
    \black(1,4)\carry(2,4)\black(3,4)\carry(4,4)\white(5,4)\nofil(6,4)\black(7,4)
    \black(1,5)\white(2,5)\white(3,5)\black(4,5)\nofil(5,5)\white(6,5)\black(7,5)
    \black(1,6)\white(2,6)\white(3,6)\white(4,6)\nofil(5,6)\nofil(6,6)\white(7,6)
    \black(1,7)\white(2,7)\white(3,7)\black(4,7)\white(5,7)\white(6,7)\white(7,7)
    \carry(1,4)
    \carry(3,4)
    \carry(4,2)
    \carry(5,3)
    \cell(4,3)\winmove
  \end{hexboard}
\offset\hspace{-6\increment}
  \begin{hexboard}[baseline={($(current bounding box.south)-(0,1ex)$)},scale=\scale]
    \rotation{-30}
    \board(4,6)
    \white(1,1)\white(2,1)\white(3,1)\black(4,1)
    \white(1,2)\black(2,2)\carry(3,2)\black(4,2)
    \carry(1,3)\carry(2,3)\black(3,3)\white(4,3)
    \black(1,4)\carry(2,4)\carry(3,4)\white(4,4)
    \carry(1,5)\carry(2,5)\white(3,5)\white(4,5)
    \black(1,6)\white(2,6)\white(3,6)\white(4,6)
    \carry(1,4)
    \carry(1,6)
    \carry(2,2)
    \carry(3,3)
    \cell(2,4)\winmove
  \end{hexboard}
  \offset\hspace{-6\increment}
  \begin{hexboard}[baseline={($(current bounding box.south)-(0,1ex)$)},scale=\scale]
    \rotation{-30}
    \board(4,6)
    \black(1,1)\white(2,1)\black(3,1)\carry(4,1)
    \black(1,2)\carry(2,2)\carry(3,2)\black(4,2)
    \black(1,3)\black(2,3)\carry(3,3)\carry(4,3)
    \white(1,4)\carry(2,4)\carry(3,4)\nofil(4,4)
    \white(1,5)\black(2,5)\nofil(3,5)\black(4,5)
    \white(1,6)\white(2,6)\white(3,6)\black(4,6)
    \carry(2,3)
    \carry(2,5)
    \carry(3,1)
    \carry(4,2)
    \cell(3,4)\winmove
  \end{hexboard}
  \offset\hspace{-4\increment}
  \begin{hexboard}[baseline={($(current bounding box.south)-(0,1ex)$)},scale=\scale]
    \rotation{-30}
    \board(5,4)
    \white(1,1)\white(2,1)\carry(3,1)\black(4,1)\carry(5,1)
    \white(1,2)\black(2,2)\carry(3,2)\carry(4,2)\carry(5,2)
    \white(1,3)\carry(2,3)\black(3,3)\carry(4,3)\black(5,3)
    \black(1,4)\white(2,4)\white(3,4)\white(4,4)\white(5,4)
    \carry(2,2)
    \carry(3,3)
    \carry(4,1)
    \carry(5,3)
    \cell(2,3)\winmove
  \end{hexboard}
  \offset\hspace{-6\increment}
  \begin{hexboard}[baseline={($(current bounding box.south)-(0,1ex)$)},scale=\scale]
    \rotation{-30}
    \board(4,6)
    \black(1,1)\white(2,1)\white(3,1)\white(4,1)
    \black(1,2)\white(2,2)\black(3,2)\carry(4,2)
    \black(1,3)\carry(2,3)\carry(3,3)\black(4,3)
    \black(1,4)\black(2,4)\carry(3,4)\carry(4,4)
    \white(1,5)\carry(2,5)\carry(3,5)\white(4,5)
    \white(1,6)\black(2,6)\white(3,6)\white(4,6)
    \carry(2,4)
    \carry(2,6)
    \carry(3,2)
    \carry(4,3)
    \cell(2,5)\winmove
  \end{hexboard}
  \]
  \caption{Witnesses for the non-inferiority of probes in various
    templates, part 1. In each position, the white dot denotes the unique
    winning move for White.}
  \label{fig:witnesses-part1}\label{fig:witnesses-ziggurat}
\end{figure}

\begin{figure}
  \def\scale{0.5}
  \setlength{\increment}{0.175cm}
  \newcommand{\offset}{\hspace{1.7em}}
  Shopping cart:
  \[
  \begin{hexboard}[baseline={($(current bounding box.south)-(0,1ex)$)},scale=\scale]
    \rotation{-30}
    \board(4,5)
    \white(1,1)\carry(2,1)\black(3,1)\white(4,1)
    \black(1,2)\carry(2,2)\carry(3,2)\black(4,2)
    \carry(1,3)\carry(2,3)\carry(3,3)\black(4,3)
    \black(1,4)\black(2,4)\white(3,4)\black(4,4)
    \white(1,5)\white(2,5)\white(3,5)\black(4,5)
    \carry(3,1)
    \carry(1,2)
    \carry(1,4)
    \carry(2,4)
    \cell(3,2)\winmove
  \end{hexboard}
  \offset\hspace{-4\increment}
  \begin{hexboard}[baseline={($(current bounding box.south)-(0,1ex)$)},scale=\scale]
    \rotation{-30}
    \board(5,4)
    \white(1,1)\white(2,1)\white(3,1)\white(4,1)\black(5,1)
    \white(1,2)\black(2,2)\carry(3,2)\carry(4,2)\white(5,2)
    \carry(1,3)\carry(2,3)\carry(3,3)\black(4,3)\white(5,3)
    \black(1,4)\carry(2,4)\black(3,4)\white(4,4)\white(5,4)
    \carry(2,2)
    \carry(1,4)
    \carry(4,3)
    \carry(3,4)
    \cell(4,2)\winmove
  \end{hexboard}
  \offset\hspace{-4\increment}
  \begin{hexboard}[baseline={($(current bounding box.south)-(0,1ex)$)},scale=\scale]
    \rotation{-30}
    \board(5,4)
    \black(1,1)\white(2,1)\white(3,1)\white(4,1)\white(5,1)
    \black(1,2)\white(2,2)\black(3,2)\carry(4,2)\carry(5,2)
    \black(1,3)\carry(2,3)\carry(3,3)\carry(4,3)\black(5,3)
    \white(1,4)\black(2,4)\carry(3,4)\black(4,4)\white(5,4)
    \carry(3,2)
    \carry(2,4)
    \carry(4,4)
    \carry(5,3)
    \cell(2,3)\winmove
  \end{hexboard}
  \offset\hspace{-6\increment}
  \begin{hexboard}[baseline={($(current bounding box.south)-(0,1ex)$)},scale=\scale]
    \rotation{-30}
    \board(8,6)
    \black(1,1)\white(2,1)\white(3,1)\white(4,1)\white(5,1)\white(6,1)\white(7,1)\black(8,1)
    \black(1,2)\white(2,2)\nofil(3,2)\black(4,2)\nofil(5,2)\white(6,2)\white(7,2)\nofil(8,2)
    \black(1,3)\nofil(2,3)\nofil(3,3)\nofil(4,3)\black(5,3)\carry(6,3)\carry(7,3)\nofil(8,3)
    \white(1,4)\white(2,4)\black(3,4)\carry(4,4)\carry(5,4)\carry(6,4)\black(7,4)\white(8,4)
    \white(1,5)\white(2,5)\black(3,5)\black(4,5)\carry(5,5)\black(6,5)\white(7,5)\white(8,5)
    \white(1,6)\white(2,6)\white(3,6)\black(4,6)\white(5,6)\white(6,6)\white(7,6)\white(8,6)
    \carry(5,3)
    \carry(4,5)
    \carry(6,5)
    \carry(7,4)
    \cell(5,4)\winmove
  \end{hexboard}
  \offset\hspace{-5\increment}
  \begin{hexboard}[baseline={($(current bounding box.south)-(0,1ex)$)},scale=\scale]
    \rotation{-30}
    \board(6,5)
    \white(1,1)\white(2,1)\black(3,1)\white(4,1)\black(5,1)\nofil(6,1)
    \white(1,2)\black(2,2)\nofil(3,2)\white(4,2)\carry(5,2)\black(6,2)
    \black(1,3)\white(2,3)\nofil(3,3)\black(4,3)\carry(5,3)\carry(6,3)
    \nofil(1,4)\nofil(2,4)\white(3,4)\carry(4,4)\carry(5,4)\carry(6,4)
    \white(1,5)\nofil(2,5)\white(3,5)\black(4,5)\black(5,5)\white(6,5)
    \carry(4,3)
    \carry(6,2)
    \carry(4,5)
    \carry(5,5)
    \cell(5,4)\winmove
  \end{hexboard}
  \offset\hspace{-7\increment}
  \begin{hexboard}[baseline={($(current bounding box.south)-(0,1ex)$)},scale=\scale]
    \rotation{-30}
    \board(6,7)
    \white(1,1)\white(2,1)\white(3,1)\black(4,1)\white(5,1)\white(6,1)
    \white(1,2)\black(2,2)\nofil(3,2)\carry(4,2)\black(5,2)\white(6,2)
    \black(1,3)\white(2,3)\black(3,3)\carry(4,3)\carry(5,3)\white(6,3)
    \black(1,4)\white(2,4)\carry(3,4)\carry(4,4)\carry(5,4)\black(6,4)
    \black(1,5)\white(2,5)\black(3,5)\black(4,5)\white(5,5)\black(6,5)
    \black(1,6)\nofil(2,6)\white(3,6)\white(4,6)\white(5,6)\black(6,6)
    \white(1,7)\black(2,7)\white(3,7)\white(4,7)\white(5,7)\black(6,7)
    \carry(3,3)
    \carry(5,2)
    \carry(3,5)
    \carry(4,5)
    \cell(3,4)\winmove
  \end{hexboard}
  \]
  Scooter:
  \[
  \begin{hexboard}[baseline={($(current bounding box.south)-(0,1ex)$)},scale=\scale]
    \rotation{-30}
    \board(5,5)
    \white(1,1)\white(2,1)\white(3,1)\white(4,1)\black(5,1)
    \white(1,2)\white(2,2)\black(3,2)\carry(4,2)\white(5,2)
    \white(1,3)\black(2,3)\carry(3,3)\carry(4,3)\white(5,3)
    \carry(1,4)\carry(2,4)\carry(3,4)\carry(4,4)\white(5,4)
    \black(1,5)\carry(2,5)\black(3,5)\white(4,5)\white(5,5)
    \carry(3,2)
    \carry(2,3)
    \carry(1,5)
    \carry(3,5)
    \cell(4,2)\winmove
  \end{hexboard}
  \offset\hspace{-11\increment}
  \begin{hexboard}[baseline={($(current bounding box.south)-(0,1ex)$)},scale=\scale]
    \rotation{-30}
    \board(12,11)
    \white(1,1)\white(2,1)\black(3,1)\white(4,1)\white(5,1)\white(6,1)\white(7,1)\black(8,1)\white(9,1)\white(10,1)\white(11,1)\white(12,1)
    \white(1,2)\white(2,2)\black(3,2)\white(4,2)\white(5,2)\white(6,2)\white(7,2)\nofil(8,2)\white(9,2)\black(10,2)\black(11,2)\black(12,2)
    \white(1,3)\white(2,3)\black(3,3)\nofil(4,3)\white(5,3)\white(6,3)\white(7,3)\black(8,3)\carry(9,3)\white(10,3)\white(11,3)\black(12,3)
    \white(1,4)\white(2,4)\nofil(3,4)\black(4,4)\nofil(5,4)\white(6,4)\black(7,4)\carry(8,4)\carry(9,4)\black(10,4)\nofil(11,4)\white(12,4)
    \white(1,5)\black(2,5)\white(3,5)\nofil(4,5)\black(5,5)\carry(6,5)\carry(7,5)\carry(8,5)\carry(9,5)\nofil(10,5)\nofil(11,5)\white(12,5)
    \black(1,6)\white(2,6)\black(3,6)\nofil(4,6)\white(5,6)\black(6,6)\carry(7,6)\black(8,6)\white(9,6)\white(10,6)\black(11,6)\black(12,6)
    \black(1,7)\white(2,7)\black(3,7)\black(4,7)\nofil(5,7)\white(6,7)\white(7,7)\black(8,7)\black(9,7)\nofil(10,7)\white(11,7)\black(12,7)
    \black(1,8)\white(2,8)\white(3,8)\white(4,8)\white(5,8)\white(6,8)\white(7,8)\black(8,8)\white(9,8)\black(10,8)\white(11,8)\black(12,8)
    \black(1,9)\black(2,9)\black(3,9)\black(4,9)\black(5,9)\black(6,9)\white(7,9)\nofil(8,9)\black(9,9)\nofil(10,9)\nofil(11,9)\white(12,9)
    \white(1,10)\white(2,10)\white(3,10)\white(4,10)\white(5,10)\black(6,10)\black(7,10)\nofil(8,10)\white(9,10)\white(10,10)\black(11,10)\white(12,10)
    \white(1,11)\white(2,11)\white(3,11)\white(4,11)\white(5,11)\white(6,11)\white(7,11)\black(8,11)\white(9,11)\white(10,11)\black(11,11)\white(12,11)
    \carry(8,3)
    \carry(7,4)
    \carry(6,6)
    \carry(8,6)
    \cell(8,4)\winmove
  \end{hexboard}
  \offset\hspace{-5\increment}
  \begin{hexboard}[baseline={($(current bounding box.south)-(0,1ex)$)},scale=\scale]
    \rotation{-30}
    \board(5,5)
    \white(1,1)\white(2,1)\black(3,1)\carry(4,1)\white(5,1)
    \white(1,2)\black(2,2)\carry(3,2)\carry(4,2)\black(5,2)
    \carry(1,3)\carry(2,3)\carry(3,3)\carry(4,3)\black(5,3)
    \black(1,4)\carry(2,4)\black(3,4)\white(4,4)\black(5,4)
    \white(1,5)\white(2,5)\white(3,5)\white(4,5)\black(5,5)
    \carry(3,1)
    \carry(2,2)
    \carry(1,4)
    \carry(3,4)
    \cell(4,2)\winmove
  \end{hexboard}
  \offset\hspace{-5\increment}
  \begin{hexboard}[baseline={($(current bounding box.south)-(0,1ex)$)},scale=\scale]
    \rotation{-30}
    \board(5,5)
    \black(1,1)\white(2,1)\white(3,1)\white(4,1)\white(5,1)
    \black(1,2)\white(2,2)\white(3,2)\black(4,2)\carry(5,2)
    \black(1,3)\white(2,3)\black(3,3)\carry(4,3)\carry(5,3)
    \black(1,4)\carry(2,4)\carry(3,4)\carry(4,4)\carry(5,4)
    \white(1,5)\black(2,5)\carry(3,5)\black(4,5)\white(5,5)
    \carry(4,2)
    \carry(3,3)
    \carry(2,5)
    \carry(4,5)
    \cell(2,4)\winmove
  \end{hexboard}
  \]
  \[
  \begin{hexboard}[baseline={($(current bounding box.south)-(0,1ex)$)},scale=\scale]
    \rotation{-30}
    \board(8,6)
    \white(1,1)\white(2,1)\white(3,1)\black(4,1)\white(5,1)\carry(6,1)\black(7,1)\black(8,1)
    \white(1,2)\black(2,2)\nofil(3,2)\nofil(4,2)\black(5,2)\carry(6,2)\carry(7,2)\carry(8,2)
    \black(1,3)\white(2,3)\black(3,3)\nofil(4,3)\carry(5,3)\carry(6,3)\carry(7,3)\white(8,3)
    \black(1,4)\white(2,4)\nofil(3,4)\white(4,4)\black(5,4)\carry(6,4)\white(7,4)\white(8,4)
    \black(1,5)\black(2,5)\white(3,5)\white(4,5)\nofil(5,5)\nofil(6,5)\white(7,5)\white(8,5)
    \white(1,6)\nofil(2,6)\white(3,6)\white(4,6)\white(5,6)\black(6,6)\white(7,6)\white(8,6)
    \carry(8,1)
    \carry(7,1)
    \carry(5,2)
    \carry(5,4)
    \cell(6,2)\winmove
  \end{hexboard}
  \offset\hspace{-6\increment}
  \begin{hexboard}[baseline={($(current bounding box.south)-(0,1ex)$)},scale=\scale]
    \rotation{-30}
    \board(7,6)
    \white(1,1)\white(2,1)\white(3,1)\white(4,1)\white(5,1)\white(6,1)\black(7,1)
    \white(1,2)\white(2,2)\white(3,2)\black(4,2)\carry(5,2)\nofil(6,2)\nofil(7,2)
    \white(1,3)\black(2,3)\black(3,3)\carry(4,3)\carry(5,3)\white(6,3)\black(7,3)
    \black(1,4)\carry(2,4)\carry(3,4)\carry(4,4)\carry(5,4)\black(6,4)\white(7,4)
    \black(1,5)\black(2,5)\carry(3,5)\black(4,5)\nofil(5,5)\white(6,5)\white(7,5)
    \black(1,6)\white(2,6)\white(3,6)\white(4,6)\white(5,6)\white(6,6)\white(7,6)
    \carry(4,2)
    \carry(3,3)
    \carry(2,5)
    \carry(4,5)
    \cell(4,4)\winmove
  \end{hexboard}
  \offset\hspace{-5\increment}
  \begin{hexboard}[baseline={($(current bounding box.south)-(0,1ex)$)},scale=\scale]
    \rotation{-30}
    \board(5,5)
    \white(1,1)\black(2,1)\carry(3,1)\black(4,1)\black(5,1)
    \white(1,2)\black(2,2)\carry(3,2)\carry(4,2)\carry(5,2)
    \black(1,3)\carry(2,3)\carry(3,3)\carry(4,3)\white(5,3)
    \black(1,4)\black(2,4)\carry(3,4)\white(4,4)\white(5,4)
    \white(1,5)\white(2,5)\black(3,5)\white(4,5)\white(5,5)
    \carry(5,1)
    \carry(4,1)
    \carry(2,2)
    \carry(2,4)
    \cell(3,4)\winmove
  \end{hexboard}
  \offset\hspace{-7\increment}
  \begin{hexboard}[baseline={($(current bounding box.south)-(0,1ex)$)},scale=\scale]
    \rotation{-30}
    \board(7,7)
    \black(1,1)\white(2,1)\white(3,1)\white(4,1)\white(5,1)\white(6,1)\black(7,1)
    \black(1,2)\white(2,2)\black(3,2)\black(4,2)\carry(5,2)\white(6,2)\black(7,2)
    \black(1,3)\carry(2,3)\carry(3,3)\carry(4,3)\black(5,3)\nofil(6,3)\white(7,3)
    \white(1,4)\carry(2,4)\carry(3,4)\carry(4,4)\white(5,4)\black(6,4)\white(7,4)
    \white(1,5)\carry(2,5)\black(3,5)\white(4,5)\black(5,5)\white(6,5)\white(7,5)
    \black(1,6)\white(2,6)\nofil(3,6)\black(4,6)\white(5,6)\white(6,6)\white(7,6)
    \black(1,7)\black(2,7)\white(3,7)\white(4,7)\white(5,7)\white(6,7)\white(7,7)
    \carry(3,2)
    \carry(4,2)
    \carry(5,3)
    \carry(3,5)
    \cell(4,4)\winmove
  \end{hexboard}
  \]
  Bicycle:
  \[
  \begin{hexboard}[baseline={($(current bounding box.south)-(0,1ex)$)},scale=\scale]
    \rotation{-30}
    \board(5,5)
    \white(1,1)\white(2,1)\black(3,1)\white(4,1)\white(5,1)
    \white(1,2)\white(2,2)\carry(3,2)\black(4,2)\black(5,2)
    \white(1,3)\black(2,3)\carry(3,3)\carry(4,3)\black(5,3)
    \carry(1,4)\carry(2,4)\carry(3,4)\carry(4,4)\black(5,4)
    \black(1,5)\carry(2,5)\black(3,5)\white(4,5)\black(5,5)
    \carry(2,3)
    \carry(1,5)
    \carry(4,2)
    \carry(3,5)
    \cell(3,2)\winmove
  \end{hexboard}
  \offset\hspace{-7\increment}
  \begin{hexboard}[baseline={($(current bounding box.south)-(0,1ex)$)},scale=\scale]
    \rotation{-30}
    \board(8,7)
    \white(1,1)\white(2,1)\black(3,1)\white(4,1)\white(5,1)\white(6,1)\white(7,1)\black(8,1)
    \white(1,2)\black(2,2)\white(3,2)\black(4,2)\black(5,2)\black(6,2)\white(7,2)\black(8,2)
    \white(1,3)\nofil(2,3)\nofil(3,3)\white(4,3)\white(5,3)\carry(6,3)\carry(7,3)\black(8,3)
    \nofil(1,4)\white(2,4)\nofil(3,4)\white(4,4)\black(5,4)\carry(6,4)\carry(7,4)\carry(8,4)
    \black(1,5)\black(2,5)\black(3,5)\white(4,5)\carry(5,5)\carry(6,5)\black(7,5)\white(8,5)
    \white(1,6)\white(2,6)\nofil(3,6)\black(4,6)\black(5,6)\carry(6,6)\white(7,6)\white(8,6)
    \white(1,7)\black(2,7)\white(3,7)\white(4,7)\white(5,7)\black(6,7)\white(7,7)\white(8,7)
    \carry(5,4)
    \carry(5,6)
    \carry(7,5)
    \carry(8,3)
    \cell(7,4)\winmove
  \end{hexboard}
  \offset\hspace{-5\increment}  
  \begin{hexboard}[baseline={($(current bounding box.south)-(0,1ex)$)},scale=\scale]
    \rotation{-30}
    \board(5,5)
    \white(1,1)\white(2,1)\carry(3,1)\black(4,1)\white(5,1)
    \white(1,2)\black(2,2)\carry(3,2)\carry(4,2)\black(5,2)
    \carry(1,3)\carry(2,3)\carry(3,3)\carry(4,3)\black(5,3)
    \black(1,4)\carry(2,4)\black(3,4)\white(4,4)\black(5,4)
    \white(1,5)\white(2,5)\white(3,5)\white(4,5)\black(5,5)
    \carry(2,2)
    \carry(1,4)
    \carry(4,1)
    \carry(3,4)
    \cell(4,2)\winmove
  \end{hexboard}
  \offset\hspace{-5\increment}
  \begin{hexboard}[baseline={($(current bounding box.south)-(0,1ex)$)},scale=\scale]
    \rotation{-30}
    \board(5,5)
    \white(1,1)\white(2,1)\black(3,1)\white(4,1)\white(5,1)
    \white(1,2)\white(2,2)\carry(3,2)\black(4,2)\carry(5,2)
    \white(1,3)\black(2,3)\carry(3,3)\carry(4,3)\black(5,3)
    \black(1,4)\carry(2,4)\carry(3,4)\carry(4,4)\white(5,4)
    \black(1,5)\black(2,5)\carry(3,5)\white(4,5)\white(5,5)
    \carry(2,3)
    \carry(2,5)
    \carry(4,2)
    \carry(5,3)
    \cell(3,2)\winmove
  \end{hexboard}
  \]
  \[
  \begin{hexboard}[baseline={($(current bounding box.south)-(0,1ex)$)},scale=\scale]
    \rotation{-30}
    \board(8,6)
    \white(1,1)\white(2,1)\white(3,1)\black(4,1)\white(5,1)\carry(6,1)\black(7,1)\carry(8,1)
    \white(1,2)\black(2,2)\nofil(3,2)\nofil(4,2)\black(5,2)\carry(6,2)\carry(7,2)\black(8,2)
    \black(1,3)\white(2,3)\black(3,3)\nofil(4,3)\carry(5,3)\carry(6,3)\carry(7,3)\white(8,3)
    \black(1,4)\white(2,4)\nofil(3,4)\white(4,4)\black(5,4)\carry(6,4)\white(7,4)\white(8,4)
    \black(1,5)\black(2,5)\white(3,5)\white(4,5)\nofil(5,5)\nofil(6,5)\white(7,5)\white(8,5)
    \white(1,6)\nofil(2,6)\white(3,6)\white(4,6)\white(5,6)\black(6,6)\white(7,6)\white(8,6)
    \carry(5,2)
    \carry(7,1)
    \carry(8,2)
    \carry(5,4)
    \cell(6,2)\winmove
  \end{hexboard}
  \offset\hspace{-5\increment}
  \begin{hexboard}[baseline={($(current bounding box.south)-(0,1ex)$)},scale=\scale]
    \rotation{-30}
    \board(7,5)
    \black(1,1)\white(2,1)\white(3,1)\white(4,1)\black(5,1)\black(6,1)\black(7,1)
    \nofil(1,2)\carry(2,2)\black(3,2)\carry(4,2)\white(5,2)\nofil(6,2)\black(7,2)
    \black(1,3)\carry(2,3)\carry(3,3)\black(4,3)\nofil(5,3)\white(6,3)\black(7,3)
    \carry(1,4)\carry(2,4)\carry(3,4)\white(4,4)\nofil(5,4)\nofil(6,4)\white(7,4)
    \carry(1,5)\black(2,5)\white(3,5)\nofil(4,5)\white(5,5)\white(6,5)\white(7,5)
    \carry(1,3)
    \carry(3,2)
    \carry(2,5)
    \carry(4,3)
    \cell(2,4)\winmove
  \end{hexboard}
  \offset\hspace{-5\increment}
  \begin{hexboard}[baseline={($(current bounding box.south)-(0,1ex)$)},scale=\scale]
    \rotation{-30}
    \board(5,5)
    \black(1,1)\white(2,1)\carry(3,1)\black(4,1)\carry(5,1)
    \black(1,2)\black(2,2)\carry(3,2)\carry(4,2)\black(5,2)
    \black(1,3)\carry(2,3)\carry(3,3)\carry(4,3)\white(5,3)
    \black(1,4)\black(2,4)\carry(3,4)\white(4,4)\white(5,4)
    \white(1,5)\white(2,5)\black(3,5)\white(4,5)\white(5,5)
    \carry(2,2)
    \carry(2,4)
    \carry(4,1)
    \carry(5,2)
    \cell(3,4)\winmove
  \end{hexboard}
  \offset\hspace{-7\increment}
  \begin{hexboard}[baseline={($(current bounding box.south)-(0,1ex)$)},scale=\scale]
    \rotation{-30}
    \board(7,7)
    \white(1,1)\white(2,1)\black(3,1)\white(4,1)\white(5,1)\white(6,1)\black(7,1)
    \white(1,2)\black(2,2)\white(3,2)\carry(4,2)\black(5,2)\carry(6,2)\white(7,2)
    \white(1,3)\nofil(2,3)\black(3,3)\carry(4,3)\carry(5,3)\black(6,3)\white(7,3)
    \black(1,4)\white(2,4)\carry(3,4)\carry(4,4)\carry(5,4)\white(6,4)\white(7,4)
    \black(1,5)\white(2,5)\black(3,5)\carry(4,5)\white(5,5)\white(6,5)\white(7,5)
    \black(1,6)\nofil(2,6)\white(3,6)\black(4,6)\white(5,6)\white(6,6)\white(7,6)
    \white(1,7)\black(2,7)\black(3,7)\white(4,7)\white(5,7)\white(6,7)\white(7,7)
    \carry(5,2)
    \carry(3,3)
    \carry(3,5)
    \carry(6,3)
    \cell(3,4)\winmove
  \end{hexboard}
  \]
  Wide parallelogram:
  \[
  \begin{hexboard}[baseline={($(current bounding box.south)-(0,1ex)$)},scale=\scale]
    \rotation{-30}
    \board(4,5)
    \white(1,1)\black(2,1)\white(3,1)\white(4,1)
    \white(1,2)\carry(2,2)\carry(3,2)\carry(4,2)
    \black(1,3)\carry(2,3)\carry(3,3)\black(4,3)
    \black(1,4)\carry(2,4)\carry(3,4)\black(4,4)
    \carry(1,5)\carry(2,5)\carry(3,5)\black(4,5)
    \carry(1,3)
    \carry(1,4)
    \carry(4,3)
    \carry(4,4)
    \cell(2,2)\winmove
  \end{hexboard}
  \offset\hspace{-5\increment}
  \begin{hexboard}[baseline={($(current bounding box.south)-(0,1ex)$)},scale=\scale]
    \rotation{-30}
    \board(4,5)
    \white(1,1)\white(2,1)\nofil(3,1)\nofil(4,1)
    \white(1,2)\carry(2,2)\carry(3,2)\carry(4,2)
    \black(1,3)\carry(2,3)\carry(3,3)\black(4,3)
    \black(1,4)\carry(2,4)\carry(3,4)\black(4,4)
    \carry(1,5)\carry(2,5)\carry(3,5)\black(4,5)
    \carry(1,3)
    \carry(1,4)
    \carry(4,3)
    \carry(4,4)
    \cell(3,2)\winmove
  \end{hexboard}
  \offset\hspace{-5\increment}
  \begin{hexboard}[baseline={($(current bounding box.south)-(0,1ex)$)},scale=\scale]
    \rotation{-30}
    \board(5,5)
    \white(1,1)\white(2,1)\white(3,1)\white(4,1)\black(5,1)
    \white(1,2)\carry(2,2)\carry(3,2)\carry(4,2)\white(5,2)
    \black(1,3)\carry(2,3)\carry(3,3)\black(4,3)\white(5,3)
    \black(1,4)\carry(2,4)\carry(3,4)\black(4,4)\white(5,4)
    \carry(1,5)\carry(2,5)\carry(3,5)\black(4,5)\white(5,5)
    \carry(1,3)
    \carry(1,4)
    \carry(4,3)
    \carry(4,4)
    \cell(4,2)\winmove
  \end{hexboard}
  \offset\hspace{-9\increment}
  \begin{hexboard}[baseline={($(current bounding box.south)-(0,1ex)$)},scale=\scale]
    \rotation{-30}
    \board(6,9)
    \white(1,1)\white(2,1)\black(3,1)\white(4,1)\white(5,1)\black(6,1)
    \white(1,2)\white(2,2)\nofil(3,2)\black(4,2)\white(5,2)\nofil(6,2)
    \white(1,3)\black(2,3)\white(3,3)\nofil(4,3)\black(5,3)\black(6,3)
    \black(1,4)\white(2,4)\nofil(3,4)\white(4,4)\white(5,4)\black(6,4)
    \black(1,5)\white(2,5)\black(3,5)\black(4,5)\carry(5,5)\black(6,5)
    \black(1,6)\carry(2,6)\carry(3,6)\carry(4,6)\carry(5,6)\white(6,6)
    \white(1,7)\carry(2,7)\carry(3,7)\carry(4,7)\carry(5,7)\black(6,7)
    \black(1,8)\carry(2,8)\black(3,8)\black(4,8)\white(5,8)\black(6,8)
    \black(1,9)\white(2,9)\white(3,9)\white(4,9)\white(5,9)\black(6,9)
    \carry(3,5)
    \carry(4,5)
    \carry(3,8)
    \carry(4,8)
    \cell(4,7)\winmove
  \end{hexboard}
  \offset\hspace{-8\increment}
  \begin{hexboard}[baseline={($(current bounding box.south)-(0,1ex)$)},scale=\scale]
    \rotation{-30}
    \board(7,8)
    \black(1,1)\white(2,1)\white(3,1)\white(4,1)\white(5,1)\black(6,1)\white(7,1)
    \black(1,2)\white(2,2)\white(3,2)\white(4,2)\white(5,2)\nofil(6,2)\black(7,2)
    \black(1,3)\white(2,3)\white(3,3)\white(4,3)\black(5,3)\white(6,3)\black(7,3)
    \black(1,4)\white(2,4)\black(3,4)\black(4,4)\carry(5,4)\white(6,4)\black(7,4)
    \black(1,5)\carry(2,5)\carry(3,5)\carry(4,5)\carry(5,5)\white(6,5)\black(7,5)
    \white(1,6)\carry(2,6)\carry(3,6)\carry(4,6)\carry(5,6)\white(6,6)\black(7,6)
    \white(1,7)\carry(2,7)\black(3,7)\black(4,7)\black(5,7)\nofil(6,7)\white(7,7)
    \black(1,8)\white(2,8)\white(3,8)\white(4,8)\white(5,8)\black(6,8)\white(7,8)
    \carry(3,4)
    \carry(4,4)
    \carry(3,7)
    \carry(4,7)
    \cell(4,5)\winmove
  \end{hexboard}
  \]
  \caption{Witnesses for the non-inferiority of probes in various
    templates, part 2. In each position, the white dot denotes the unique
    winning move for White.}
  \label{fig:witnesses-part2}
\end{figure}

\subsubsection{Probes in long crescents}

Some templates form infinite families. An example of this is the
\emph{long crescent}, which is a generalization of the crescent that
can be of any length:
\[
  \begin{array}[b]{c}
    \begin{hexboard}[baseline={($(current bounding box.south)-(0,1ex)$)},scale=0.8]
      \rotation{30}
      \foreach\i in {2,...,3} {\hex(\i,1)}
      \foreach\i in {1,...,3} {\hex(\i,2)}
      \foreach\i in {1,...,3} {\hex(\i,3)}
      \black(2,1)
      \black(1,2)
      \black(1,3)
      \black(3,3)
      \cell(3,1)\sflabel{1}
      \cell(2,2)\sflabel{2}
      \cell(3,2)\sflabel{3}
      \cell(2,3)\sflabel{4}
    \end{hexboard}
    \\
    \mbox{The crescent.}
  \end{array}
  \begin{array}[b]{c}
    \begin{hexboard}[baseline={($(current bounding box.south)-(0,1ex)$)},scale=0.8]
      \rotation{30}
      \foreach\i in {2,...,3} {\hex(\i,1)}
      \foreach\i in {1,...,3} {\hex(\i,2)}
      \foreach\i in {1,...,3} {\hex(\i,3)}
      \foreach\i in {1,...,3} {\hex(\i,4)}
      \black(2,1)
      \black(1,2)
      \black(1,3)
      \black(1,4)
      \black(3,4)
      \cell(3,1)\sflabel{1}
      \cell(2,2)\sflabel{2}
      \cell(3,2)\sflabel{3}
      \cell(2,3)\sflabel{4}
      \cell(3,3)\sflabel{5}
      \cell(2,4)\sflabel{6}
    \end{hexboard}
    \\
    \mbox{The 6-cell long crescent.}
  \end{array}
  \begin{array}[b]{c}
    \begin{hexboard}[baseline={($(current bounding box.south)-(0,1ex)$)},scale=0.8]
      \rotation{30}
      \foreach\i in {2,...,3} {\hex(\i,1)}
      \foreach\i in {1,...,3} {\hex(\i,2)}
      \foreach\i in {1,...,3} {\hex(\i,3)}
      \foreach\i in {1,...,3} {\hex(\i,4)}
      \foreach\i in {1,...,3} {\hex(\i,5)}
      \black(2,1)
      \black(1,2)
      \black(1,3)
      \black(1,4)
      \black(1,5)
      \black(3,5)
      \cell(3,1)\sflabel{1}
      \cell(2,2)\sflabel{2}
      \cell(3,2)\sflabel{3}
      \cell(2,3)\sflabel{4}
      \cell(3,3)\sflabel{5}
      \cell(2,4)\sflabel{6}
      \cell(3,4)\sflabel{7}
      \cell(2,5)\sflabel{8}
    \end{hexboard}
    \\
    \mbox{The 8-cell long crescent.}
  \end{array}
  \begin{array}[b]{c}
    \begin{hexboard}[baseline={($(current bounding box.south)-(0,1ex)$)},scale=0.8]
      \rotation{30}
      \foreach\i in {2,...,3} {\hex(\i,1)}
      \foreach\i in {1,...,3} {\hex(\i,2)}
      \foreach\i in {1,...,3} {\hex(\i,3)}
      \foreach\i in {1,...,3} {\hex(\i,4)}
      \foreach\i in {1,...,3} {\hex(\i,5)}
      \foreach\i in {1,...,3} {\hex(\i,6)}
      \black(2,1)
      \black(1,2)
      \black(1,3)
      \black(1,4)
      \black(1,5)
      \black(1,6)
      \black(3,6)
      \cell(3,1)\sflabel{1}
      \cell(2,2)\sflabel{2}
      \cell(3,2)\sflabel{3}
      \cell(2,3)\sflabel{4}
      \cell(3,3)\sflabel{5}
      \cell(2,4)\sflabel{6}
      \cell(3,4)\sflabel{7}
      \cell(2,5)\sflabel{8}
      \cell(3,5)\sflabel{9}
      \cell(2,6)\sflabel{10}
    \end{hexboard}
    \\
    \mbox{The 10-cell long crescent.}
  \end{array}
\]
In the long crescent of any length, probes 2 and 4 are inferior for
the same reason as in the crescent. The remaining probes (i.e., all
odd probes, and all even probes $\geq 6$) are non-inferior; witnesses
for the 10-cell long crescent are shown in
Figure~\ref{fig:witnesses-long-crescent}. Moreover, these witnesses
can be adjusted to all possible lengths, due to the following
equivalences, which can be verified by computing their
combinatorial values.
\[
\begin{hexboard}[baseline={($(current bounding box.center)-(0,0.55ex)$)},scale=0.6]
  \rotation{30}
  \foreach\i in {2,...,4} {\hex(\i,4)}
  \foreach\i in {1,...,3} {\hex(\i,5)}
  \foreach\i in {1,...,2} {\hex(\i,6)}
  \white(4,4)
  \black(1,5)
  \black(1,6)
  \black(2,4)
\end{hexboard}
~\eq~
\begin{hexboard}[baseline={($(current bounding box.center)-(0,0.55ex)$)},scale=0.6]
  \rotation{30}
  \foreach\i in {2,...,4} {\hex(\i,3)}
  \foreach\i in {1,...,4} {\hex(\i,4)}
  \foreach\i in {1,...,3} {\hex(\i,5)}
  \foreach\i in {1,...,2} {\hex(\i,6)}
  \white(4,3)
  \white(4,4)
  \black(1,4)
  \black(1,5)
  \black(1,6)
  \black(2,3)
\end{hexboard}
\hspace{1cm}
\begin{hexboard}[baseline={($(current bounding box.center)-(0,0.55ex)$)},scale=0.6]
  \rotation{30}
  \foreach\i in {3,...,4} {\hex(\i,3)}
  \foreach\i in {2,...,4} {\hex(\i,4)}
  \foreach\i in {1,...,3} {\hex(\i,5)}
  \foreach\i in {1,...,2} {\hex(\i,6)}
  \black(1,5)
  \black(1,6)
  \white(4,3)
  \white(4,4)
\end{hexboard}
~\eq~
\begin{hexboard}[baseline={($(current bounding box.center)-(0,0.55ex)$)},scale=0.6]
  \rotation{30}
  \foreach\i in {3,...,4} {\hex(\i,2)}
  \foreach\i in {2,...,4} {\hex(\i,3)}
  \foreach\i in {1,...,4} {\hex(\i,4)}
  \foreach\i in {1,...,3} {\hex(\i,5)}
  \foreach\i in {1,...,2} {\hex(\i,6)}
  \black(1,4)
  \black(1,5)
  \black(1,6)
  \white(4,2)
  \white(4,3)
  \white(4,4)
\end{hexboard}
\hspace{1cm}
\begin{hexboard}[baseline={($(current bounding box.center)-(0,0.55ex)$)},scale=0.6]
  \rotation{30}
  \foreach\i in {3,...,4} {\hex(\i,4)}
  \foreach\i in {1,...,3} {\hex(\i,5)}
  \foreach\i in {1,...,3} {\hex(\i,6)}
  \white(4,4)
  \black(1,5)
  \black(1,6)
  \black(3,6)
\end{hexboard}
~\eq~
\begin{hexboard}[baseline={($(current bounding box.center)-(0,0.55ex)$)},scale=0.6]
  \rotation{30}
  \foreach\i in {3,...,4} {\hex(\i,3)}
  \foreach\i in {1,...,4} {\hex(\i,4)}
  \foreach\i in {1,...,3} {\hex(\i,5)}
  \foreach\i in {1,...,3} {\hex(\i,6)}
  \white(4,3)
  \white(4,4)
  \black(1,4)
  \black(1,5)
  \black(1,6)
  \black(3,6)
\end{hexboard}
\]

\subsubsection{Probes in bolstered ziggurats}

The inferiority of probes in some Hex region can be affected by nearby
stones outside of the region. For example, consider the bridge in (a)
below. Both probes in the bridge are non-inferior.  However, if we
\emph{bolster} the bridge by adding a white stone on the left, as in
(b), then probe $1$ becomes inferior, because it is then strongly
reversed by $2$. Only the white intrusion at $2$ remains useful. If
the bridge is bolstered on both sides as in (c), then there are no
useful intrusions at all.
\[
\begin{array}[b]{c}
\begin{hexboard}[baseline={($(current bounding box.south)-(0,1ex)$)},scale=0.8]
  \rotation{-30}
  \foreach\i in {2,...,2} {\hex(\i,1)}
  \foreach\i in {1,...,2} {\hex(\i,2)}
  \foreach\i in {1,...,1} {\hex(\i,3)}
  \black(2,1)
  \black(1,3)
  \cell(1,2)\sflabel{1}
  \cell(2,2)\sflabel{2}
\end{hexboard}
\\
\mbox{(a) The bridge.}
\end{array}
\begin{array}[b]{c}
\begin{hexboard}[baseline={($(current bounding box.south)-(0,1ex)$)},scale=0.8]
  \rotation{-30}
  \foreach\i in {2,...,2} {\hex(\i,1)}
  \foreach\i in {0,...,2} {\hex(\i,2)}
  \foreach\i in {1,...,1} {\hex(\i,3)}
  \black(2,1)
  \black(1,3)
  \white(0,2)
  \cell(1,2)\sflabel{1}
  \cell(2,2)\sflabel{2}
\end{hexboard}
\\
\mbox{(b) A bridge bolstered on the left.}
\end{array}
\begin{array}[b]{c}
\begin{hexboard}[baseline={($(current bounding box.south)-(0,1ex)$)},scale=0.8]
  \rotation{-30}
  \foreach\i in {2,...,2} {\hex(\i,1)}
  \foreach\i in {0,...,3} {\hex(\i,2)}
  \foreach\i in {1,...,1} {\hex(\i,3)}
  \black(2,1)
  \black(1,3)
  \white(0,2)
  \white(3,2)
  \cell(1,2)\sflabel{1}
  \cell(2,2)\sflabel{2}
\end{hexboard}
\\
\mbox{(c) A bridge bolstered on both sides.}
\end{array}
\]
Let us investigate what happens to a ziggurat in the presence of a
single white stone on the perimeter.  We label the cells adjacent to
the ziggurat with letters $a,\ldots,i$, as follows:
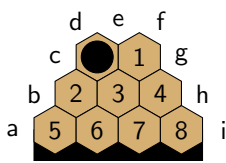
\[
\begin{hexboard}[scale=0.8]
  \template(4,3)
  \foreach \i in {3,...,4} {\hex(\i,1)}
  \foreach \i in {2,...,4} {\hex(\i,2)}
  \foreach \i in {1,...,4} {\hex(\i,3)}
  \black(3,1)
  \cell(4,1)\sflabel{1}
  \cell(2,2)\sflabel{2}
  \cell(3,2)\sflabel{3}
  \cell(4,2)\sflabel{4}
  \cell(1,3)\sflabel{5}
  \cell(2,3)\sflabel{6}
  \cell(3,3)\sflabel{7}
  \cell(4,3)\sflabel{8}
  \cell(0,3)\sflabel{a}
  \cell(1,2)\sflabel{b}
  \cell(2,1)\sflabel{c}
  \cell(3,0)\sflabel{d}
  \cell(4,0)\sflabel{e}
  \cell(5,0)\sflabel{f}
  \cell(5,1)\sflabel{g}
  \cell(5,2)\sflabel{h}
  \cell(5,3)\sflabel{i}
\end{hexboard}
\]
For example, a ziggurat with a white stone at $b$ is the following
position, which we call the \emph{half-star}.
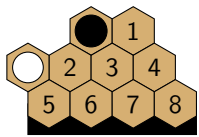
\[
\begin{hexboard}[scale=0.8]
  \template(4,3)
  \foreach \i in {3,...,4} {\hex(\i,1)}
  \foreach \i in {2,...,4} {\hex(\i,2)}
  \foreach \i in {1,...,4} {\hex(\i,3)}
  \black(3,1)
  \cell(4,1)\sflabel{1}
  \cell(2,2)\sflabel{2}
  \cell(3,2)\sflabel{3}
  \cell(4,2)\sflabel{4}
  \cell(1,3)\sflabel{5}
  \cell(2,3)\sflabel{6}
  \cell(3,3)\sflabel{7}
  \cell(4,3)\sflabel{8}
  \hex(1,2)\white(1,2)
\end{hexboard}
\]
Henderson and Hayward
showed in {\cite{star-decomposition}} that in the half-star, probe~4
dominates probes 2, 3, 5, and 7. In other words, probes 2, 3, 5, and 7
are inferior in this situation. One may ask whether any of the
remaining probes, 1, 4, 6, or 8 are inferior in the half-star. This is
not the case; in fact all are non-inferior, and our witnesses
for these probes in Figure~\ref{fig:witnesses-part1} are
half-stars.

The following table summarizes the status of each probe 1--8 in the
presence of a white stone at $a$--$i$. 
\[
\begin{array}{c|ccccccccc}
  &   a  & b  & c  & d  & e  & f  & g  & h  & i  \\\hline
  1 & O  & W  & O  & W  & W  & A  & W  & W  & W  \\
  2 & W  &\ifr& W  & O  & W  & W  & W  & W  & W  \\
  3 & O  &\ifr& W  & W  & O  &\ifr& W  &\ifr& W  \\
  4 & W  & W  & W  & W  & W  & B  & B  & W  & W  \\
  5 & W  &\ifr& W  & O  & O  & W  & W  & C  & W  \\
  6 & W  & W  & W  & W  & O  &\ifr& W  & D  & D  \\
  7 & W  &\ifr& W  & E  & W  &\ifr& W  &\ifr& W  \\
  8 & O  & W  & F  & W  & W  & F  & W  &\ifr& W  \\
\end{array}
\]
In this table, ``\,$\ifr$\,'' means that the probe is provably
inferior, and all other entries mean that the probe is non-inferior.
``$W$'' indicates that the corresponding witness in
Figure~\ref{fig:witnesses-ziggurat} works, and ``$O$'' indicates that
an obvious modification of that witness works. By ``obvious
modification'', we mean one that is isomorphic to the original witness
as a set coloring game (see Section~\ref{sec:set-coloring}). For
example, the following is an obvious modification of the probe 3
witness from Figure~\ref{fig:witnesses-ziggurat}, which works in the
presence of a white stone at $e$:
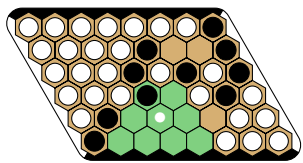
\[
\begin{hexboard}[baseline={($(current bounding box.center)-(0,1ex)$)},scale=0.5]
  \rotation{-30}
  \board(8,6)
  \white(1,1)\white(2,1)\white(3,1)\white(4,1)\white(5,1)\white(6,1)\white(7,1)\black(8,1)
  \white(1,2)\white(2,2)\white(3,2)\white(4,2)\black(5,2)\nofil(6,2)\nofil(7,2)\black(8,2)
  \white(1,3)\white(2,3)\white(3,3)\black(4,3)\white(5,3)\black(6,3)\white(7,3)\black(8,3)
  \white(1,4)\white(2,4)\white(3,4)\black(4,4)\carry(5,4)\white(6,4)\black(7,4)\white(8,4)
  \white(1,5)\black(2,5)\carry(3,5)\carry(4,5)\carry(5,5)\nofil(6,5)\white(7,5)\white(8,5)
  \black(1,6)\carry(2,6)\carry(3,6)\carry(4,6)\carry(5,6)\white(6,6)\white(7,6)\white(8,6)
  \carry(4,4)
  \cell(4,5)\winmove
\end{hexboard}
\]
Finally, the letters $A$--$F$ refer to the six additional witnesses
shown in Figure~\ref{fig:witnesses-ziggurat-extra}.

\begin{figure}
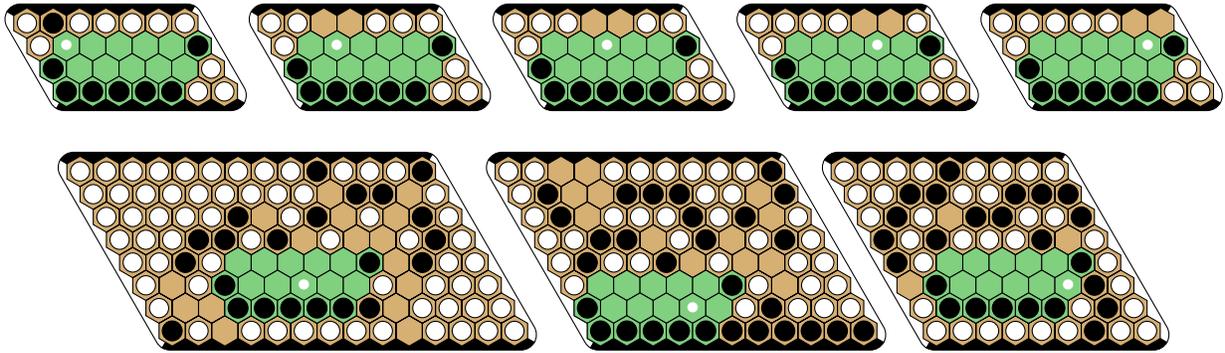

  \def\scale{0.5}
  \setlength{\increment}{0.175cm}
  \newcommand{\offset}{\hspace{2em}}
  \[
  \begin{hexboard}[baseline={($(current bounding box.south)-(0,1ex)$)},scale=\scale]
    \rotation{-30}
    \board(7,4)
    \white(1,1)\black(2,1)\white(3,1)\white(4,1)\white(5,1)\white(6,1)\white(7,1)
    \white(1,2)\carry(2,2)\carry(3,2)\carry(4,2)\carry(5,2)\carry(6,2)\black(7,2)
    \black(1,3)\carry(2,3)\carry(3,3)\carry(4,3)\carry(5,3)\carry(6,3)\white(7,3)
    \black(1,4)\black(2,4)\black(3,4)\black(4,4)\black(5,4)\white(6,4)\white(7,4)
    \carry(1,3)
    \carry(1,4)
    \carry(2,4)
    \carry(3,4)
    \carry(4,4)
    \carry(5,4)
    \carry(7,2)
    \cell(2,2)\winmove
  \end{hexboard}
  \offset\hspace{-4\increment}
  \begin{hexboard}[baseline={($(current bounding box.south)-(0,1ex)$)},scale=\scale]
    \rotation{-30}
    \board(7,4)
    \white(1,1)\white(2,1)\nofil(3,1)\nofil(4,1)\white(5,1)\white(6,1)\white(7,1)
    \white(1,2)\carry(2,2)\carry(3,2)\carry(4,2)\carry(5,2)\carry(6,2)\black(7,2)
    \black(1,3)\carry(2,3)\carry(3,3)\carry(4,3)\carry(5,3)\carry(6,3)\white(7,3)
    \black(1,4)\black(2,4)\black(3,4)\black(4,4)\black(5,4)\white(6,4)\white(7,4)
    \carry(1,3)
    \carry(1,4)
    \carry(2,4)
    \carry(3,4)
    \carry(4,4)
    \carry(5,4)
    \carry(7,2)
    \cell(3,2)\winmove
  \end{hexboard}
  \offset\hspace{-4\increment}
  \begin{hexboard}[baseline={($(current bounding box.south)-(0,1ex)$)},scale=\scale]
    \rotation{-30}
    \board(7,4)
    \white(1,1)\white(2,1)\white(3,1)\nofil(4,1)\nofil(5,1)\white(6,1)\white(7,1)
    \white(1,2)\carry(2,2)\carry(3,2)\carry(4,2)\carry(5,2)\carry(6,2)\black(7,2)
    \black(1,3)\carry(2,3)\carry(3,3)\carry(4,3)\carry(5,3)\carry(6,3)\white(7,3)
    \black(1,4)\black(2,4)\black(3,4)\black(4,4)\black(5,4)\white(6,4)\white(7,4)
    \carry(1,3)
    \carry(1,4)
    \carry(2,4)
    \carry(3,4)
    \carry(4,4)
    \carry(5,4)
    \carry(7,2)
    \cell(4,2)\winmove
  \end{hexboard}
  \offset\hspace{-4\increment}
  \begin{hexboard}[baseline={($(current bounding box.south)-(0,1ex)$)},scale=\scale]
    \rotation{-30}
    \board(7,4)
    \white(1,1)\white(2,1)\white(3,1)\white(4,1)\nofil(5,1)\nofil(6,1)\white(7,1)
    \white(1,2)\carry(2,2)\carry(3,2)\carry(4,2)\carry(5,2)\carry(6,2)\black(7,2)
    \black(1,3)\carry(2,3)\carry(3,3)\carry(4,3)\carry(5,3)\carry(6,3)\white(7,3)
    \black(1,4)\black(2,4)\black(3,4)\black(4,4)\black(5,4)\white(6,4)\white(7,4)
    \carry(1,3)
    \carry(1,4)
    \carry(2,4)
    \carry(3,4)
    \carry(4,4)
    \carry(5,4)
    \carry(7,2)
    \cell(5,2)\winmove
  \end{hexboard}
  \offset\hspace{-4\increment}
  \begin{hexboard}[baseline={($(current bounding box.south)-(0,1ex)$)},scale=\scale]
    \rotation{-30}
    \board(7,4)
    \white(1,1)\white(2,1)\white(3,1)\white(4,1)\white(5,1)\nofil(6,1)\nofil(7,1)
    \white(1,2)\carry(2,2)\carry(3,2)\carry(4,2)\carry(5,2)\carry(6,2)\black(7,2)
    \black(1,3)\carry(2,3)\carry(3,3)\carry(4,3)\carry(5,3)\carry(6,3)\white(7,3)
    \black(1,4)\black(2,4)\black(3,4)\black(4,4)\black(5,4)\white(6,4)\white(7,4)
    \carry(1,3)
    \carry(1,4)
    \carry(2,4)
    \carry(3,4)
    \carry(4,4)
    \carry(5,4)
    \carry(7,2)
    \cell(6,2)\winmove
  \end{hexboard}
  \]
  \[
  \begin{hexboard}[baseline={($(current bounding box.south)-(0,1ex)$)},scale=\scale]
    \rotation{-30}
    \board(14,8)
    \white(1,1)\white(2,1)\white(3,1)\white(4,1)\white(5,1)\white(6,1)\white(7,1)\white(8,1)\white(9,1)\black(10,1)\white(11,1)\white(12,1)\white(13,1)\black(14,1)
    \white(1,2)\white(2,2)\white(3,2)\white(4,2)\white(5,2)\white(6,2)\white(7,2)\white(8,2)\white(9,2)\nofil(10,2)\black(11,2)\black(12,2)\nofil(13,2)\white(14,2)
    \white(1,3)\white(2,3)\white(3,3)\white(4,3)\white(5,3)\black(6,3)\nofil(7,3)\white(8,3)\black(9,3)\nofil(10,3)\white(11,3)\nofil(12,3)\black(13,3)\white(14,3)
    \white(1,4)\white(2,4)\white(3,4)\black(4,4)\black(5,4)\white(6,4)\black(7,4)\nofil(8,4)\white(9,4)\nofil(10,4)\nofil(11,4)\white(12,4)\black(13,4)\white(14,4)
    \white(1,5)\white(2,5)\black(3,5)\white(4,5)\carry(5,5)\carry(6,5)\carry(7,5)\carry(8,5)\carry(9,5)\black(10,5)\nofil(11,5)\black(12,5)\white(13,5)\white(14,5)
    \white(1,6)\nofil(2,6)\white(3,6)\black(4,6)\carry(5,6)\carry(6,6)\carry(7,6)\carry(8,6)\carry(9,6)\white(10,6)\nofil(11,6)\white(12,6)\white(13,6)\white(14,6)
    \white(1,7)\nofil(2,7)\nofil(3,7)\black(4,7)\black(5,7)\black(6,7)\black(7,7)\black(8,7)\black(9,7)\nofil(10,7)\white(11,7)\white(12,7)\white(13,7)\white(14,7)
    \black(1,8)\white(2,8)\nofil(3,8)\white(4,8)\white(5,8)\white(6,8)\white(7,8)\white(8,8)\white(9,8)\nofil(10,8)\white(11,8)\white(12,8)\white(13,8)\white(14,8)
    \carry(4,6)
    \carry(4,7)
    \carry(5,7)
    \carry(6,7)
    \carry(7,7)
    \carry(8,7)
    \carry(10,5)
    \cell(7,6)\winmove
  \end{hexboard}
  \offset\hspace{-8\increment}
  \begin{hexboard}[baseline={($(current bounding box.south)-(0,1ex)$)},scale=\scale]
    \rotation{-30}
    \board(11,8)
    \white(1,1)\white(2,1)\nofil(3,1)\nofil(4,1)\white(5,1)\white(6,1)\white(7,1)\white(8,1)\white(9,1)\white(10,1)\black(11,1)
    \white(1,2)\black(2,2)\nofil(3,2)\nofil(4,2)\black(5,2)\black(6,2)\black(7,2)\white(8,2)\white(9,2)\nofil(10,2)\black(11,2)
    \white(1,3)\black(2,3)\nofil(3,3)\white(4,3)\white(5,3)\white(6,3)\black(7,3)\white(8,3)\black(9,3)\nofil(10,3)\white(11,3)
    \white(1,4)\white(2,4)\black(3,4)\black(4,4)\nofil(5,4)\white(6,4)\black(7,4)\nofil(8,4)\white(9,4)\black(10,4)\white(11,4)
    \white(1,5)\black(2,5)\white(3,5)\white(4,5)\black(5,5)\nofil(6,5)\white(7,5)\nofil(8,5)\nofil(9,5)\white(10,5)\white(11,5)
    \white(1,6)\carry(2,6)\carry(3,6)\carry(4,6)\carry(5,6)\carry(6,6)\black(7,6)\white(8,6)\nofil(9,6)\white(10,6)\white(11,6)
    \black(1,7)\carry(2,7)\carry(3,7)\carry(4,7)\carry(5,7)\carry(6,7)\white(7,7)\black(8,7)\white(9,7)\white(10,7)\white(11,7)
    \black(1,8)\black(2,8)\black(3,8)\black(4,8)\black(5,8)\black(6,8)\black(7,8)\black(8,8)\black(9,8)\black(10,8)\black(11,8)
    \carry(1,7)
    \carry(1,8)
    \carry(2,8)
    \carry(3,8)
    \carry(4,8)
    \carry(5,8)
    \carry(7,6)
    \cell(5,7)\winmove
  \end{hexboard}
  \offset\hspace{-9\increment}
  \begin{hexboard}[baseline={($(current bounding box.south)-(0,1ex)$)},scale=\scale]
    \rotation{-30}
    \board(9,8)
    \white(1,1)\white(2,1)\white(3,1)\white(4,1)\black(5,1)\white(6,1)\white(7,1)\white(8,1)\white(9,1)
    \white(1,2)\white(2,2)\black(3,2)\black(4,2)\white(5,2)\white(6,2)\black(7,2)\black(8,2)\black(9,2)
    \white(1,3)\black(2,3)\white(3,3)\nofil(4,3)\black(5,3)\black(6,3)\white(7,3)\white(8,3)\black(9,3)
    \black(1,4)\white(2,4)\black(3,4)\white(4,4)\white(5,4)\white(6,4)\black(7,4)\nofil(8,4)\white(9,4)
    \black(1,5)\white(2,5)\carry(3,5)\carry(4,5)\carry(5,5)\carry(6,5)\carry(7,5)\black(8,5)\white(9,5)
    \nofil(1,6)\black(2,6)\carry(3,6)\carry(4,6)\carry(5,6)\carry(6,6)\carry(7,6)\black(8,6)\white(9,6)
    \white(1,7)\black(2,7)\black(3,7)\black(4,7)\black(5,7)\black(6,7)\white(7,7)\black(8,7)\white(9,7)
    \white(1,8)\white(2,8)\white(3,8)\white(4,8)\white(5,8)\white(6,8)\black(7,8)\white(8,8)\white(9,8)
    \carry(2,6)
    \carry(2,7)
    \carry(3,7)
    \carry(4,7)
    \carry(5,7)
    \carry(6,7)
    \carry(8,5)
    \cell(7,6)\winmove
  \end{hexboard}
  \]
  \caption{Witnesses for the non-inferiority of all probes except
    probes 2 and 4 of the long crescent.}
  \label{fig:witnesses-long-crescent}
\end{figure}

\begin{figure}
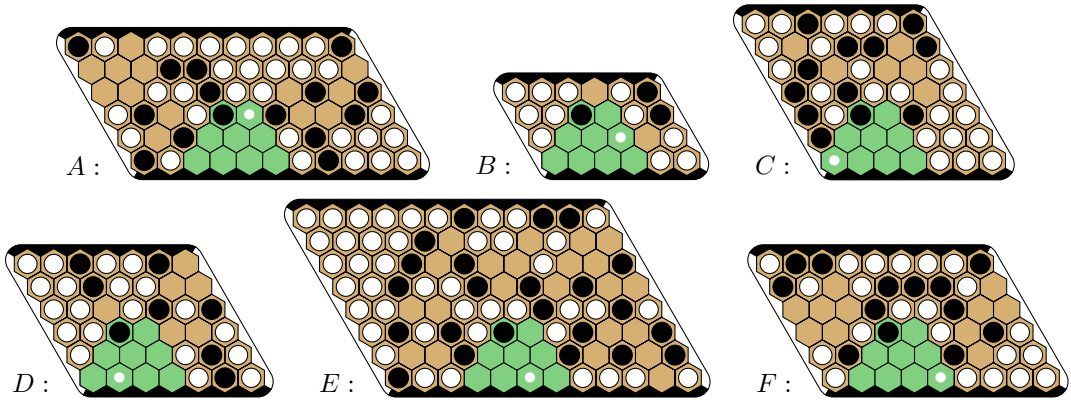

  \def\scale{0.5}
  \setlength{\increment}{0.175cm}
  \newcommand{\offset}{\hspace{1.7em}}
  \newcommand{\vshift}{0.5ex}
  \newcommand{\mylabel}[1]{#1\quad}
  \[
  \mylabel{A:}\hspace{-6\increment}
  \begin{hexboard}[baseline={($(current bounding box.south)+(0,\vshift)$)},scale=\scale]
    \rotation{-30}
    \board(11,6)
    \black(1,1)\white(2,1)\nofil(3,1)\white(4,1)\white(5,1)\white(6,1)\white(7,1)\white(8,1)\white(9,1)\white(10,1)\black(11,1)
    \nofil(1,2)\nofil(2,2)\nofil(3,2)\black(4,2)\black(5,2)\white(6,2)\white(7,2)\white(8,2)\white(9,2)\white(10,2)\nofil(11,2)
    \nofil(1,3)\nofil(2,3)\white(3,3)\white(4,3)\black(5,3)\white(6,3)\white(7,3)\nofil(8,3)\black(9,3)\nofil(10,3)\black(11,3)
    \white(1,4)\black(2,4)\nofil(3,4)\white(4,4)\black(5,4)\carry(6,4)\black(7,4)\nofil(8,4)\nofil(9,4)\black(10,4)\white(11,4)
    \white(1,5)\nofil(2,5)\black(3,5)\carry(4,5)\carry(5,5)\carry(6,5)\white(7,5)\black(8,5)\white(9,5)\white(10,5)\white(11,5)
    \black(1,6)\white(2,6)\carry(3,6)\carry(4,6)\carry(5,6)\carry(6,6)\white(7,6)\black(8,6)\white(9,6)\white(10,6)\white(11,6)
    \carry(5,4)
    \cell(6,4)\winmove
  \end{hexboard}
  \offset\mylabel{B:}\hspace{-4\increment}
  \begin{hexboard}[baseline={($(current bounding box.south)+(0,\vshift)$)},scale=\scale]
    \rotation{-30}
    \board(6,4)
    \white(1,1)\white(2,1)\white(3,1)\nofil(4,1)\white(5,1)\black(6,1)
    \white(1,2)\white(2,2)\black(3,2)\carry(4,2)\white(5,2)\black(6,2)
    \white(1,3)\carry(2,3)\carry(3,3)\carry(4,3)\nofil(5,3)\white(6,3)
    \carry(1,4)\carry(2,4)\carry(3,4)\carry(4,4)\white(5,4)\white(6,4)
    \carry(3,2)
    \cell(4,3)\winmove
  \end{hexboard}
  \offset\mylabel{C:}\hspace{-7\increment}
  \begin{hexboard}[baseline={($(current bounding box.south)+(0,\vshift)$)},scale=\scale]
    \rotation{-30}
    \board(7,7)
    \white(1,1)\white(2,1)\black(3,1)\white(4,1)\white(5,1)\white(6,1)\black(7,1)
    \white(1,2)\nofil(2,2)\white(3,2)\black(4,2)\black(5,2)\nofil(6,2)\black(7,2)
    \white(1,3)\black(2,3)\nofil(3,3)\white(4,3)\nofil(5,3)\nofil(6,3)\white(7,3)
    \black(1,4)\white(2,4)\black(3,4)\white(4,4)\white(5,4)\nofil(6,4)\white(7,4)
    \black(1,5)\white(2,5)\black(3,5)\carry(4,5)\black(5,5)\white(6,5)\white(7,5)
    \black(1,6)\carry(2,6)\carry(3,6)\carry(4,6)\white(5,6)\white(6,6)\white(7,6)
    \carry(1,7)\carry(2,7)\carry(3,7)\carry(4,7)\white(5,7)\white(6,7)\white(7,7)
    \carry(3,5)
    \cell(1,7)\winmove
  \end{hexboard}
  \]\[
  \mylabel{D:}\hspace{-6\increment}
  \begin{hexboard}[baseline={($(current bounding box.south)+(0,\vshift)$)},scale=\scale]
    \rotation{-30}
    \board(7,6)
    \white(1,1)\white(2,1)\black(3,1)\white(4,1)\white(5,1)\black(6,1)\nofil(7,1)
    \white(1,2)\white(2,2)\black(3,2)\white(4,2)\white(5,2)\nofil(6,2)\nofil(7,2)
    \white(1,3)\white(2,3)\nofil(3,3)\white(4,3)\black(5,3)\white(6,3)\black(7,3)
    \white(1,4)\white(2,4)\black(3,4)\carry(4,4)\nofil(5,4)\nofil(6,4)\white(7,4)
    \white(1,5)\carry(2,5)\carry(3,5)\carry(4,5)\white(5,5)\black(6,5)\white(7,5)
    \carry(1,6)\carry(2,6)\carry(3,6)\carry(4,6)\white(5,6)\black(6,6)\white(7,6)
    \carry(3,4)
    \cell(2,6)\winmove
  \end{hexboard}
  \offset\mylabel{E:}\hspace{-8\increment}
  \begin{hexboard}[baseline={($(current bounding box.south)+(0,\vshift)$)},scale=\scale]
    \rotation{-30}
    \board(12,8)
    \white(1,1)\white(2,1)\white(3,1)\white(4,1)\white(5,1)\white(6,1)\black(7,1)\white(8,1)\white(9,1)\black(10,1)\black(11,1)\white(12,1)
    \white(1,2)\white(2,2)\white(3,2)\white(4,2)\black(5,2)\nofil(6,2)\white(7,2)\white(8,2)\nofil(9,2)\white(10,2)\nofil(11,2)\nofil(12,2)
    \white(1,3)\white(2,3)\white(3,3)\black(4,3)\nofil(5,3)\black(6,3)\nofil(7,3)\nofil(8,3)\white(9,3)\nofil(10,3)\nofil(11,3)\black(12,3)
    \white(1,4)\white(2,4)\black(3,4)\white(4,4)\white(5,4)\black(6,4)\nofil(7,4)\black(8,4)\nofil(9,4)\black(10,4)\nofil(11,4)\nofil(12,4)
    \white(1,5)\white(2,5)\nofil(3,5)\nofil(4,5)\white(5,5)\white(6,5)\white(7,5)\black(8,5)\white(9,5)\white(10,5)\black(11,5)\white(12,5)
    \white(1,6)\black(2,6)\white(3,6)\black(4,6)\white(5,6)\black(6,6)\carry(7,6)\white(8,6)\black(9,6)\nofil(10,6)\nofil(11,6)\black(12,6)
    \white(1,7)\nofil(2,7)\nofil(3,7)\black(4,7)\carry(5,7)\carry(6,7)\carry(7,7)\black(8,7)\nofil(9,7)\black(10,7)\nofil(11,7)\black(12,7)
    \black(1,8)\white(2,8)\white(3,8)\carry(4,8)\carry(5,8)\carry(6,8)\carry(7,8)\white(8,8)\white(9,8)\white(10,8)\nofil(11,8)\white(12,8)
    \carry(6,6)
    \cell(6,8)\winmove
  \end{hexboard}
  \offset\mylabel{F:}\hspace{-6\increment}
  \begin{hexboard}[baseline={($(current bounding box.south)+(0,\vshift)$)},scale=\scale]
    \rotation{-30}
    \board(9,6)
    \white(1,1)\black(2,1)\black(3,1)\white(4,1)\white(5,1)\white(6,1)\white(7,1)\white(8,1)\black(9,1)
    \black(1,2)\white(2,2)\nofil(3,2)\white(4,2)\black(5,2)\black(6,2)\black(7,2)\white(8,2)\nofil(9,2)
    \nofil(1,3)\nofil(2,3)\nofil(3,3)\black(4,3)\white(5,3)\white(6,3)\black(7,3)\nofil(8,3)\nofil(9,3)
    \nofil(1,4)\nofil(2,4)\white(3,4)\black(4,4)\carry(5,4)\white(6,4)\nofil(7,4)\black(8,4)\white(9,4)
    \white(1,5)\black(2,5)\carry(3,5)\carry(4,5)\carry(5,5)\black(6,5)\nofil(7,5)\white(8,5)\white(9,5)
    \white(1,6)\carry(2,6)\carry(3,6)\carry(4,6)\carry(5,6)\white(6,6)\white(7,6)\white(8,6)\white(9,6)
    \carry(4,4)
    \cell(5,6)\winmove
  \end{hexboard}
  \]
  \caption{Witnesses for the non-inferiority of some probes in the
    ziggurat in the presence of certain white boundary stones.}
  \label{fig:witnesses-ziggurat-extra}
\end{figure}

\section*{Conclusions and future work}

In this paper, we studied 3-terminal positions in Hex using
combinatorial game theory. It turns out that there are many 3-terminal
positions with interesting and intricate behaviors. We started by
investigating superswitches, an infinite class of positions with
remarkable properties. The basic superswitch is a relatively complex
Hex position with 12 empty cells, but it has the strikingly simple
combinatorial value $\g{a,b|a}$.  This means that under optimal play,
no player can achieve any outcome but $a$ or $b$, and the left player
gets to ``set'' the switch. This allows the basic superswitch to be
used as a building block for making other interesting gadgets, such as
higher superswitches and tripleswitches. Among other things, we used
this to show that there are infinitely many Hex-realizable and
Hex-distinct values for 3-terminal positions. From a theoretical point
of view, superswitches are interesting because they form a strictly
increasing, cofinal sequence among all games in which Left cannot
achieve outcome $\top$ or $c$.  Superswitches (and the related
simpleswitches) also have practical applications, such as the
verification of connects-both templates.

The most useful tool in our toolbox is a large database of
Hex-realizable 3-terminal positions, which we generated by starting
from a few simple positions and repeatedly closing them under
symmetries, duality, and pinwheel composition. This is how we were
able to find Hex implementations of the superswitches and
simpleswitches and many other Hex devices. The database is useful
because many difficult questions in Hex boil down to constructing Hex
positions with carefully crafted combinatorial values. A compelling
example of this is a characterization of pivoting templates, a kind of
Hex template defined by an intricate balance of threats and
answers. We were able to give a combinatorial game theory
characterization of two kinds of pivoting templates. Moreover, we
found particular 3-terminal contexts that act as perfect tests for
pivoting templates: any candidate pivoting template can be
mechanically verified by composing it with the corresponding testing
context.

In our final application of 3-terminal positions, we settled an open
question by Henderson and Hayward on the inferiority of certain probes
in the ziggurat template. Inferiority can often be proved using local
methods, such as domination and strong reversibility of
moves. However, non-inferiority is much harder to show, as it requires
the construction of a witnessing position in which the probe in
question is the unique winning move. Such witnessing positions are
often extremely rare and hard to find. Prior to the present paper,
there was no systematic method for constructing them. Our database of
3-terminal positions, along with a process of stepwise refinement (and
considerable computational efforts) allowed us to determine the status
of all probes in the ziggurat and many other Hex templates.

The paper probably raises as many questions as it answers. Our
database provides a practical way of searching for concrete Hex
realizations of many combinatorial 3-terminal values. But even the
largest such database is necessarily finite. The question remains open
whether \emph{every} passable combinatorial 3-terminal value is
Hex-realizable (and failing that, how to characterize the ones that
are). Progress on this question might allow us to eventually replace
our database by a procedure that can compute Hex realizations of
arbitrary passable game values.

Another area for future improvement is the stepwise refinement method
of Section~\ref{sec:witnesses}. The theory of abstract probes already
proved to be useful in answering open questions, but the process of
stepwise refinement is still very labor intensive. The main problem is
that, given a probe $P$ and a value $A$, there is currently no a
priori way to predict whether the refined probe $P+A$ will be
viable. We also lack a notion of which refinements are ``more'' viable
than others. As a result, there is very little to guide our selection
of candidate positions for subdivisions, and we are left to compute a
large number of candidates in the hope that enough of them will be
viable. It would be nice to have future theorems that can make this
process more goal directed.

\section*{Acknowledgements}

We thank Ryan Hayward for useful comments on an earlier draft of this
paper, and Andreas Tsevas for pointing out a typo in
Figure~\ref{fig:hex-auto}.

\bibliographystyle{abbrv}
\bibliography{hex-3term}

\end{document}